\documentclass{amsart}
\usepackage[letterpaper, margin=1.4in]{geometry}
\usepackage{amssymb, latexsym, MnSymbol}
\usepackage{mathtools}
\usepackage{color, xcolor}
\usepackage{amscd, graphicx, psfrag, tikz}
\usepackage[mathscr]{euscript}
\usepackage[all]{xy}

\usepackage{enumitem}
\usepackage{stmaryrd}
\usepackage{url}
\usepackage{hyperref}
\hypersetup{
    colorlinks=true,
    linkcolor=black,
    citecolor=black,
    urlcolor=black
}
\usepackage{verbatim}

\newcommand{\Si}{ {\Sigma} }

\newcommand{\bC}{ {\mathbb{C}} }

\newcommand{\bE}{\mathbb{E}}
\newcommand{\bF}{\mathbb{F}}

\newcommand{\bK}{\mathbb{K}}
\newcommand{\bL}{\mathbb{L}}

\newcommand{\bP}{\mathbb{P}}
\newcommand{\bQ}{\mathbb{Q}}
\newcommand{\bR}{\mathbb{R}}
\newcommand{\bS}{\mathbb{S}}

\newcommand{\bZ}{\mathbb{Z}}

\newcommand{\cB}{\mathcal{B}}
\newcommand{\cC}{\mathcal{C}}
\newcommand{\cD}{\mathcal{D}}
\newcommand{\cE}{\mathcal{E}}
\newcommand{\cF}{\mathcal{F}}

\newcommand{\cI}{\mathcal{I}}
\newcommand{\cL}{\mathcal{L}}
\newcommand{\cM}{\mathcal{M}}
\newcommand{\cO}{\mathcal{O}}

\newcommand{\cQ}{\mathcal{Q}}
\newcommand{\cR}{\mathcal{R}}
\newcommand{\cS}{\mathcal{S}}

\newcommand{\cX}{\mathcal{X}}
\newcommand{\cU}{\mathcal{U}}
\newcommand{\cV}{\mathcal{V}}

\newcommand{\cZ}{\mathcal{Z}}

\newcommand{\age}{\mathrm{age}}
\newcommand{\Aut}{\mathrm{Aut}}
\newcommand{\CR}{ {\mathrm{CR}} }

\newcommand{\Hom}{\mathrm{Hom}}

\newcommand{\End}{\mathrm{End}}

\newcommand{\Spec}{\mathrm{Spec}}

\newcommand{\eff}{ {\mathrm{eff}} }

\newcommand{\ev}{\mathrm{ev}}
\newcommand{\val}{ {\mathrm{val}} }
\newcommand{\vir}{ {\mathrm{vir}} }

\newcommand{\inv}{\mathrm{inv}}

\newcommand{\pt}{\mathrm{pt}}

\newcommand{\Nef}{{\mathrm{Nef}}}
\newcommand{\NE}{{\mathrm{NE}}}

\newcommand{\rank}{\mathrm{rank}}

\newcommand{\Tot}{\mathrm{Tot}}

\renewcommand{\Box}{\mathrm{Box}}
\newcommand{\GL}{\mathrm{GL}}

\newcommand{\bb}{\text{$\mathbf{b}$}}

\newcommand{\one}{\mathbf{1}}

\newcommand{\btau}{\boldsymbol{\tau}}

\newcommand{\bSi}{\mathbf{\Si}}
\newcommand{\bXi}{\mathbf{\Xi}}
\newcommand{\bh}{\text{$\mathbf{h}$}}
\newcommand{\bw}{\text{$\mathbf{w}$}}

\newcommand{\fa}{\mathfrak{a}}
\newcommand{\fg}{\mathfrak{g}}
\newcommand{\fl}{\mathfrak{l}}
\newcommand{\fm}{\mathfrak{m}}

\newcommand{\fn}{\mathfrak{n}}
\newcommand{\fo}{\mathfrak{o}}
\newcommand{\fp}{\mathfrak{p}}
\newcommand{\fr}{\mathfrak{r}}
\newcommand{\fs}{\mathfrak{s}}

\newcommand{\fu}{\mathfrak{u}}

\newcommand{\fb}{\mathfrak{b}}

\newcommand{\fx}{\mathfrak{x}}

\newcommand{\su}{\mathsf{u}}

\newcommand{\sw}{\mathsf{w}}
\newcommand{\sX}{\mathsf{X}} 

\newcommand{\hG}{\hat{G}}

\newcommand{\hcC}{\hat{\cC}}

\newcommand{\tSi}{\widetilde{\Sigma}}
\newcommand{\trho}{\widetilde{\rho}}

\newcommand{\tbL}{\widetilde{\bL}}

\newcommand{\tA}{\widetilde{A}}

\newcommand{\tC}{\widetilde{C}}
\newcommand{\tD}{\widetilde{D}}

\newcommand{\tG}{\widetilde{G}}
\newcommand{\tH}{\widetilde{H}}

\newcommand{\tM}{\widetilde{M}}
\newcommand{\tN}{\widetilde{N}}
\newcommand{\tP}{\widetilde{P}}
\newcommand{\tQ}{\widetilde{Q}}

\newcommand{\tS}{\widetilde{S}}
\newcommand{\tT}{\widetilde{T}}

\newcommand{\tX}{\widetilde{X}}

\newcommand{\tb}{\widetilde{b}}

\renewcommand{\th}{\widetilde{h}}

\newcommand{\tk}{\widetilde{k}}
\newcommand{\tl}{\widetilde{l}}
\newcommand{\tm}{\widetilde{m}}

\newcommand{\tq}{\widetilde{q}}
\newcommand{\tu}{\widetilde{u}}
\newcommand{\tv}{\widetilde{v}}

\newcommand{\tmu}{\widetilde{\mu}}

\newcommand{\tgamma}{\widetilde{\gamma}}
\newcommand{\tsi}{\widetilde{\sigma}}
\newcommand{\tbeta}{\widetilde{\beta}}
\newcommand{\tpsi}{\widetilde{\psi}}

\newcommand{\ttau}{\widetilde{\tau}}
\newcommand{\talpha}{\widetilde{\alpha}}
\newcommand{\tbtau}{\widetilde{\btau}}
\newcommand{\tkappa}{\widetilde{\kappa}}
\newcommand{\tOmega}{\widetilde{\Omega}}
\newcommand{\tlambda}{\widetilde{\lambda}}

\newcommand{\tbh}{\widetilde{\bh}}

\newcommand{\tbw}{\widetilde{\bw}}

\newcommand{\tNef}{\widetilde{\Nef}}
\newcommand{\tNE}{\widetilde{\NE}}

\newcommand{\tcC}{\widetilde{\cC}}
\newcommand{\tcD}{\widetilde{\cD}}
\newcommand{\tcX}{\widetilde{\cX}}
\newcommand{\tbSi}{\widetilde{\bSi}}

\newcommand{\vGa}{\vec{\Gamma}}
\newcommand{\vd}{\vec{d}}
\newcommand{\vf}{\vec{f}}

\newcommand{\vj}{\vec{j}}
\newcommand{\vk}{\vec{k}}
\newcommand{\vs}{\vec{s}}

\newcommand{\vzero}{\vec{0}}

\newcommand{\Mbar}{\overline{\cM}}

\newcommand{\lra}{\longrightarrow}

\newcommand{\inner}[1]{\langle  #1 \rangle}
\newcommand{\floor}[1]{\lfloor  #1 \rfloor}
\newcommand{\ceil}[1]{\lceil  #1 \rceil}

\newtheorem{dummy}{dummy}[section]
\newtheorem{lemma}[dummy]{Lemma}
\newtheorem{theorem}[dummy]{Theorem}

\newtheorem{proposition}[dummy]{Proposition}
\newtheorem{remark}[dummy]{Remark}
\newtheorem{definition}[dummy]{Definition}

\newtheorem{notation}[dummy]{Notation}
\newtheorem{convention}[dummy]{Convention}
\newtheorem{assumption}[dummy]{Assumption}
\newtheorem{observation}[dummy]{Observation}

\begin{document}
\title{Orbifold Open/Closed Correspondence and Mirror Symmetry}

\author{Chiu-Chu Melissa Liu}
\address{Chiu-Chu Melissa Liu, Department of Mathematics, Columbia University, 2990 Broadway, New York, NY 10027}
\email{ccliu@math.columbia.edu}

\author{Song Yu}
\address{Song Yu, Yau Mathematical Sciences Center, Tsinghua University, Haidian District, Beijing 100084, China}
\email{song-yu@tsinghua.edu.cn}

\begin{abstract}
We continue the mathematical development of the open/closed correspondence proposed by Mayr and Lerche-Mayr. Given an open geometry on a toric Calabi-Yau 3-orbifold $\mathcal{X}$ relative to a framed Aganagic-Vafa outer brane $(\mathcal{L},f)$, we construct a toric Calabi-Yau 4-orbifold $\widetilde{\mathcal{X}}$ and identify its genus-zero Gromov-Witten invariants with the disk invariants of $(\mathcal{X},\mathcal{L},f)$, generalizing prior work of the authors in the smooth case. We then upgrade the correspondence to the level of generating functions, and prove that the disk function of $(\mathcal{X},\mathcal{L},f)$ can be recovered from the equivariant $J$-function of $\widetilde{\mathcal{X}}$. We further establish a B-model correspondence that retrieves the B-model disk function of $(\mathcal{X},\mathcal{L},f)$ from the equivariant $I$-function of $\widetilde{\mathcal{X}}$, and show that the correspondences are compatible with mirror symmetry in both the open and closed sectors.
\end{abstract}
\maketitle

\setcounter{tocdepth}{1}
\tableofcontents

\section{Introduction}\label{sect:Intro}

This paper continues the mathematical study of the \emph{open/closed correspondence} initiated in \cite{LY21}. Proposed by Mayr \cite{Mayr01} and Lerche-Mayr \cite{LM01} as a class of string dualities, the correspondence predicts that the genus-zero topological amplitudes of an open string geometry on a Calabi-Yau 3-fold with a prescribed Lagrangian boundary condition should coincide with those of a closed string geometry on a dual Calabi-Yau 4-fold. In mathematical terms, the correspondence conjecturally relates the Gromov-Witten theories of the two geometries. Specifically, the disk invariants of the open 3-fold geometry should correspond to the genus-zero closed Gromov-Witten invariants of the 4-fold geometry, and this could be upgraded to the level of generating functions. Via mirror symmetry and especially that for open strings \cite{AV00}, the correspondence further carries over to the B-model side, predicting that the periods of the two mirror families for the two geometries match up and solve a common system of differential equations.

The preceding paper \cite{LY21} of the authors made a first step at the above proposal by establishing the numerical open/closed correspondence in the smooth case. Starting with the open geometry on a smooth toric Calabi-Yau 3-fold $\cX$ and a Lagrangian submanifold $\cL$ of Aganagic-Vafa type \cite{AV00}, under mild assumptions, we explicitly constructed a smooth toric Calabi-Yau 4-fold $\tcX$ as the dual closed geometry. We further showed that for corresponding curve classes, the disk invariants of $(\cX,\cL)$, which virtually count stable maps from genus-zero Riemann surfaces with a single boundary component to $(\cX,\cL)$, are identified with the genus-zero closed Gromov-Witten invariants of $\tcX$.

In the present paper, we develop the above result and take it to the next levels. First, we extend the numerical correspondence to the more general setting where the 3-fold $\cX$ can be a toric \emph{orbifold}, i.e. smooth toric Deligne-Mumford stack \cite{BCS05, FMN10} with trivial generic stablizer, and the Lagrangian boundary condition $\cL$ can be a suborbifold.  
We assume that $\cX$ has \emph{semi-projective} coarse moduli space and ensure that the 4-orbifold $\tcX$ that we construct satisfies the same property, which has desirable and helpful implications such as the convergence of the quantum product and the applicability of the mirror theorem of \cite{CCIT15}. Based on the numerical correspondence, we further establish a correspondence between the generating function of disk invariants of $(\cX, \cL)$ and the generating function of genus-zero Gromov-Witten invariants of $\tcX$. This also allows us to recover the former generating function from the equivariant \emph{$J$-function} of $\tcX$. In addition, on the B-model side, we prove the analogous result that the B-model disk potential of $(\cX, \cL)$ can be recovered from the equivariant \emph{$I$-function} of $\tcX$, which we further show is compatible with the A-model correspondences under mirror symmetry. Thereby, we substantiate the mathematical theory of the open/closed correspondence, fulfilling various original predictions in \cite{Mayr01, LM01} and getting better prepared for further developments and applications outlined in \cite{LY21}, including correspondences of B-model mirror families and periods as well as compatibilities with wall-crossings.

After the first version of this paper appeared, it has been applied to study the integrality properties of the open invariants of $(\cX, \cL)$ and closed invariants of $\tcX$ \cite{Yu23}, as well as the open Witten-Dijkgraaf-Verlinde-Verlinde equation for $(\cX, \cL)$ and the induced Frobenius structures \cite{YZ23}, in the case $\cX$ is smooth. The open/closed correspondence has also been established for the quintic 3-fold \cite{AL23} and the projective line \cite{Zong25}.

We now provide additional details and discussions for our results.

\subsection{Numerical correspondence for orbifolds}
Let $\cX$ be a toric Calabi-Yau 3-orbifold with semi-projective coarse moduli space $X$, $\cL \subset \cX$ be a Lagrangian brane of Aganagic-Vafa type. Let $f = \frac{\fb}{\fa} \in \bQ$ be an additional parameter called the \emph{framing} of the brane $\cL$, where $\fa \in \bZ_{>0}, \fb \in \bZ$ are coprime integers. The coarse moduli space $L$ of $\cL$ intersects a unique torus-fixed line $l$ in $X$, corresponding to a 2-cone in the fan. The corresponding torus-invariant substack $\fl \subset \cX$ has a cyclic generic stablizer group $\mu_{\fm}$ for some $\fm \in \bZ_{>0}$. The Lagrangian $\cL$ intersects $\fl$ along a stacky circle: $\cL \cap \fl \cong S^1 \times B\mu_{\fm}$.
We assume that $\cL$ is \emph{outer}, i.e. $l \cong \bC^1$.\footnote{We note that all our results can as well be established for the case where $\cL$ is inner, i.e. $l \cong \bP^1$.}

\emph{Open Gromov-Witten invariants} of $(\cX, \cL, f)$ \cite{CP14,FL13,FLT12,KL01} give virtual counts of twisted stable maps from bordered orbifold Riemann surfaces
$$
    u: (\cC, \partial \cC) \to (\cX, \cL).
$$
We focus on \emph{disk invariants}, which concern the case where the domain $\cC$ has arithmetic genus zero and one boundary component. Let
$$
    \beta' = \beta + d[B] \in H_2(X, L; \bZ),
$$
be an effective class, where $\beta \in H_2(X; \bZ)$ is effective, $B$ is the disk in $l$ bounded by the intersection $L \cap l$, and $d \in \bZ_{>0}$, and let $\lambda \in \mu_{\fm}$ be a monodromy profile. Then, the disk invariant
$$
    \inner{\gamma_1, \dots, \gamma_n}^{\cX, (\cL, f)}_{\beta', (d, \lambda)} \in \bQ
$$
virtually counts degree-$\beta'$ maps $u$ with boundary profile $u_*[\partial \cC] = (d, \lambda) \in H_1(\cL; \bZ) \cong \bZ \times \mu_{\fm}$. Here, $\gamma_1, \dots, \gamma_n \in H^2_{\CR}(\cX; \bQ)$ are any second Chen-Ruan orbifold cohomology \cite{CR04} classes of $\cX$ given as insertions to the interior of the domain. We refer to Section \ref{sect:GW} for formal definitions, noting for now that the disk invariants are defined by equivariant localization with respect to a $2$-dimensional subtorus $T'$ of the algebraic $3$-torus of $\cX$ specified by the Calabi-Yau condition, and further depend on a circle action on the pair $(\cX, \cL)$ specified by the framing $f$. Moreover, the definition adopts suitable $T'$-equivariant lifts of the insertions $\gamma_i$.

Our first main result is the explicit construction of a toric Calabi-Yau 4-orbifold $\tcX$ whose genus-zero Gromov-Witten invariants coincide with the disk invariants of $(\cX, \cL, f)$. This establishes the \emph{numerical open/closed correspondence} for orbifolds, generalizing \cite{LY21}. We note in advance that similar to above, the closed invariants are defined by equivariant localization with respect to a $3$-dimensional subtorus $\tT'$ of the algebraic $4$-torus of $\tcX$ specified by the Calabi-Yau condition, and further depend the action of a $1$-dimensional subtorus $T_f$ specified by $f$.

\begin{theorem}[See Theorem \ref{thm:NumericalOuter}]\label{thm:IntroNumerical}
With $\cX, \cL, f$ as above, there is a toric Calabi-Yau 4-orbifold $\tcX$ satisfying that:
\begin{itemize}
    \item The coarse moduli space $\tX$ of $\tcX$ is semi-projective.

    \item There is an inclusion $\iota: \cX \to \tcX$ which induces an inclusion $\iota_*: H_2(X,L; \bZ) \to H_2(\tX; \bZ)$.

    \item Given any effective class $\beta' = \beta + d[B] \in H_2(X,L;\bZ)$, monodromy $\lambda \in \mu_{\fm}$, and insertions $\gamma_1, \dots, \gamma_n \in H^2_{\CR}(\cX; \bQ)$, we have
    $$
        \inner{\gamma_1, \dots, \gamma_n}^{\cX, (\cL, f)}_{\beta', (d, \lambda)} = \inner{\tgamma_1, \dots, \tgamma_n, \tgamma_{\tk}}^{\tcX, T_f}_{\tbeta}
    $$
    where 
    \begin{itemize}
        \item[$\circ$] $\tbeta = \iota_*(\beta') \in H_2(\tX; \bZ)$ and $\inner{\tgamma_1, \dots, \tgamma_n,  \tgamma_{\tk}}^{\tcX, T_f}_{\tbeta}$ is a genus-zero, degree-$\tbeta$ closed Gromov-Witten invariant of $\tcX$;
        \item[$\circ$] each $\tgamma_i \in H^2_{\CR, \tT'}(\tcX; \bQ)$ is a suitable lift of $\gamma_i$ under $H^*_{\CR, \tT'}(\tcX;\bQ) \to H^*_{\CR}(\cX; \bQ)$;
        \item[$\circ$] $\tgamma_{\tk} \in H^4_{\CR, \tT'}(\tcX; \bQ)$ is a fixed class depending only on $d$ and $\lambda$.
    \end{itemize}
\end{itemize}
\end{theorem}

We note that the dependence of the additional class $\tgamma_{\tk}$ on $d$ and $\lambda$ is given by $\tk = \th(d,\lambda) \in \mu_{\fa\fm}$ where
$$
    \th: \bZ \times \mu_{\fm} \to \mu_{\fa\fm}
$$
is a surjective homomorphism that geometrically records (the inverse of) the twisted sector corresponding to the curve class $\tbeta$ in $\tcX$. In the case $f \in \bZ$, i.e. $\fa=1$, we have $\th(d,\lambda) = \lambda \in \mu_{\fa\fm}$ and $\tk$ depends on $\lambda$ only.

Informally, the $4$-orbifold $\tcX$ is constructed as follows: We first add a new irreducible toric divisor $\cD$ to $\cX$ whose position depends on the position of $\cL$, obtaining a log Calabi-Yau pair $(\cX \sqcup \cD, \cD)$. This is done so that $\cD$ contains a new torus-fixed point which, at the level of coarse moduli, compactifies the line $l \cong \bC^1$ into a $\bP^1$. Then, we consider the toric Calabi-Yau 4-orbifold $\Tot(\cO_{\cX \sqcup \cD}(-\cD))$; in the case $\cX = X$ is smooth and $f \in \bZ$, this is already the smooth $4$-fold constructed in \cite{LY21}. Finally, we take $\tcX$ to be the ``minimal'' toric partial compactification of $\Tot(\cO_{\cX \sqcup \cD}(-\cD))$ with a semi-projective coarse moduli space, in the sense that no additional toric divisors are added. We refer to Section \ref{sect:construction} for the formal construction and specifically Section \ref{sect:examples} for explicit examples.

As noted before, Theorem \ref{thm:IntroNumerical} extends \cite{LY21} in that it applies to the more general orbifold setting and produces a $4$-orbifold $\tcX$ with semi-projective coarse moduli space. In addition, it applies to disk invariants that admit an arbitrary set of degree-$2$ insertions from Chen-Ruan cohomology. Finally, it applies to an arbitrary framing $f \in \bQ$, while \cite{LY21} requires $f$ to be integral and generic.

We note that \cite{LY21} covered a family of examples for the open/closed correspondence obtained by Bousseau-Brini-van Garrel \cite{BBvG20} from constructions and enumerative theories based on \emph{Looijenga pairs}. Such examples are expected to extend to the orbifold setting \cite{BBvG20}, as exemplified in \cite{BBvG20b}. We expect our Theorem \ref{thm:IntroNumerical} to cover this more general family of orbifold examples.

Theorem \ref{thm:IntroNumerical} is proven by a direct and careful comparison of the localization computations of the open and closed Gromov-Witten invariants. One challenge is that, as we take additional steps in partially compactifying the $4$-orbifold into one with semi-projective coarse moduli space, we also introduce new connected components to the torus-fixed loci of the moduli spaces of twisted stable maps. We address this issue by arguing that these additional components have no contribution to the closed invariants.

We remark that the approach adopted in \cite{LY21} bridges the open and closed invariants by the \emph{relative} invariants \cite{Li01, Li02} of the log Calabi-Yau pair $(\cX \sqcup \cD, \cD)$. As an intermediate step, a version of the \emph{log-local principle} of van Garrel-Graber-Ruddat \cite{vGGR19} in the non-compact setting was obtained (see also Conjecture 1.1 of \cite{BBvG20}). Similar results may be obtained in the orbifold setting as well.

\subsection{Correspondence of generating functions}
Our next main result promotes the numerical open/closed correspondence to the level of \emph{generating functions}. For $\tk \in \mu_{\fa\fm}$, we define the following generating function of disk invariants of $(\cX, \cL, f)$, following Fang-Liu-Tseng \cite{FLT12}:
$$
    F^{\cX, (\cL, f)}_{\tk}(\btau_2, \sX):= \sum_{\beta \in E(X)} \sum_{\substack{(d, \lambda) \in \bZ_{>0} \times \mu_{\fm} \\ \tk = \th(d,\lambda)}} \sum_{l \in \bZ_{\ge 0}} \frac{\inner{\btau_2^l}^{\cX, (\cL, f)}_{\beta+d[B], (d, \lambda)}}{l!} \sX^d
$$
which takes value in $\bC$. Here, we take
$$
    \btau_2 = \sum_{a = 1}^{K} \tau_a u_a
$$
where $\{u_1, \dots, u_K\}$ is a suitable basis for $H^2_{\CR}(\cX; \bQ)$ and $\tau_1, \dots, \tau_K$ are complex variables. Moreover, $E(X)$ is the semigroup $\NE(X)\cap H_2(X;\bZ)$ and $\sX$ is an additional variable for the open sector.

On the other hand, we consider the following generating function of genus-zero closed Gromov-Witten invariants of $\cX$:
$$
    \llangle \tgamma_{\tk}\rrangle^{\tcX,T_f} (\tbtau_2) 
     := \sum_{\tbeta \in E(\tX)}  \sum_{l \in \bZ_{\ge 0}}  \frac{1}{l!}  \inner{ \tbtau_2^l ,\tgamma_{\tk}}^{\tcX, T_f}_{\tbeta},
$$
which also takes value in $\bC$. Here, we take
$$
    \tbtau_2 = \sum_{a=1}^{K+1} \ttau_a \tu_a
$$
where $\{\tu_1, \dots, \tu_{K+1}\}$ is a suitable basis for $H^2_{\CR}(\tcX; \bQ)$, with $\tu_a$ an appropriate lift of $u_a$ for $a = 1, \dots, K$, and $\ttau_1, \dots, \ttau_{K+1}$ are complex variables. The class $\tgamma_{\tk}$ is as in Theorem \ref{thm:IntroNumerical}. Moreover, $E(\tX)$ is the semigroup $\NE(\tX)\cap H_2(\tX;\bZ)$.

Using the numerical correspondence (Theorem \ref{thm:IntroNumerical}), we give the following correspondence between the two generating functions.

\begin{theorem}[See Theorem \ref{thm:GenCorr}]\label{thm:IntroGenCorr}
For any $\tk \in \mu_{\fa\fm}$,  the correspondence
$$
    F^{\cX, (\cL, f)}_{\tk}(\btau_2, \sX) = \llangle \tgamma_{\tk}\rrangle^{\tcX,T_f}(\tbtau_2)
$$
holds under the relation $\ttau_a = \tau_a$ for $a = 1, \dots, K$ and $\ttau_{K+1} = \log \sX$.
\end{theorem}

A notable feature of Theorem \ref{thm:IntroGenCorr} is that the extra closed variable $\ttau_{K+1}$ is identified with (the logarithm) of the open variable $\sX$, which was predicted by the original proposal \cite{Mayr01,LM01}.

The generating function $\llangle \tgamma_{\tk}\rrangle^{\tcX,T_f}$ is related in a standard way to another generating function of genus-zero closed Gromov-Witten invariants of $\tcX$ known as the \emph{$J$-function} \cite{Tseng10, CG07, Givental98}. As an application of Theorem \ref{thm:IntroGenCorr}, we show that the generating function $F^{\cX, (\cL, f)}_{\tk}$ of disk invariants of $(\cX, \cL, f)$ can be recovered from the $\tT'$-equivariant small $J$-function of $\tcX$, denoted $J_{\tcX}^{\tT'}(\tbtau_2, z)$. It is a power series in the inverse of an extra variable $z$ with coefficients valued in the $\tT'$-equivariant Chen-Ruan cohomology of $\tcX$.

\begin{theorem}[See Theorem \ref{thm:JPairing}]\label{thm:IntroJPairing}
For any $\tk \in \mu_{\fa\fm}$, 
$$
    F^{\cX, (\cL, f)}_{\tk} (\btau_2, \sX) = [z^{-2}]  \left(J_{\tcX}^{\tT'}(\tbtau_2, z), \tgamma_{\tk} \right)_{\tcX}^{\tT'} \bigg|_{T_f}
$$
under the relation $\ttau_a = \tau_a$ for $a = 1, \dots, K$ and $\ttau_{K+1} = \log \sX$.
\end{theorem}

Here, $\left(-,- \right)_{\tcX}^{\tT'}$ denotes the $\tT'$-equivariant orbifold Poincar\'e pairing of $\tcX$, the notation $\big|_{T_f}$ stands for taking weight restriction to the 1-torus $T_f$, and the notation $[z^{-2}]$ stands for taking the coefficient of $z^{-2}$ in the power series expansion.

\subsection{B-model correspondence and mirror symmetry}
Under the well-known \emph{mirror theorem} \cite{Givental98, CCK15, CCIT15}, the $J$-function of $\tcX$ as the A-model side is identified with the \emph{$I$-function} of $\tcX$ on the B-model side, which is an explicit hypergeometric function defined by the toric data. Let $I_{\tcX}^{\tT'}(\tq, z)$ denote the $\tT'$-equivariant $I$-function of $\tcX$, which depends on variables $\tq = (\tq_1, \dots, \tq_{K+1})$ that are related to $\ttau_1, \dots, \ttau_{K+1}$ under the \emph{closed mirror map} $\tbtau_2 = \tbtau_2(\tq)$.

Similarly, on the open side, the \emph{open mirror theorem} of \cite{FL13, FLT12} identifies the generating function $F^{\cX, (\cL, f)}_{\tk}(\btau_2, \sX)$ of disk invariants of $(\cX, \cL, f)$ with an explicit hypergeometric function $W^{\cX, (\cL, f)}_{\tk}(q, x)$ which we refer to as the \emph{B-model disk function}. It depends on closed-sector variables $q = (q_1, \dots, q_K)$ and an open-sector variable $x$ that are related to $\tau_1, \dots, \tau_K, \sX$ under the closed mirror map $\btau_2 = \btau_2(q)$ and the open mirror map $\sX = \sX(q,x)$. We note that \cite{FL13, FLT12} required that $f \in \bZ$, although their method can be directly generalized to the case $f \in \bQ$.

Now consider the web of relations in Figure \ref{fig:Web}, where the horizontal arrows represent the mirror theorems mentioned above. The vertical arrow on the left represents our Theorem \ref{thm:IntroJPairing} as a version of the open/closed correspondence on the A-model side. On the B-model side, we prove the parallel statement that the B-model disk function can be recovered from the $\tT'$-equivariant $I$-function of $\tcX$, filling in the vertical arrow on the right.

\begin{figure}[h]
$$
    \xymatrix{
        F^{\cX, (\cL, f)}_{\tk}(\btau_2, \sX) \ar@{<->}[r]^{\substack{\text{open mirror} \\ \text{theorem \cite{FL13, FLT12}} \\ \text{(Thm \ref{thm:OpenMirror})}}}  & W^{\cX, (\cL, f)}_{\tk}(q, x) & \text{$(\cX, \cL, f)$ (open)}\\
        J_{\tcX}^{\tT'}(\tbtau_2, z) \ar@{<->}[r]_{\substack{\text{closed mirror}\\ \text{theorem} \\ \text{\cite{Givental98, CCK15, CCIT15}} \\ \text{(Thm \ref{thm:Mirror})}}} 
         \ar[u]^{\text{Thm \ref{thm:IntroJPairing}/\ref{thm:JPairing}}} 
        & I_{\tcX}^{\tT'}(\tq, z).  \ar[u]_{\text{Thm \ref{thm:IntroIPairing}/\ref{thm:IPairing}}} & \text{$\tcX$ (closed)}\\
        \text{A-model} & \text{B-model} &
    }
$$
\caption{Open/closed correspondences and mirror symmetry.}
\label{fig:Web}
\end{figure}

\begin{theorem}[See Theorem \ref{thm:IPairing}]\label{thm:IntroIPairing}
For each $\tk \in \mu_{\fa\fm}$,  
$$
W^{\cX, (\cL, f)}_{\tk} (q, x) = [z^{-2}] \left(I_{\tcX}^{\tT'}(\tq, z), \tgamma_{\tk} \right)_{\tcX}^{\tT'} \bigg|_{T_f}
$$
under the relation $\tq_a = q_a$ for $a = 1, \dots, K$ and $\tq_{K+1} = x$.
\end{theorem}

Our proof is by an explicit comparison of the hypergeometric functions and is independent of mirror symmetry or the A-model correspondence. Again, the extra closed variable $\tq_{K+1}$ is identified with the open variable $x$. In fact, we show that the closed mirror map of $\tcX$ can be identified with the open-closed mirror map of $(\cX, \cL, f)$ as follows:
$$
    \ttau_a(\tq) = \tau_a(q), \quad a = 1, \dots, K,
$$
$$
    \ttau_{K+1}(\tq) = \log \sX(q, x). 
$$
See Proposition \ref{prop:MirrorMapCorr}. As a consequence, we confirm that the diagram in Figure \ref{fig:Web} is ``commutative'' in the sense that our A- and B-model open/closed correspondences are compatible with mirror symmetry. In particular, we obtain an alternative proof of the open mirror theorem \cite{FL13, FLT12} in the more general case where the framing $f \in \bQ$ is not necessarily an integer.

\subsection{Organization of the paper}
We start in Section \ref{sect:construction} with the construction of the closed geometry $\tcX$ from the open geometry $(\cX, \cL, f)$. In Section \ref{sect:GW}, we review the open and closed orbifold Gromov-Witten theories and localization computations of the invariants. We prove Theorem \ref{thm:IntroNumerical}, the numerical open/closed correspondence, in Section \ref{sect:Numerical}. Then, in Section \ref{sect:Generating}, we prove the correspondences Theorems \ref{thm:IntroGenCorr} and \ref{thm:IntroJPairing} at the level of generating functions. Finally, in Section \ref{sect:BModel}, we prove the B-model correspondence Theorem \ref{thm:IntroIPairing} and show that our correspondences are compatible with mirror symmetry.

\subsection{Acknowledgments}
The authors would like to thank Zhengyu Zong for helpful comments. The authors would also like to thank the hospitality and support of the Simons Center for Geometry and Physics  during the program on Integrability, Enumerative Geometry and Quantization, where part of the paper was completed. The program was partially supported by the NSF grant DMS-1564497.


\section{Toric geometry and constructions}\label{sect:construction}
In this section, we define the toric Calabi-Yau 3-orbifold $\cX$ and the framed Aganagic-Vafa brane $(\cL, f)$. Then, we give the construction of the dual toric Calabi-Yau 4-orbifold $\tcX$ and describe its relations to the open geometry of $(\cX, \cL, f)$. In general, we use notations with with tilde ($\tilde{\phantom{a}}$) while discussing $\tcX$ and its closed Gromov-Witten invariants. We work over $\bC$ throughout.

\subsection{Preliminaries on toric orbifolds}\label{sect:PrelimToric}
We start by reviewing the basics of \emph{toric orbifolds}, or \emph{smooth toric Deligne-Mumford stacks} with trivial generic stablizer, and introducing some notations. We refer to \cite{CLS11,Fulton93} for the general theory of toric varieties, and to \cite{BCS05, FMN10} for the general theory of smooth toric Deligne-Mumford stacks.

\subsubsection{Extended stacky fan}
Let $\cZ$ be an $r$-dimensional toric orbifold specified by an \emph{extended stacky fan} $\bXi = (\bZ^r, \Xi, \alpha)$ in the sense of Jiang \cite{Jiang08}, where $\Xi$ is a finite simplicial fan in $\bR^r = \bZ^r \otimes \bR$ and $\alpha:\bZ^R \to \bZ^r$ is a group homomorphism determined by a list of vectors $(b_1, \dots, b_R)$ in $\bZ^r$. The coarse moduli space $Z$ of $\cZ$ is the simplicial toric variety defined by the fan $\Xi$. Since $\cZ$ is an orbifold, the Deligne-Mumford torus acting on $\cZ$ is also the dense algebraic torus acting on $Z$, which is isomorphic to $(\bC^*)^r$.

For each $d = 0, \dots, r$, let $\Xi(d)$ denote the set of $d$-dimensional cones in $\Xi$. In particular, there exists $1 \leq R' \leq R$ such that
$$
   \Xi(1) = \{\bR_{\ge 0}b_1, \dots, \bR_{\ge 0}b_{R'}\}.
$$
If $\alpha':\bZ^{R'} \to \bZ^r$ is the group homomorphism determined by $(b_1, \dots, b_{R'})$, then the triple $(\bZ^r, \Xi, \alpha')$ is the \emph{stacky fan} of $\cZ$ in the sense of Borisov-Chen-Smith \cite{BCS05}.

For each $\sigma \in \Xi(d)$, let $\cV(\sigma) \subseteq \cZ$ denote the codimension-$d$ $(\bC^*)^r$-invariant closed substack of $\cZ$ corresponding to $\sigma$. Let $V(\sigma) \subseteq Z$ denote the codimension-$d$ $(\bC^*)^r$-orbit closure in $Z$ corresponding to $\sigma$, which is the coarse moduli space of $\cV(\sigma)$. Let
$$
    \iota_{\sigma}: \cV(\sigma) \to \cZ, \quad V(\sigma) \to Z,
$$
denote the inclusion maps.\footnote{By an abuse of notation, in this paper, the letter $\iota$ will be used to denote various natural inclusion maps of substacks/subvarieties, fixed loci of moduli spaces, or cones. The precise meaning and usage will be made clear in the context.}

\subsubsection{Stablizers}
Let $\sigma \in \Xi(d)$. We set index sets
$$
    I_\sigma' := \{ i \in \{1, \dots, R'\} : \rho_i \subseteq \sigma \}, \qquad I_\sigma := \{1, \dots, R\} \setminus I_\sigma'.
$$
Note that $|I_\sigma'|=d$. The generic stablizer group of the substack $\cV(\sigma)$, denoted $G_\sigma$, is a finite abelian group and can be identified as
$$
    G_\sigma \cong \left. \left( \bZ^r \cap \sum_{i \in I_\sigma'} \bR b_i \right) \middle/ \sum_{i \in I_\sigma'} \bZ b_i \right. .
$$
We define
$$
    \Box(\sigma) := \left\{ v \in \bZ^r: v = \sum_{i \in I_\sigma'} c_ib_i \text{ for some } 0 \le c_i < 1 \right\},
$$
which gives a set of representatives for $G_\sigma$. Given $v = \sum_{i \in I_\sigma'} c_i(v)b_i \in \Box(\sigma)$, we define
$$
    \age(v) := \sum_{i \in I_\sigma'} c_i(v).
$$
Given cones $\tau \subseteq \sigma$ in $\Xi$, we have natural inclusions $G_\tau \subseteq G_\sigma$, $\Box(\tau) \subseteq \Box(\sigma)$.

\subsubsection{Fixed points, torus-invariant lines, fundamental groups, flags}
For each $\sigma \in \Xi(r)$, let $\fp_\sigma := \cV(\sigma)$ denote the corresponding $(\bC^*)^r$-fixed point in $\cZ$, and $p_\sigma:= V(\sigma)$ denote the corresponding $(\bC^*)^r$-fixed point in $Z$. For each $\tau \in \Xi(r-1)$, let $\fo_\tau \cong \bC^* \times \cB G_\tau$ denote the corresponding $(\bC^*)^r$-orbit in $\cZ$, and $o_\tau \cong \bC^*$ denote the corresponding $(\bC^*)^r$-orbit in $Z$. Let $\fl_\tau:= \cV(\tau)$ denote the corresponding closed $(\bC^*)^r$-invariant line in $\cZ$, which is the closure of $\fo_\tau$, and $l_\tau := V(\tau)$ denote the corresponding closed $(\bC^*)^r$-invariant line in $Z$, which is the closure of $o_\tau$. We set
$$
    \Xi(r-1)_c:= \{\tau \in \Xi(r-1) : l_\tau \text{ is compact} \},
$$
and define
$$
    \cZ^1_c := \bigcup_{\tau \in \Xi(r-1)_c} \fl_{\tau}, \quad Z^1_c:= \bigcup_{\tau \in \Xi(r-1)_c} l_{\tau}.
$$

For each $\tau \in \Xi(r-1)$, let $H_\tau := \pi_1(\fo_\tau)$ be the fundamental group of $\fo_\tau$. The projection $\fo_\tau \to o_\tau$ to the coarse moduli space induces a map
$$
    \pi_\tau: H_\tau = \pi_1(\fo_\tau) \to \pi_1(o_\tau) \cong \bZ
$$
on fundamental groups, which fits into a split short exact sequence
\begin{equation}\label{eqn:LineFundGroup}
    \xymatrix{
        1 \ar[r] & G_\tau \ar[r] & H_\tau \ar[r]^{\pi_\tau} & \bZ \ar[r] & 1.
    }
\end{equation}

Let
$$
    F(\Xi):= \{ (\tau, \sigma) \in \Xi(r-1) \times \Xi(r) : \tau \text{ is a face of } \sigma\}
$$
be the set of \emph{flags} in $\Xi$. Given a flag $(\tau, \sigma) \in F(\Xi)$, we have $\fp_\sigma \subset \fl_\tau$ and $p_\sigma \in l_\tau$. Let
$$
\chi_{(\tau, \sigma)}: G_\sigma \to \bC^*
$$
be the representation of $G_\sigma$ on the tangent line $T_{\fp_\sigma}\fl_\tau$. The image of $\chi_{(\tau, \sigma)}$ is $\mu_{\fr(\tau, \sigma)}$, where
$$
    \fr(\tau, \sigma):= \frac{|G_\sigma|}{|G_\tau|},
$$
and for $r \in \bZ_{>0}$, $\mu_r \subset \bC^*$ is the cyclic group of $r$-th roots of unity. The map $\chi_{(\tau, \sigma)}$ and the inclusion $G_\tau \to G_\sigma$ fit into a short exact sequence
\begin{equation}\label{eqn:FlagStab}
    \xymatrix{
        1 \ar[r] & G_\tau \ar[r] & G_\sigma \ar[r]^{\chi_{(\tau, \sigma)}} & \mu_{\fr(\tau, \sigma)} \ar[r] & 1.
    }
\end{equation}

 Let $\fu_{(\tau, \sigma)}:= \fo_\tau \cup \fp_\sigma$, which is an open substack of $\fl_\tau$. The inclusion $\fo_\tau \to \fu_{(\tau, \sigma)}$ induces a surjective map
\begin{equation}\label{eqn:FlagFundGroup}
    \pi_{(\tau, \sigma)}: H_\tau = \pi_1(\fo_\tau) \to \pi_1(\fu_{(\tau, \sigma)}) \cong G_\sigma
\end{equation}
on fundamental groups which together with \eqref{eqn:LineFundGroup}, \eqref{eqn:FlagStab} fits into the following commutative diagram:
\begin{equation}\label{eqn:FlagMapDiagram}
    \xymatrix{
        1 \ar[r] & G_\tau \ar[r] \ar[d]^{\text{id}} & H_\tau \ar[r]^{\pi_\tau} \ar[d]^{\pi_{(\tau, \sigma)}} & \bZ \ar[r] \ar[d]^{d \mapsto e^{2\pi\sqrt{-1}d/\fr(\tau, \sigma)}} & 1\\
        1 \ar[r] & G_\tau \ar[r] & G_\sigma \ar[r]_{\chi_{(\tau, \sigma)}} & \mu_{\fr(\tau, \sigma)} \ar[r] & 1.
    }
\end{equation}

\subsubsection{Chen-Ruan orbifold cohomology}
Let
$$
    \Box(\cZ) := \bigcup_{\text{cone $\sigma$ in $\Xi$}} \Box(\sigma) = \bigcup_{\text{maximal cone $\sigma$ in $\Xi$}} \Box(\sigma),
$$
which indexes the inertia components of $\bZ$. The inertia stack of $\cZ$ is
$$
    \cI \cZ = \bigsqcup_{j \in \Box(\cZ)} \cZ_j.
$$
In particular, $\cZ_{\vzero} = \cZ$ is the untwisted sector. Let
$$
    \inv^*: \cI \cZ \to \cI \cZ
$$
denote the involution on $\cI \cZ$ which $(z,g)$ with $z \in \cZ, g \in \Aut(z)$ to $(z, g^{-1})$.

As a graded vector space over $\bQ$  (and as the state-space of relevant quantum theory in physics \cite{Za93}), the \emph{Chen-Ruan cohomology group} \cite{CR04} of $\cZ$ is defined as
$$
    H_{\CR}^*(\cZ; \bQ) := \bigoplus_{j \in \Box(\cZ)} H^*(\cZ_j; \bQ)[2\age(j)],
$$
where $[2\age(j)]$ denotes a degree shift by $2\age(j)$. We write $\one_j$ for the unit of $H^*(\cZ_j; \bQ)$, viewed as an element of $H_{\CR}^{2\age(j)}(\cZ; \bQ)$. In addition, for any subtorus $Q \subseteq (\bC^*)^r$, the \emph{$Q$-equivariant} Chen-Ruan cohomology group of $\cZ$ is
$$
    H_{\CR, Q}^*(\cZ; \bQ) := \bigoplus_{j \in \Box(\cZ)} H^*_Q(\cZ_j; \bQ)[2\age(j)],
$$
which is a module over $H^*_{Q}(\pt;\bQ)$. The above definitions can be extended to $\bC$-coefficients.

\subsection{The toric Calabi-Yau 3-orbifold \texorpdfstring{$\cX$}{X}}
Let $N \cong \bZ^3$. Let $\cX$ be a toric Calabi-Yau 3-orbifold defined by an extended stacky fan $\bSi = (N, \Sigma, \alpha)$ and $X$ be the coarse moduli space of $\cX$. Let $T = N \otimes \bC^* \cong (\bC^*)^3$ be the Deligne-Mumford torus acting on $\cX$ (and $X$).

\begin{assumption}\label{assump:OnX} \rm{
We make the following assumptions:
    \begin{itemize}
        \item $X$ is \emph{Calabi-Yau}: The canonical bundle $K_X$ of $X$ is trivial.
        \item $X$ is \emph{semi-projective}: $X$ is projective over its affinization $\Spec(H^0(X, \cO_X))$.
    \end{itemize}
}\end{assumption}

Suppose the homomorphism $\alpha:\bZ^R \to N$ is specified by $b_i = \alpha(e_i)$ for each $i$, where $\{e_1, \dots, e_R\}$ is the standard basis for $\bZ^R$. Let $M:= \Hom(N, \bZ)$, which is canonically identified with the character lattice $\Hom(T, \bC^*)$ of $T$. The Calabi-Yau condition implies the existence of a character $u_3 \in M$ such that, if $R' = |\Sigma(1)|$ is the number of rays, then $\inner{u_3, b_i} = 1$ for all $i = 1, \dots, R'$. That is, $b_1, \dots, b_{R'}$ belong to $N' \times \{1\}$ where $N':= \ker(u_3) \subset N$.

Let $P$ be the cross section of the support $|\Sigma|$ of $\Sigma$ in the hyperplane $N'_{\bR} \times \{1\}$.\footnote{In this paper, for a lattice $L$ and a field $\bF = \bR, \bQ$, or $\bC$, we let $L_{\bF}$ denote $L \otimes_{\bZ} \bF$.} Then $|\Sigma|$ is the cone over $P$, and $\Sigma$ induces a triangulation of $P$. Moreover, the semi-projectivity condition implies that $P$ is convex. We assume that the additional lattice points $b_{R'+1}, \dots, b_R$ are chosen in a way that $(b_1, \dots, b_R)$ is a listing of the points in $P \cap N$. In particular, the homomorphism $\alpha$ is surjective and fits into the following short exact sequence of lattices:
\begin{equation}\label{eqn:XSES}
    \xymatrix{
        0 \ar[r] & \bL \ar[r]^\psi & \bZ^R \ar[r]^\alpha & N \ar[r] & 0,
    }
\end{equation}
where $\bL := \ker(\alpha) \cong \bZ^{R-3}$.

Let $G:= \bL \otimes (\bC^*) \cong (\bC^*)^{R-3}$. Let $\{\epsilon_1, \dots, \epsilon_{R-3}\}$ be a basis for $\bL$, and for each $a = 1, \dots, R-3$, let
$$
    l^{(a)} = (l^{(a)}_1, \dots, l^{(a)}_R) := \psi(\epsilon_a) \in \bZ^R.
$$
The vectors $l^{(a)}$ are known as \emph{charge vectors}, which describe the linear action of $G$ on $\bC^R = \bZ^R \otimes \bC = \Spec(\bC[x_1, \dots, x_R])$ induced by the inclusion $\psi$, as follows:
\begin{equation}\label{eqn:GAction}
    (s_1, \dots, s_{R-3}) \cdot (x_1, \dots, x_R) = \left(\prod_{a=1}^{R-3}s_a^{l_1^{(a)}}x_1, \dots, \prod_{a=1}^{R-3}s_a^{l_R^{(a)}}x_R  \right),
\end{equation}
where $(s_1, \dots, s_{R-3})$ are coordinates on $G$ specified by the basis $\{\epsilon_1, \dots, \epsilon_{R-3}\}$. Under this action, $\cX$ can be described as the quotient stack
\begin{equation}\label{eqn:XQuotient}
    \cX = \left[ ((\bC^{R'} \setminus Z(\Sigma)) \times (\bC^*)^{R-R'})/G \right],
\end{equation}
where $Z(\Sigma)$ is a closed subvariety of $\bC^{R'}$ defined by $\Sigma$. The semi-projectivity condition (Assumption \ref{assump:OnX}) implies that the above is a GIT quotient.

Given a flag $(\tau, \sigma) \in F(\Sigma)$, $G_\tau$ is a cyclic subgroup of $G_\sigma$. We define
$$
    \fm(\tau, \sigma):= |G_\tau|.
$$

\subsection{The framed Aganagic-Vafa brane \texorpdfstring{$(\cL,f)$}{L}}\label{sect:AVBrane}
In this section, we describe the symplectic structure on $\cX$ and define the Aganagic-Vafa brane $\cL$, following Fang-Liu-Tseng \cite{FLT12}. Let $G_{\bR} \cong U(1)^{R-3}$ be the maximal compact subgroup of $G$, which carries a Hamiltonian action on $\bC^R$ induced by \eqref{eqn:GAction}. The moment map $\tmu: \bC^R \to \fg_{\bR}^*$ of the $G_{\bR}$-action, where $\fg_{\bR}^* \cong \bR^{R-3}$ is the dual of the Lie algebra $\fg_{\bR}$ of $G_{\bR}$, can be described as
$$
    \tmu(x_1, \dots, x_R) = \left(\sum_{i = 1}^R l_i^{(1)}|x_i|^2, \dots, \sum_{i = 1}^R l_i^{(R-3)}|x_i|^2  \right).
$$

Applying $\Hom(-, \bZ)$ to \eqref{eqn:XSES}, we obtain a short exact sequence
\begin{equation}\label{eqn:XDualSES}
    \xymatrix{
        0 \ar[r] & M \ar[r]^{\alpha^\vee} & \bZ^R \ar[r]^{\psi^\vee} & \bL^\vee \ar[r] & 0.
    }
\end{equation}
There is a canonical identification $\fg_{\bR}^* \cong \bL_{\bR}^\vee$. Let $\{e_1^\vee, \dots, e_R^\vee\}$ be the basis dual to $\{e_1, \dots, e_R\}$. For each $i = 1, \dots, R$, define
$$
    D_i := \psi^\vee(e_i^\vee) \in \bL^\vee.
$$
For each maximal cone $\sigma \in \Sigma(3)$, define the \emph{extended $\sigma$-nef cone} as
$$
    \tNef(\sigma):= \sum_{i \in I_\sigma} \bR_{\ge 0} D_i.
$$
The \emph{extended nef cone} of $\cX$ is defined to be
$$
    \tNef(\cX):= \bigcap_{\sigma \in \Sigma(3)} \tNef(\sigma),
$$
which is an $(R-3)$-dimensional simplicial cone in $\bL_{\bR}^\vee$.

Let $r = (r_1, \dots, r_{R-3})$ be a point in the interior of $\tNef(\cX)$, which can be viewed as an \emph{extended K\"ahler class} of $\cX$. Then $\cX$ is the symplectic quotient
$$
    [\tmu^{-1}(r)/G_{\bR}],
$$
and the standard K\"ahler form
$$
    \frac{\sqrt{-1}}{2} \sum_{i=1}^R dx_i \wedge d\bar{x}_i
$$
on $\bC^{R}$ descends to a K\"ahler form $\omega_r$ on $\cX$.

An \emph{Aganagic-Vafa brane} \cite{AV00,FLT12} $\cL$ in $\cX$ is a Lagrangian suborbifold of form
$$
    \cL:= \left[ \left. \left\{ (x_1, \dots, x_R) \in \tmu^{-1}(r) : \sum_{i = 1}^R l_i' |x_i|^2 = c', \sum_{i=1}^R l_i''|x_i|^2 = c'', \arg(\prod_{i=1}^R x_i) = c'''       \right\}  \middle/ G_{\bR} \right. \right],
$$
where $c', c'', c''' \in \bR$ are constants and vectors $l' = (l_1', \dots, l_R'), l'' = (l_1'', \dots, l_R'') \in \bZ^R$ satisfy
$$
    \sum_{i=1}^R l_i' = \sum_{i=1}^R l_i'' = 0.
$$
Let $L \subset X$ be the coarse moduli space of $\cL$. The brane $\cL$ intersects a unique $T$-invariant line $\fl_{\tau_0}$ in $\cX$, where $\tau_0 \in \Sigma(2)$. 

\begin{assumption}\label{assump:OnL} \rm{
We assume that $\cL$ is an \emph{outer} brane: $\tau_0 \not \in \Sigma(2)_c$.
}\end{assumption}

The inclusions $\cL \cap \fl_{\tau_0} \to \cL$, $\cL \cap \fl_{\tau_0} \to \fo_{\tau_0}$ are homotopy equivalences. We have
$$
    \pi_1(\cL) = \pi_1(\fo_{\tau_0}) =  H_{\tau_0} \cong \bZ \times G_{\tau_0},
$$
which is abelian and thus also equal to $H_1(\cL;\bZ)$. 

Moreover, $\tau_0$ is contained in a unique $3$-cone $\sigma_0 \in \Sigma(3)$. After a possible permutation of indices, we assume that $I_{\sigma_0}' = \{1,2,3\}$ with $b_1, b_2, b_3$ appearing in $N' \times \{1\}$ in counterclockwise order, and that $I_{\tau_0}' = \{2, 3\}$. Let $\tau_2, \tau_3 \in \Sigma(2)$ be the other two facets of $\sigma_0$, with $I_{\tau_2}' = \{1,3\}$ and $I_{\tau_3}' = \{1,2\}$. See Figure \ref{fig:DistConesOuter}.

\begin{figure}[h]
    \begin{tikzpicture}[scale=0.8]
        \coordinate (01) at (4, -1);
        \coordinate (02) at (0, 3);        
        \coordinate (03) at (0, 0);
                
        \node at (01) {$\bullet$};
        \node at (02) {$\bullet$};
        \node at (03) {$\bullet$};

        \node at (5.4, -1.2) {$b_1 = (\fr, -\fs, 1)$};
        \node at (-1.4, 3.2) {$b_2 = (0, \fm, 1)$};
        \node at (-1.4, -0.2) {$b_3 = (0, 0, 1)$};

        \draw (01) -- (02) -- (03) -- (01);

        \node at (1.1, 0.7) {$\sigma_0$};
        \node at (-0.5, 1.5) {$\tau_0$};
        \node at (1.5, -0.8) {$\tau_2$};
        \node at (2.5, 1.4) {$\tau_3$};
    \end{tikzpicture}    

    \caption{Distinguished rays and cones associated to the Aganagic-Vafa outer brane $\cL$.}
    \label{fig:DistConesOuter}
\end{figure}

Let
$$
    \fm:= \fm(\tau_0, \sigma_0), \quad \fr:= \fr(\tau_0, \sigma_0).
$$
The flag $(\tau_0, \sigma_0)$ determines a basis $\{v_1, v_2, v_3\}$ for $N$ under which
$$
    b_1 = (\fr, -\fs, 1), \quad b_2 = (0, \fm, 1), \quad b_3 = (0,0,1)
$$
for some $\fs \in \{0, 1, \dots, \fr-1\}$. For $i = 1, \dots, R$, let $(m_i, n_i, 1)$ be the coordiante of $b_i$ under the basis $\{v_1, v_2, v_3\}$. Since $\cL$ is outer, we have $m_i \ge 0$ for all $i$. Let $u_1, u_2 \in M$ such that $\{u_1, u_2, u_3\}$ is the basis dual to $\{v_1, v_2, v_3\}$. The corresponding characters $\{\su_1, \su_2, \su_3\}$ of $T$ serve as equivariant parameters:
$$
    H^*_T(\pt;\bZ) = \bZ[\su_1, \su_2, \su_3].
$$
Let $T':= \ker(\su_3)$ be the 2-dimensional \emph{Calabi-Yau subtorus} of $T$ and $T'_{\bR}$ be the maximal compact subgroup of $T'$. Then $\cL$ is preserved under the $T_{\bR}'$-action. We have
$$
    H^*_{T'}(\pt;\bZ) = H^*_{T_{\bR}'}(\pt;\bZ) = \bZ[\su_1, \su_2].
$$
Let
$$
    \cQ_{T'} := \bQ(\su_1, \su_2)
$$
be the fractional field of $H^*_{T'}(\pt;\bQ) = H^*_{T_{\bR}'}(\pt;\bQ)$.

\subsubsection{Framing}
Let $f \in \bQ$ be a rational number called the \emph{framing} on $\cL$. We write
$$
    f = \frac{\fb}{\fa}
$$
where $\fa \in \bZ_{>0}$, $\fb \in \bZ$ are coprime integers. The framing $f$ determines a $1$-dimensional \emph{framing subtorus} $T_f:= \ker(\su_2 - f\su_1)$ of $T'$.
We have 
$$
    H^*_{T_f}(\pt;\bZ) = \bZ[\su]
$$
where $\su$ is a generator of $H^2_{T_f}(\pt;\bZ)\cong \Hom(T_f, \bC^*)\cong \bZ$.  
The injective group homomorphism $T_f \to  T'  $ induces a surjective ring homomorphism
$$
H^*_{T'}(\pt;\bZ)=\bZ[\su_1, \su_2] \lra  H^*_{T_f}(\pt;\bZ)=\bZ[\su], \quad  \su_1\mapsto \fa\su, \quad \su_2\mapsto \fb\su.
$$
Let
$$
    \cQ_{T_f} := \bQ(\su)
$$
be the fractional field of $H^*_{T_f}(\pt;\bQ)=\bQ[\su]$.


\subsection{The corresponding toric Calabi-Yau 4-orbifold \texorpdfstring{$\tcX$}{X}}
The corresponding toric Calabi-Yau 4-orbifold $\tcX$ is specified by the extended stacky fan $\tbSi = (\tSi, \tN, \talpha)$ where:
\begin{itemize}
    \item $\tN := N \oplus \bZ v_4 \cong \bZ^4$ is a 4-dimensional lattice.

    \item Let $\bZ^{R+2} \cong \bZ^R \oplus \bZ^2$ whose standard basis extends $\{e_1, \dots, e_R\}$ by $e_{R+1}, e_{R+2}$. The homomorphism $\talpha: \bZ^{R+2} \to \tN$ maps $e_1, \dots, e_{R+2}$ to $\tb_1, \dots, \tb_{R+2} \in \tN$, where
    $$
        \tb_i = (m_i, n_i, 1, 0) \quad \text{for } i = 1, \dots, R,
    $$
    $$
        \tb_{R+1} = (-\fa, -\fb, 1,1), \quad \tb_{R+2} = (0, 0, 1, 1).
    $$

    \item $\tSi$ has the following description: Let $\tM:= \Hom(\tN, \bZ)$ with basis $\{u_1, u_2, u_3, u_4\}$ dual to $\{v_1, v_2, v_3, v_4\}$. We abuse notation here since $u_1, u_2, u_3 \in \tM$ maps to $u_1, u_2, u_3 \in M$ respectively under the natural projection. Let $\tP$ be the convex hull of $\{\tb_1, \dots, \tb_{R+2}\}$ in $\tN_{\bR}$, which is a 3-dimensional convex polytope contained in the hyperplane $\tN'_{\bR} \times \{1\}$, where $\tN':= \ker(u_3) \subset \tN$. Note that $\tP$ contains $P$ as a facet. Let $\tP_0$ be the convex hull of $P \cup \{\tb_{R+2}\}$. We triangulate $\tP_0$ by taking the cone over the triangulation of $P$ induced by $\Sigma$, and extend this to a triangulation of $\tP$ that includes the simplex
    $$
        \{2, 3, R+1, R+2\}.
    $$
    Finally, let $\tSi$ be the fan which has support $|\tSi|$ equal to the cone over $\tP$ and is induced by the above triangulation of $\tP$.
\end{itemize}
Let $\tX$ be the coarse moduli space of $\tcX$, which is the simplicial toric 4-fold determined by the fan $\tSi$. By construction, $\tX$ is both \emph{Calabi-Yau} and \emph{semi-projective}.

Note that
$$
    \tSi(1) = \{\trho_{1}, \dots, \trho_{R'}, \trho_{R+1}, \trho_{R+2}\} \quad \text{where } \trho_i:= \bR_{\ge 0} \tb_i.
$$
Let $\tSi_0$ be the subfan of $\tSi$ whose support is the cone over $\tP_0$. We have
$$
    \tSi_0(1) = \{\trho_{1}, \dots, \trho_{R'}, \trho_{R+2}\}.
$$
The inclusions of fans
$$
    \Sigma \to \tSi_0 \to \tSi
$$
induce inclusions of toric orbifolds
$$
    \cX \to \cX \times \bC \to \tcX,
$$
and we denote the composition by $\iota: \cX \to \tcX$.

\begin{remark}\rm{
When $\cX = X$ is a smooth toric Calabi-Yau 3-fold, $\cL = L$ is an outer brane, and $f \in \bZ$, which is considered by \cite{LY21}, $\tcX = \tX$ constructed above is a smooth toric Calabi-Yau 4-fold and a semi-projective partial compactification of the 4-fold corresponding to $(X,L,f)$ constructed by \cite{LY21}, which is specified by the fan $\tSi_0$ together with the $4$-cone $\tsi_0$.
}\end{remark}

The lattice map $\talpha:\bZ^{R+2} \to \tN$ fits into the short exact sequence at the second row of the following commutative diagram:
\begin{equation}\label{eqn:tXSES}
    \xymatrix{
        0 \ar[r] & \bL \ar[r]^\psi \ar[d] & \bZ^R \ar[r]^\alpha \ar[d] & N \ar[r] \ar[d] & 0\\
        0 \ar[r] & \tbL \ar[r]^{\tpsi} & \bZ^{R+2} \ar[r]^{\talpha} & \tN \ar[r] & 0,
    }
\end{equation}
where $\tbL := \ker(\talpha) \cong \bZ^{R-2}$. We can identify $\bL$ as a sublattice of $\tbL$ and extend $\{\epsilon_1, \dots, \epsilon_{R-3}\}$ to a basis for $\tbL$ by including an additional vector $\epsilon_{R-2} \in \tbL$, in a way that
$$
    \tl^{(a)} := \tpsi(\epsilon_a) = (l_1^{(a)}, \dots, l_{R}^{(a)}, 0, 0) \quad \text{for } a = 1, \dots, R-3.
$$
Denote $\tl^{(R-2)} := \tpsi(\epsilon_{R-2})$. Then similar to \eqref{eqn:XQuotient}, $\tcX$ can be described as a quotient stack $\tcX = [((\bC^{R'+2} \setminus Z(\tSi)) \times (\bC^*)^{R-R'})/\tG]$ which is also a GIT quotient, where the linear action of $\tG:= \tbL \otimes (\bC^*) = (\bC^*)^{R-2}$ on $\bC^{R+2}$ is specified by $\tl^{(1)}, \dots, \tl^{(R-2)}$.

Let $\tT = \tN \otimes \bC^* \cong (\bC^*)^4$ be the complex algebraic torus acting on $\tcX$ (and $\tX$), which contains $T$ as a subtorus. The character lattice of $\tT$ is canonically identified with $\tM$. We have
$$
    H^*_{\tT}(\pt;\bZ) = \bZ[\su_1, \su_2, \su_3, \su_4],
$$
where $\{\su_1, \su_2, \su_3, \su_4\}$ are characters corresponding to $\{u_1, u_2, u_3, u_4\}$. Let $\tT' := \ker(\su_3)$ be the 3-dimensional \emph{Calabi-Yau subtorus} of $\tT$, which contains $T'$ as a subtorus. We have
$$
    H^*_{\tT'}(\pt;\bZ) = \bZ[\su_1, \su_2, \su_4].
$$
Let
$$
    \cQ_{\tT'} := \bQ(\su_1, \su_2, \su_4)
$$
be the fractional field of $H^*_{\tT'}(\pt;\bQ)$.

\subsection{Comparison of toric geometry and topology}\label{sect:ComparisonOuter}
We draw additional comparisons between the toric geometry and topology of $(\cX,\cL)$ and $\tcX$, mainly regarding the second homology groups, stablizer groups of torus-invariant substacks, and Chen-Ruan orbifold cohomology groups.

\subsubsection{Cones and flags}\label{sect:4ConeOuter}
Recall that we have an inclusion $\iota: X \to \tX$. On the level of cones, we have an injective map 
\begin{equation}\label{eqn:ConeMapIotaOuter}
    \iota: \Sigma(d) \to \tSi(d+1)
\end{equation}
for each $d = 0, 1, 2, 3$, given by
$$
    I_{\iota(\sigma)}' = I_{\sigma}' \sqcup \{R+2\} \quad \text{for all } \sigma \in \Sigma(d).
$$
In particular, $\iota(\Sigma(3))$ gives a set of $4$-cones in $\tSi$. Given any other $4$-cone $\tsi \in \tSi(4) \setminus \iota(\Sigma(3))$, we have that $R+1, R+2 \in I_{\tsi}'$ and the other two indices specify a $2$-cone $\delta_0(\tsi) \in \Sigma(2) \setminus \Sigma(2)_c$. This yields a map
\begin{equation}\label{eqn:MapDelta0}
    \delta_0: \tSi(4) \setminus \iota(\Sigma(3)) \to \Sigma(2) \setminus \Sigma(2)_c.
\end{equation}
Then for any $\tau$ in the image of $\delta_0$, we have $\iota(\tau) \in \tSi(3)_c$. In particular, let $\tsi_0 \in \tSi(4)$ be the $4$-cone with
$$
    I_{\tsi_0}' = \{2, 3, R+1, R+2\}.
$$
Then $\delta_0(\tsi_0) = \tau_0$. For each $\tsi \in \tSi(4) \setminus \iota(\Sigma(3))$, we set
$$
    I_{\delta_0(\tsi)}' = \{i_2(\tsi), i_3(\tsi)\}
$$
such that $b_{i_2(\tsi)}, b_{i_3(\tsi)}$ appear on the boundary of $\Delta$ in $N_{\bR}' \times \{1\}$ in counterclockwise order. In particular, $i_2(\tsi_0) = 2$ and $i_3(\tsi_0) = 3$. We have
$$
    I_{\tsi}' = \{i_2(\tsi), i_3(\tsi), R+1, R+2\}.
$$

There is an induced injective map of flags $\iota: F(\Sigma) \to F(\tSi)$ given by $\iota(\tau, \sigma) = (\iota(\tau), \iota(\sigma))$. Moreover, for any $3$-cone $\sigma \in \Sigma(3)$, we have $(\sigma, \iota(\sigma)) \in F(\tSi)$. Any other flag in $\tSi$ consitutes of a $4$-cone $\tsi \in \tSi(4) \setminus \iota(\Sigma(3))$ and one of its facets. In particular, we have $(\iota(\delta_0(\tsi)), \tsi) \in F(\tSi)$. The other facets $\delta_2(\tsi)$, $\delta_3(\tsi)$, $\delta_4(\tsi) \in \tSi(3)$ are specified by
$$
    I_{\delta_2(\tsi)}' = \{i_3(\tsi), R+1, R+2\}, \quad I_{\delta_3(\tsi)}' = \{i_2(\tsi), R+1, R+2\}, \quad I_{\delta_4(\tsi)}' = \{i_2(\tsi), i_3(\tsi), R+1\}.
$$
This yields maps
$$
    \delta_2, \delta_3, \delta_4: \tSi(4) \setminus \iota(\Sigma(3)) \to \tSi(3).
$$

\subsubsection{Stablizers}\label{sect:stablizer}
For any cone $\sigma$ in $\Sigma$, we have
$$
    G_{\iota(\sigma)} = G_{\sigma}, \quad \Box(\iota(\sigma)) = \Box(\sigma).
$$
Here and in the rest of this subsection, in view of the natural inclusion $N \to \tN$ which also makes $\sigma$ a cone in $\tSi$, it does not make a difference to define $\Box(\sigma)$ with respect to $\Sigma$ or $\tSi$. In particular, if $\sigma \in \Sigma(3)$, we have
$$
    \fr(\sigma, \iota(\sigma)) = 1
$$
for the flag $(\sigma, \iota(\sigma))$; for any flag $(\tau, \sigma) \in F(\Sigma)$, $G_{\iota(\tau)}$ is a cyclic subgroup of $G_{\iota(\sigma)}$ of order $\fm(\tau, \sigma)$, and
$$
    \fr(\iota(\tau,\sigma)) = \fr(\tau, \sigma).
$$

Moreover, a direct computation gives the following charaterization of stablizers of cones in $\tSi(4) \setminus \iota(\Sigma(3))$.

\begin{lemma}\label{lem:ExtraStab}
For any $\tsi \in \tSi(4) \setminus \iota(\Sigma(3))$, $G_{\tsi}$ is a cyclic group of order
$$
    |G_{\tsi}| = \fb(m_{i_3(\tsi)} - m_{i_2(\tsi)}) + \fa(n_{i_2(\tsi)} - n_{i_3(\tsi)}).
$$
If $|G_{\tsi}|>1$, then a generator is given by the element
$$
    \frac{|G_{\tsi}|-1}{|G_{\tsi}|}\tb_{i_2(\tsi)} + \frac{1}{|G_{\tsi}|}\tb_{i_3(\tsi)} + \frac{1}{\fa}\inner{\frac{m_{i_3(\tsi)} - m_{i_2(\tsi)}}{|G_{\tsi}|}}\tb_{R+1} + \inner{1-\frac{1}{\fa}\inner{\frac{m_{i_3(\tsi)} - m_{i_2(\tsi)}}{|G_{\tsi}|}}}\tb_{R+2}
$$
in $\Box(\tsi)$.
\end{lemma}

Elements of age at most 1 are precisely those contained in the cyclic subgroup
$$
    G_{\iota(\delta_0(\tsi))} = G_{\delta_0(\tsi)} = G_{\delta_4(\tsi)} \cong \mu_{\gcd(|m_{i_2(\tsi)} - m_{i_3(\tsi)}|, |n_{i_2(\tsi)} - n_{i_3(\tsi)}|)}.
$$
In addition, we have
$$
    G_{\delta_2(\tsi)} = G_{\delta_3(\tsi)} = \{1\}.
$$
Elements in $G_{\tsi} \setminus G_{\delta_0(\tsi)}$ all have age 2. In the example of $\tsi_0$, $G_{\tsi_0}$ has order $\fa\fm$ and the subgroup above has order $\fm$; when $\fm>1$, this subgroup is generated by $(0,1,1,0)$.

In summary, there is an inclusion
$$
    \Box(\cX) \subseteq \Box(\tcX),
$$
where any $j \in \Box(\tcX) \setminus \Box(\cX)$ is contained in $\Box(\tsi) \setminus \Box(\delta_0(\tsi))$ for some $\tsi \in \tSi(4) \setminus \iota(\Sigma(3))$ and has age $2$.

\subsubsection{Second homology}\label{sect:HomologyOuter}
The inclusion $\iota: X \to \tX$ induces an inclusion on second integral homology, described as follows. The inclusion $\iota: \Sigma(2) \to \tSi(3)$ restricts to an inclusion $\Sigma(2)_c \to \tSi(3)_c$. Then we have
$$
    \iota_*: H_2(X; \bZ) \to H_2(\tX; \bZ), \quad [l_{\tau}] \mapsto [l_{\iota(\tau)}] \quad \text{for all } \tau \in \Sigma(2)_c.
$$
Moreover, the brane $L$ bounds a holomorphic disk $B$ in $l_{\tau_0}$, which we orient by the holomorphic structure of $X$. Then
$$
    H_1(L; \bZ) \cong \bZ \partial [B], \quad H_2(X,L;\bZ) \cong H_2(X; \bZ) \oplus \bZ [B].
$$
The map $\iota_*$ on second homology above can be extended to an inclusion
$$
    \iota_*: H_2(X, L; \bZ) \to H_2(\tX; \bZ)
$$
which maps $[B]$ to $[l_{\iota(\tau_0)}]$. Note that $\iota(\tau_0) \in \tSi(3)_c$.

\subsubsection{Chen-Ruan cohomology ring}
The Chen-Ruan cohomology ring of a smooth toric Deligne-Mumford stack is equal to its orbifold Chow ring. Borisov-Chen-Smith \cite{BCS05} provided an explicit description of the orbifold Chow ring of smooth toric Deligne-Mumford stacks with projective coarse moduli spaces. For any toric Deligne-Mumford stack $\cZ$, Borisov-Horja \cite[Section 3]{BH06} introduced the SR-cohomology ring which is determined by the stacky fan. When the coarse moduli space $Z$ is projective, the SR-cohomology ring of $\cZ$ coincides with the orbifold Chow ring of $\cZ$ by \cite[Theorem 1.1]{BCS05}; when $Z$ is not projective, the SR-cohomology ring and the orbifold Chow ring can be different, even when $\cZ =Z$ is a smooth toric variety. 
Jiang-Tseng \cite{JT08} generalized \cite[Theorem 1.1]{BCS05} to smooth toric Deligne-Mumford stacks with semi-projective coarse moduli spaces; the formula is in terms of the extended stacky fan introduced by Jiang \cite{Jiang08}.  Applying \cite[Theorem 1.1]{JT08} to $\cX$ and $\tcX$, we obtain the following statements. 
\begin{itemize}
\item As graded $\bQ$-algebras, $H^*_{\CR}(\cX; \bQ)$ is generated by $\{\one_j : j \in \Box(\cX)\}$ and the divisor classes
$$
    \cD_i := [\cV(\rho_i)] \in H^2_{\CR}(\cX; \bZ), \quad i = 1, \dots, R',
$$
and $H^*_{\CR}(\tcX; \bQ)$ is generated by $\{\one_j : j \in \Box(\tcX)\}$ and the divisor classes
$$
    \tcD_i := [\cV(\trho_i)] \in H^2_{\CR}(\tcX; \bZ), \quad i = 1, \dots, R', R+1, R+2.
$$
\item The inclusion $\iota: \cX \to \tcX$ induces a homogenous $\bQ$-algebra homomorphism  
$$
    \iota^*: H^*_{\CR}(\tcX; \bQ) \to H^*_{\CR}(\cX; \bQ)
$$
specified by
$$
    \one_j \mapsto \begin{cases}
        \one_j & \text{if } j \in \Box(\cX)\\
        0 &  \text{if } j \in \Box(\tcX) \setminus \Box(\cX),
        \end{cases}
$$
$$
     \tcD_i \mapsto \cD_i \text{ for } i = 1, \dots, R', \quad \tcD_{R+1}, \tcD_{R+2} \mapsto 0.
$$
\end{itemize} 

There is a canonical identification $H^2_{\CR}(\cX;\bQ) \cong \bL^\vee_{\bQ}$ which identifies $\cD_i$ with $D_i$ for each $i = 1, \dots, R'$ and $\one_{j(i)}$ with $D_{i}$ if $b_i \in N$ is a representative of $j(i) \in \Box(\cX)$. On the other hand, applying $\Hom(-, \bZ)$ to \eqref{eqn:tXSES}, we obtain a commutative diagram
\begin{equation}\label{eqn:tXDualSES}
    \xymatrix{
        0 \ar[r] & \tM \ar[r]^{\talpha^\vee} \ar[d] & \bZ^{R+2} \ar[r]^{\tpsi^\vee} \ar[d] & \tbL^\vee \ar[r] \ar[d] & 0\\
        0 \ar[r] & M \ar[r]^{\alpha^\vee} & \bZ^R \ar[r]^{\psi^\vee} & \bL^\vee \ar[r] & 0
    }
\end{equation}
where the second row is \eqref{eqn:XDualSES}. Let $\{e_1^\vee, \dots, e_{R+2}^\vee\}$ be the basis of $\bZ^{R+2}$ dual to $\{e_1, \dots, e_{R+2}\}$, and for each $i = 1, \dots, R+2$ define
\begin{equation}\label{eqn:tDiOuter}
    \tD_i := \tpsi^\vee(e_i^\vee) \in \tbL^\vee.
\end{equation}
For $i = 1, \dots, R$, $\tD_i$ projects to $D_i \in \bL^\vee$. Moreover,
$$
    \tD_{R+1} = -\tD_{R+2} \in \bZ_{\neq 0} \epsilon_{R-2}^\vee,
$$
where $\{\epsilon_1^\vee, \dots, \epsilon_{R-2}^\vee\}$ be the basis of $\tbL^\vee$ dual to $\{\epsilon_1, \dots, \epsilon_{R-2}\}$. There is a canonical identification $H^2_{\CR}(\tcX;\bQ) \cong \tbL^\vee_{\bQ}$ which identifies $\tcD_i$ with $\tD_i$ for each $i = 1, \dots, R',R+1, R+2$ and $\one_{j(i)}$ with $\tD_{i}$ if $\tb_{i} \in \tN$ is a representative of $j(i) \in \Box(\tcX)$. The map $\iota^*: H^2_{\CR}(\tcX; \bQ) \to H^2_{\CR}(\cX; \bQ)$ is canonically identifed with the projection $\tbL^\vee_{\bQ} \to \bL^\vee_{\bQ}$.

Moreover, we define $T'$-equivariant divisor classes
$$
    \cD_i^{T'} := [\cV(\rho_i)] \in H^2_{\CR,T'}(\cX; \bQ), \quad i = 1, \dots, R'
$$
of $\cX$ and $\tT'$-equivariant divisor classes
$$
    \tcD_i^{\tT'} := [\cV(\trho_i)] \in H^2_{\CR,\tT'}(\tcX; \bQ), \quad i = 1, \dots, R', R+1, R+2
$$
of $\tcX$. The non-equivariant limit of each $\cD_i^{T'}$ (resp. $\tcD_i^{\tT'}$) is $\cD_i$ (resp. $\tcD_i$). Moreover, for $i = 1, \dots, R'$, $\tcD_i^{\tT'}$ restricts to $\cD_i^{T'}$ under
\begin{equation}\label{eqn:H2RestrictOuter}
    H^*_{\CR, \tT'}(\tcX; \bQ) \to H^*_{\CR, T'}(\tcX; \bQ) \to H^*_{\CR, T'}(\cX; \bQ)
\end{equation}
where the first map is induced by $T' \subset \tT'$ and the second maps is induced by $\iota: \cX \to \tcX$.

\subsubsection{Convention for equivariant lifts}
The Gromov-Witten invariants considered in this paper may take in suitable equivariant lifts of classes in $H^2_{\CR}(\cX;\bQ)$ and $H^2_{\CR}(\tcX;\bQ)$ as insertions. Here, we specify our convention for choosing the lifts.

\begin{convention}\label{conv:LiftOuter}
\rm{
Given $\gamma \in H^2_{\CR}(\cX; \bQ)$, we choose the unique lifts
$$
    \gamma^{T'} \in H^2_{\CR, T'}(\cX; \bQ), \quad \tgamma \in H^2_{\CR}(\tcX; \bQ), \quad \tgamma^{\tT'} \in H^2_{\CR, \tT'}(\tcX; \bQ)
$$
of $\gamma$ that are consistent with the commutative diagram
$$
\xymatrix{
  H^2_{\CR, \tT'}(\tcX; \bQ) \ar[r] \ar[rd] & H^2_{\CR, T'}(\tcX; \bQ) \ar[r] \ar[d]^{\iota^*} & H^2_{\CR}(\tcX; \bQ) \ar[d]^{\iota^*}\\
  & H^2_{\CR, T'}(\cX; \bQ) \ar[r] & H^2_{\CR}(\cX; \bQ)
} \qquad
\xymatrix{
  \tgamma^{\tT'} \ar@{|->}[rr] \ar@{|->}[rd] &  & \tgamma \ar@{|->}[d]\\
  & \gamma^{T'} \ar@{|->}[r] & \gamma
}
$$
and satisfy the following: If $\gamma \in H^2(\cX;\bQ)$, then
\begin{itemize}
    \item $\iota_{\sigma_0}^*(\gamma^{T'}) = 0$;
    \item $\tgamma$ belongs to the span of $\tD_1, \dots, \tD_{R'}$ in $H^2_{\CR}(\tcX; \bQ)$;
    \item $\iota_{\iota(\sigma_0)}^*(\tgamma^{\tT'}) = \iota_{\tsi_0}^*(\tgamma^{\tT'}) = 0$.
\end{itemize}
If $\gamma = \one_j$ for some $j \in \Box(\cX) \subset \Box(\tcX)$, then all lifts are $\one_j$. 

Given $\tgamma \in H^2_{\CR}(\tcX; \bQ)$, we choose the unique lift
$$
    \tgamma^{\tT'} \in H^2_{\CR, \tT'}(\tcX; \bQ)
$$
of $\tgamma$ that satisfies the following: If $\tgamma \in H^2(\tcX;\bQ)$, then $\iota_{\tsi_0}^*(\tgamma^{\tT'}) = 0$. If $\tgamma = \one_j$ for some $j \in \Box(\cX) \subset \Box(\tcX)$, then $\tgamma^{\tT'} = \one_j$. 
}
\end{convention}



\subsection{Examples of construction}\label{sect:examples}
In this section, we provide three examples for our construction of $\tcX$ from $(\cX, \cL, f)$.

\subsubsection{$\bC^3$}\label{sect:ExampleC3}
Let $\cX = \bC^3$ and $\cL$ be an outer brane. Then, we have
$$
    \tcX = \begin{cases}
        \Tot(\cO_{\bP(1,\fa)}(\fb) \oplus \cO_{\bP(1,\fa)}(-\fb-\fa) \oplus \cO_{\bP(1,\fa)}(-1)) & \text{if } f \in [-1, 0]\\
        \Tot(\cO_{\bP(1,\fa,\fb)}(-\fb-\fa) \oplus \cO_{\bP(1,\fa,\fb)}(-1)) & \text{if } f > 0\\
        \Tot(\cO_{\bP(1,\fa,-\fb-\fa)}(\fb) \oplus \cO_{\bP(1,\fa,-\fb-\fa)}(-1)) & \text{if } f < -1.
    \end{cases}
$$
See also \cite[Section 3]{BBvG20b}. We note that for $f \not \in \{-2, -1, 0, 1\}$, $\tcX$ is not a smooth manifold, even though $\cX$ is. See Figure \ref{fig:ExC3} for an illustration in the case $f = 1$, where the fan of $\cX$ is the cone over the triangle on the left and the position of the brane $\cL$ is indicated by the short boldfaced dash, and the fan of $\tcX$ is the cone over the triangulated 3-dimensional polytope on the right.


\begin{figure}[h]
\begin{center}
    \begin{tikzpicture}[scale=1.5]
        \coordinate (01) at (-4.2, 0.1);
        \coordinate (02) at (-5, 0.9);        
        \coordinate (03) at (-5, 0.1);
                
        \node at (01) {$\bullet$};
        \node at (02) {$\bullet$};
        \node at (03) {$\bullet$};

        \node at (-3.5, -0.3) {$b_1 = (1, 0, 1)$};
        \node at (-5, 1.3) {$b_2 = (0, 1, 1)$};
        \node at (-5.5, -0.3) {$b_3 = (0, 0, 1)$};

        \draw[ultra thick] (-5.1, 0.5) -- (-4.9, 0.5);
        \node[left] at (-5.2, 0.5) {$(\cL, f = 1)$};

        \draw (01) -- (02) -- (03) -- (01);

        \node at (-4.7, -1) {$\cX = \bC^3$};


        \coordinate (1) at (1, 0);
        \coordinate (2) at (0.6, 0.7);        
        \coordinate (3) at (0, 0);
        \coordinate (4) at (-0.6, 0.3);        
        \coordinate (5) at (0, 1);

        \node at (1) {$\bullet$};
        \node at (2) {$\bullet$};
        \node at (3) {$\bullet$};
        \node at (4) {$\bullet$};
        \node at (5) {$\bullet$};

        \node at (2.1, 0) {$\tb_1 = (1, 0, 1, 0)$};
        \node at (1.7, 0.7) {$\tb_2 = (0, 1, 1, 0)$};
        \node at (0.2, -0.4) {$\tb_3 = (0, 0, 1, 0)$};
        \node at (-1.5, 0.6) {$\tb_4 = (-1, -1, 1, 1)$};
        \node at (0, 1.4) {$\tb_5 = (0, 0, 1, 1)$};

        \draw (1) -- (3) -- (4) -- (5) -- (2) -- (1) -- (4);
        \draw (1) -- (5);
        \draw[dashed] (5) -- (3) -- (2) -- (4);

        \node at (0.3, -1) {$\tcX = \Tot(\cO_{\bP^2}(-2) \oplus \cO_{\bP^2}(-1))$};
    \end{tikzpicture}
\end{center}

    \caption{Construction of the 4-fold $\tcX$ corresponding to $\cX = \bC^3$ and $f=1$.}
    \label{fig:ExC3}
\end{figure}

\subsubsection{Local $\bP^2$}
Consider $\cX_- = [\bC^3/\bZ_3]$ and its crepant resolution $\cX_+ = \Tot(K_{\bP^2})$. An outer brane $\cL_-$ in $\cX_-$ corresponds to an outer brane $\cL_+$ in $\cX_+$. We illustrate in Figure \ref{fig:ExKP2} the construction of corresponding $4$-orbifolds $\tcX_\pm$ when the branes $\cL_\pm$ have framing $0$, which agrees with the construction in \cite{LM01}. We note that the triangulations on the ``$+$''-side are refinements of those on the ``$-$''-side. This holds for any framing and corresponds to that that $\cX_+$ (resp. $\tcX_+$) is a crepant (partial) resolution of $\cX_-$ (resp. $\tcX_-$).

In general, starting with a pair $\cX_\pm$ differing by a toric crepant transformation and corresponding framed branes in them, the corresponding 4-orbifolds also differ by a toric crepant transformation. See also the subsequent example.


\begin{figure}[h]
\begin{center}
    \begin{tikzpicture}[scale=1.3]
        \coordinate (01) at (-3.6, -0.4);
        \coordinate (02) at (-5.7, 1);        
        \coordinate (03) at (-5.7, 0.3);
        \coordinate (04) at (-5, 0.3);
                
        \node at (01) {$\bullet$};
        \node at (02) {$\bullet$};
        \node at (03) {$\bullet$};
        \node at (04) {$\bullet$};

        \node at (-3.2, 0) {$(3, -1, 1)$};
        \node at (-5.7, 1.4) {$(0, 1, 1)$};
        \node at (-5.7, -0.1) {$(0, 0, 1)$};

        \draw[ultra thick] (-5.8, 0.65) -- (-5.6, 0.65);
        \node[left] at (-5.8, 0.65) {$(\cL_+, 0)$};

        \draw (01) -- (02) -- (03) -- (01) -- (04) -- (02);
        \draw (04) -- (03);

        \node at (-4.7, -1) {$\cX_+ = \Tot(K_{\bP^2})$};


        \coordinate (1) at (2.2, -0.3);
        \coordinate (2) at (0.8, 0.7);        
        \coordinate (3) at (0, 0);
        \coordinate (4) at (0.8, 0.1);
        \coordinate (5) at (-0.8, 0.9);        
        \coordinate (6) at (0, 1);

        \node at (1) {$\bullet$};
        \node at (2) {$\bullet$};
        \node at (3) {$\bullet$};
        \node at (4) {$\bullet$};
        \node at (5) {$\bullet$};
        \node at (6) {$\bullet$};

        \node at (2.6, 0.1) {$(3, -1, 1, 0)$};
        \node at (1.5, 0.9) {$(0, 1, 1, 0)$};
        \node at (0, -0.4) {$(0, 0, 1, 0)$};
        \node at (-1.5, 0.6) {$(-1, 0, 1, 1)$};
        \node at (0, 1.4) {$(0, 0, 1, 1)$};

        \draw (1) -- (3) -- (5) -- (6) -- (2) -- (1) -- (5);
        \draw (1) -- (6);
        \draw[dashed] (2) -- (4) -- (6) -- (3) -- (2) -- (5);
        \draw[dashed] (1) -- (4) -- (3);

        \node at (0.5, -1) {$\tcX_+$};
    \end{tikzpicture}

    \begin{tikzpicture}[scale=1.3]
        \coordinate (01) at (-3.6, -0.4);
        \coordinate (02) at (-5.7, 1);        
        \coordinate (03) at (-5.7, 0.3);
        \coordinate (04) at (-5, 0.3);
                
        \node at (01) {$\bullet$};
        \node at (02) {$\bullet$};
        \node at (03) {$\bullet$};
        \node at (04) {$\bullet$};

        \node at (-3.2, 0) {$(3, -1, 1)$};
        \node at (-5.7, 1.4) {$(0, 1, 1)$};
        \node at (-5.7, -0.1) {$(0, 0, 1)$};

        \draw[ultra thick] (-5.8, 0.65) -- (-5.6, 0.65);
        \node[left] at (-5.8, 0.65) {$(\cL_-, 0)$};

        \draw (01) -- (02) -- (03) -- (01);

        \node at (-4.7, -1) {$\cX_- = [\bC^3/\bZ_3]$};


        \coordinate (1) at (2.2, -0.3);
        \coordinate (2) at (0.8, 0.7);        
        \coordinate (3) at (0, 0);
        \coordinate (4) at (0.8, 0.1);
        \coordinate (5) at (-0.8, 0.9);        
        \coordinate (6) at (0, 1);

        \node at (1) {$\bullet$};
        \node at (2) {$\bullet$};
        \node at (3) {$\bullet$};
        \node at (4) {$\bullet$};
        \node at (5) {$\bullet$};
        \node at (6) {$\bullet$};

        \node at (2.6, 0.1) {$(3, -1, 1, 0)$};
        \node at (1.5, 0.9) {$(0, 1, 1, 0)$};
        \node at (0, -0.4) {$(0, 0, 1, 0)$};
        \node at (-1.5, 0.6) {$(-1, 0, 1, 1)$};
        \node at (0, 1.4) {$(0, 0, 1, 1)$};

        \draw (1) -- (3) -- (5) -- (6) -- (2) -- (1) -- (5);
        \draw (1) -- (6);
        \draw[dashed] (6) -- (3) -- (2) -- (5);

        \node at (0.5, -1) {$\tcX_-$};
    \end{tikzpicture}
\end{center}    

    \caption{Construction of the 4-folds $\tcX_\pm$ corresponding to $\cX_+ = \Tot(K_{\bP^2})$ and $\cX_- = [\bC^3/\bZ_3]$ and corresponding framed branes.}
    \label{fig:ExKP2}
\end{figure}

\subsubsection{$A_1$-singularities}
Consider $\cX_- = [\bC^2/\bZ_2] \times \bC$ and its crepant resolution $\cX_+$. Let $\cL_-$ be the ineffective outer brane in $\cX_-$, with generic stabilizer group $\mu_2$, which correspond to two effective outer branes $\cL_+^1, \cL_+^2$ in $\cX_+$. We illustrate in Figure \ref{fig:ExA1} the construction of corresponding $4$-orbifolds $\tcX_\pm$ when $\cL_-$, $\cL_+^1, \cL_+^2$ have framings $-1,-1,0$ respectively, noting that the triples $(\cX_+, \cL_+^1, -1)$, $(\cX_+, \cL_+^2, 0)$ give rise to the same $\tcX_+$. We note that for any framing, $\tcX_+$ is a crepant (partial) resolution of $\tcX_-$.

\begin{figure}[h]
\begin{center}
    \begin{tikzpicture}[scale=1.4]
        \coordinate (01) at (-4.3, -0.2);
        \coordinate (02) at (-5, 1.2);        
        \coordinate (03) at (-5, -0.2);
        \coordinate (04) at (-5, 0.5);
                
        \node at (01) {$\bullet$};
        \node at (02) {$\bullet$};
        \node at (03) {$\bullet$};
        \node at (04) {$\bullet$};

        \node at (-3.7, -0.4) {$(1, 0, 1)$};
        \node at (-5.6, 1.4) {$(0, 2, 1)$};
        \node at (-5.6, -0.4) {$(0, 0, 1)$};

        \draw[ultra thick] (-5.1, 0.85) -- (-4.9, 0.85);
        \node[left] at (-5.1, 0.85) {$(\cL_+^2, 0)$};
        \draw[ultra thick] (-5.1, 0.15) -- (-4.9, 0.15);
        \node[left] at (-5.1, 0.15) {$(\cL_+^1, -1)$};

        \draw (01) -- (02) -- (03) -- (01) -- (04);

        \node at (-4.7, -1) {$\cX_+$};


        \coordinate (1) at (1, 0);
        \coordinate (2) at (0.6, 1);        
        \coordinate (3) at (0, 0);
        \coordinate (4) at (0.3, 0.5);
        \coordinate (5) at (-0.6, 0.4);        
        \coordinate (6) at (0, 1);

        \node at (1) {$\bullet$};
        \node at (2) {$\bullet$};
        \node at (3) {$\bullet$};
        \node at (4) {$\bullet$};
        \node at (5) {$\bullet$};
        \node at (6) {$\bullet$};

        \node at (1.8, 0) {$(1, 0, 1, 0)$};
        \node at (1.3, 1.1) {$(0, 2, 1, 0)$};
        \node at (0, -0.4) {$(0, 0, 1, 0)$};
        \node at (-1.5, 0.5) {$(-1, -1, 1, 1)$};
        \node at (-0.2, 1.3) {$(0, 0, 1, 1)$};

        \draw (1) -- (3) -- (5) -- (6) -- (2) -- (1) -- (5);
        \draw (1) -- (6);
        \draw[dashed] (5) -- (2) -- (4) -- (3) -- (6);
        \draw[dashed] (1) -- (4) -- (6);
        \draw[dashed] (4) -- (5);

        \node at (0.5, -1) {$\tcX_+$};
    \end{tikzpicture}

    \begin{tikzpicture}[scale=1.4]
        \coordinate (01) at (-4.3, -0.2);
        \coordinate (02) at (-5, 1.2);        
        \coordinate (03) at (-5, -0.2);
        \coordinate (04) at (-5, 0.5);
                
        \node at (01) {$\bullet$};
        \node at (02) {$\bullet$};
        \node at (03) {$\bullet$};
        \node at (04) {$\bullet$};

        \node at (-3.7, -0.4) {$(1, 0, 1)$};
        \node at (-5.6, 1.4) {$(0, 2, 1)$};
        \node at (-5.6, -0.4) {$(0, 0, 1)$};

        \draw[ultra thick] (-5.1, 0.5) -- (-4.9, 0.5);
        \node[left] at (-5.1, 0.55) {$(\cL_-, -1)$};

        \draw (01) -- (02) -- (03) -- (01);

        \node at (-4.7, -1) {$\cX_- = [\bC^2/\bZ_2] \times \bC$};


        \coordinate (1) at (1, 0);
        \coordinate (2) at (0.6, 1);        
        \coordinate (3) at (0, 0);
        \coordinate (4) at (0.3, 0.5);
        \coordinate (5) at (-0.6, 0.4);        
        \coordinate (6) at (0, 1);

        \node at (1) {$\bullet$};
        \node at (2) {$\bullet$};
        \node at (3) {$\bullet$};
        \node at (4) {$\bullet$};
        \node at (5) {$\bullet$};
        \node at (6) {$\bullet$};

        \node at (1.8, 0) {$(1, 0, 1, 0)$};
        \node at (1.3, 1.1) {$(0, 2, 1, 0)$};
        \node at (0, -0.4) {$(0, 0, 1, 0)$};
        \node at (-1.5, 0.5) {$(-1, -1, 1, 1)$};
        \node at (-0.2, 1.3) {$(0, 0, 1, 1)$};

        \draw (1) -- (3) -- (5) -- (6) -- (2) -- (1) -- (5);
        \draw (1) -- (6);
        \draw[dashed] (5) -- (2) -- (4) -- (3) -- (6);

        \node at (0.5, -1) {$\tcX_-$};
    \end{tikzpicture}
\end{center}

    \caption{Construction of the 4-folds $\tcX_\pm$ corresponding to $\cX_- = [\bC^2/\bZ_2] \times \bC$ and its crepant resolution $\cX_+$.}
    \label{fig:ExA1}
\end{figure}


\section{Orbifold Gromov-Witten invariants and localization}\label{sect:GW}
In this section, we give the definitions of the disk invariants of $(\cX, \cL, f)$ and the closed Gromov-Witten invariants of $\tcX$. Since $\cX$ and $\tcX$ are non-compact, these invariants are defined and computed using torus localization. We summarize the localization computations, following Fang-Liu-Tseng \cite{FLT12} on the open side and  \cite{Liu13} on the closed side, in preparation of proving the open/closed correspondence.

\subsection{Moduli of twisted stable maps to toric orbifolds}\label{sect:Moduli}
We start with some preliminaries of Gromov-Witten theory for toric orbifolds and localization computations. Orbifold Gromov-Witten theory is developed on the symplectic side by Chen-Ruan \cite{CR02} and on the algebraic side by Abramovich-Graber-Vistoli \cite{AGV02, AGV08}. Here, we review the moduli spaces of twisted stable maps to toric orbifolds, induced torus actions on them, and the description of the torus-fixed loci in terms of decorated graphs. We provide additional details, specifically on Hurwitz-Hodge integrals and twisted covers of proper torus-invariant lines, in Appendix \ref{sect:GWPrelim}. Our exposition follows \cite{Liu13}. We restrict our attention to genus zero.

In this section and the next, as in Section \ref{sect:PrelimToric}, let $\cZ$ be an $r$-dimensional toric orbifold specified by an extended stacky fan $\bXi = (\bZ^r, \Xi, \alpha)$, $Z$ be the coarse moduli space of $\cZ$, and $(\bC^*)^r$ be the $r$-dimensional Deligne-Mumford torus of $\cZ$.

Let $n \in \bZ_{\ge 0}$. A genus-zero, $n$-pointed \emph{twisted curve} is a connected, proper, $1$-dimensional Deligne-Mumford stack $\cC$ together with $n$ disjoint closed substacks $\fx_1, \dots, \fx_n$, such that:
\begin{itemize}
    \item $\cC$ is \'etale locally a nodal curve.

    \item Formally locally near a node, $\cC$ is isomorphic to
    $$
        [\Spec(\bC[x,y]/(xy))/\mu_r]
    $$
    for some $r \in \bZ_{>0}$, where $\zeta \in \mu_r$ acts by $\zeta \cdot (x,y) = (\zeta x, \zeta^{-1}y)$.

    \item Each $\fx_i$ is contained in the smooth locus of $\cC$.

    \item Each $\fx_i$ is an \'etale gerbe over $\Spec(\bC)$ with a section.

    \item $\cC$ is a scheme outside $\fx_1, \dots, \fx_n$ and the singular locus.

    \item The coarse moduli space $C$ is a nodal curve of arithmetic genus zero.
\end{itemize}
If $\pi: \cC \to C$ is the projection to the coarse moduli space and $x_i := \pi(\fx_i)$, then $x_1, \dots, x_n$ are distinct smooth points of $C$ and $(C, x_1, \dots, x_n)$ is a genus-zero, $n$-pointed prestable curve.

Let $\beta \in H_2(Z; \bZ)$ be an effective class. A genus-zero, $n$-pointed, degree-$\beta$ \emph{twisted stable map} to $\cZ$ is a representable morphism $u: (\cC, \fx_1, \dots, \fx_n) \to \cZ$, where $(\cC, \fx_1, \dots, \fx_n)$ is a genus-zero, $n$-pointed twisted curve, such that the induced map $\bar{u}:(C, x_1, \dots, x_n) \to Z$ between the coarse moduli spaces is a genus-zero, $n$-pointed, degree-$\beta$ stable map to $Z$.

Let $\Mbar_{0,n}(\cZ, \beta)$ be the moduli space of genus-zero, $n$-pointed, degree-$\beta$ twisted stable maps to $\cZ$, which is a Deligne-Mumford stack. For $i = 1, \dots, n$, there is an evaluation map $\ev_i: \Mbar_{0,n}(\cZ, \beta) \to \cI \cZ$ associated to the $i$-th twisted point $\fx_i$. Given any $\vj = (j_1, \dots, j_n) \in \Box(\cZ)^n$, we define
$$
    \Mbar_{0, \vj}(\cZ, \beta):= \bigcap_{i = 1}^n \ev_i^{-1}(\cZ_{j_i}).
$$
Then $\Mbar_{0, \vj}(\cZ, \beta)$ is a union of connected components of $\Mbar_{0,n}(\cZ, \beta)$, and admits a perfect obstruction theory of virtual dimension
$$
    \int_{\beta} c_1(T\cZ) + r - 3 + n - \sum_{i = 1}^n \age(j_i).
$$
Moreover, we have
$$
    \Mbar_{0,n}(\cZ, \beta) = \bigsqcup_{\vj \in \Box(\cZ)^n} \Mbar_{0, \vj}(\cZ, \beta).
$$

Let $\epsilon: \Mbar_{0,n}(\cZ, \beta) \to \Mbar_{0,n}(Z,\beta)$ be the natural forgetful map. For $i = 1, \dots, n$, define the \emph{descendant class}
$$
    \bar{\psi}_i := \epsilon^*\psi_i \in A^1(\Mbar_{0,n}(\cZ, \beta)).
$$
These classes pull back to descendant classes on $\Mbar_{0,\vj}(\cZ, \beta)$ for each $\vj \in \Box(\cZ)^n$.

The $(\bC^*)^r$-action on $\cZ$ induces a $(\bC^*)^r$-action on $\Mbar_{0,\vj}(\cZ, \beta)$ for any $\vj, \beta$. This makes the virtual tangent bundle of $\Mbar_{0,\vj}(\cZ, \beta)$ and the evaluation maps $\ev_1, \dots, \ev_n$ $(\bC^*)^r$-equivariant. The fixed locus $\Mbar_{0,\vj}(\cZ, \beta)^{(\bC^*)^r}$ is a proper, closed substack.

\subsection{Torus-fixed locus and decorated graphs}\label{sect:DecGraphs}
Components of the $(\bC^*)^r$-fixed loci of the moduli spaces of stable maps to $\cZ$ can be described by \emph{decorated graphs}, defined as follows:

\begin{definition}\rm{
Let $n \in \bZ_{\ge 0}$, $\vj = (j_1, \dots, j_n) \in \Box(\cZ)^n$, and $\beta \in H_2(Z; \bZ)$ be an effective curve class. A genus-zero, $\vj$-twisted, degree-$\beta$ \emph{decorated graph} for $\cZ$ is a tuple $\vGa = (\Gamma, \vf, \vd, \vs, \vk)$, where:
\begin{itemize}
    \item $\Gamma$ is a compact, connected, $1$-dimensional CW complex. Let $V(\Gamma)$ denote the vertex set of $\Gamma$, $E(\Gamma)$ denote the edge set of $\Gamma$, and
    $$
        F(\Gamma) := \{(e, v) \in E(\Gamma) \times V(\Gamma): v \in e \}
    $$
    denote the set of flags.

    \item $\vf: V(\Gamma) \sqcup E(\Gamma) \to \Xi(r) \sqcup \Xi(r-1)_c$ is the \emph{label} map that sends each $v \in V(\Gamma)$ to an $r$-cone $\sigma_v \in \Xi(r)$ and each $e \in E(\Gamma)$ to an $(r-1)$-cone $\tau_e \in \Xi(r-1)_c$ such that for each flag $(e,v) \in F(\Gamma)$, $(\tau_e, \sigma_v)$ is a flag in $F(\Xi)$. We denote $G_v:= G_{\sigma_v}$ for each $v \in V(\Gamma)$. 

    \item $\vd$ is the \emph{degree map} that sends each edge $e \to E(\Gamma)$ to an element $\gamma_e \in H_{\tau_e}$, such that $d_e:= \pi_{\tau_e}(\gamma_e)$ is a positive integer (see \eqref{eqn:LineFundGroup}). 

    \item $\vs: \{1, \dots, n\} \to V(\Gamma)$ is the \emph{marking map}, defined if $n >0$.

    \item $\vk$ is the \emph{twisting map} that sends each flag $(e,v) \in F(\Gamma)$ to some $k_{(e,v)} \in G_v$, and each marking $i \in \{1, \dots, n\}$ to some $k_i \in G_{\vs(i)}$.
\end{itemize}
such that the following conditions are satisfied:
\begin{itemize}
    \item The graph $\Gamma = (V(\Gamma), E(\Gamma))$ is a tree:  
    $$ 
    |E(\Gamma)|-|V(\Gamma)|+1 = 0.
    $$

    \item $\displaystyle{   \sum_{e \in E(\Gamma)} d_e[l_{\tau_e}] = \beta. }$

    \item (Compatibility along an edge) For any edge $e \in E(\Gamma)$, if $v, v' \in V(\Gamma)$ are the two incident vertices, then 
        $$
            \pi_{(\tau_e, \sigma_v)}(\gamma_e) = k_{(e,v)}, \quad \pi_{(\tau_e, \sigma_{v'})}(\gamma_e) = k_{(e,v')}
        $$
        (see \eqref{eqn:FlagFundGroup}).

    \item (Compatibility at a vertex) For any vertex $v \in V(\Gamma)$, the equation
        $$
            \prod_{(e,v) \in F(\Gamma)} k_{(e,v)}^{-1} \prod_{i \in \vs^{-1}(v)}k_i = 1
        $$
        holds in $G_v$.

    \item (Compatibility with $\vj$) For each $i = 1, \dots, n$, the pair $(p_{\sigma_{\vs(i)}}, k_i)$ represents a point in the inertia component $\cX_{j_i}$.
\end{itemize}
}\end{definition}

Let $\Gamma_{0, \vj}(\cZ, \beta)$ be the set of all genus-zero, $\vj$-twisted, degree-$\beta$ decorated graphs for $\cZ$. We set up the following additional notations on a decorated graph $\vGa \in \Gamma_{0, \vj}(\cZ, \beta)$:
\begin{itemize}
    \item For each $v \in V(\Gamma)$, let
    $$
        E_v := \{e \in E(\Gamma) : (e,v) \in F(\Gamma) \}, \quad S_v:= \vs^{-1}(v),
    $$
    and $\val(v):= |E_v|$, $n_v := |S_v|$. Let $\vk_v:= (k_{(e,v)}^{-1},k_i) \in G_v^{E_v \cup S_v}$.

    \item Let
    $$
        V^S(\vGa):= \{v \in V(\Gamma): \val(v) + n_v -2 >0 \}
    $$
    be the set of \emph{stable} vertices of $\Gamma$, and
    \begin{align*}
        V^1(\vGa) &:= \{v \in V(\Gamma): \val(v) = 1, n_v = 0 \},\\
        V^{1,1}(\vGa) &:= \{v \in V(\Gamma): \val(v) = n_v = 1 \},\\
        V^2(\vGa) &:= \{v \in V(\Gamma): \val(v) = 2, n_v = 0 \}
    \end{align*}
    be a partition of the \emph{unstable} vertices. 

    \item Let $\Aut(\vGa)$ be the \emph{automorphism group} of $\vGa$, which consists of all automorphisms of $\Gamma$ that make the maps $\vf, \vd, \vs, \vk$ invariant.

    \item For each $(e,v) \in F(\Gamma)$, let $r_{(e,v)}$ be the order of $k_{(e,v)}$ in $G_v$. For each $v \in V^2(\vGa)$, if $E_v = \{e_1, e_2\}$, let $r_v := r_{(e_1,v)} = r_{(e_2, v)}$.

    \item Let
    \begin{equation}\label{eqn:cGamma}
        c_{\vGa}:= \frac{1}{|\Aut(\vGa)| \cdot \prod_{e \in E(\Gamma)}(d_e|G_e|)} \cdot \prod_{(e,v) \in F(\Gamma)} \frac{|G_v|}{r_{(e,v)}}.
    \end{equation}
\end{itemize}

Given a twisted stable map $u: (\cC, \fx_1, \dots, \fx_n) \to \cZ$ that represents a point in the $(\bC^*)^r$-fixed locus $\Mbar_{0,\vj}(\cZ, \beta)^{(\bC^*)^r}$, we can assign a decorated graph $\vGa = (\Gamma, \vf, \vd, \vs, \vk) \in \Gamma_{0, \vj}(\cZ, \beta)$ as follows. Let $\bar{u}: (C,x_1, \dots, x_n) \to Z$ be the induced stable map between coarse moduli spaces.
\begin{itemize}
    \item The image of $\bar{u}$ lies in $Z^1_c \subset Z$. The vertex set $V(\Gamma)$ is in one-to-one correspondence with the set of connected components in $\bar{u}^{-1}(Z^{(\bC^*)^r})$. For $v \in V(\Gamma)$, let $C_v$ denote the component associated to $v$, and $\cC_v$ be the preimage of $C_v$ under the projection $\cC \to C$. Set $\vf(v) = \sigma_v \in \Xi(r)$ such that the image of $\cC_v$ under $u$ is $\fp_{\sigma_v}$, or equivalently, the image of $C_v$ under $\bar{u}$ is $p_{\sigma_v}$.

    \item The edge set $E(\Gamma)$ is in one-to-one correspondence with the set of irreducible components of $C$ that do not map constantly to $Z$ under $\bar{u}$. For $e \in E(\Gamma)$, let $C_e$ denote the component associated to $e$, and $\cC_e$ be the preimage of $C_e$ under the projection $\cC \to C$. Set $\vf(e) = \tau_e \in \Xi(r-1)$ such that the image of $\cC_e$ under $u$ is $\fl_{\tau_e}$, or equivalently, the image of $C_e$ under $\bar{u}$ is $l_{\tau_e}$. 

    \item The flag set $F(\Gamma)$ consists of all pairs $(e,v)$ such that $\cC_e \cap \cC_v \neq \emptyset$. For $(e,v) \in F(\Gamma)$, let $\fn(e,v) := \cC_e \cap \cC_v$. Set $k_{(e,v)} \in G_v$ to be the image of the generator of the generic stablizer group of $\fn(e,v)$ in $\cC_e$ under $u$.

    \item For an edge $e \in E(\Gamma)$ incident to vertices $v, v' \in V(\Gamma)$, we have $\cC_e \cong \cC_{r_{(e,v)}, r_{(e,v')}}$ (see Section \ref{sect:TwistedCovers}). Let $\gamma_e \in H_{\tau_e}$ be the element defined by $u|_{\cC_e}$ and set $d_e= \pi_{\tau_e}(\gamma_e)$. The compatibility conditions $\pi_{(\tau_e, \sigma_v)}(\gamma_e) = k_{(e,v)}$ and $\pi_{(\tau_e, \sigma_v')}(\gamma_e) = k_{(e,v')}$ are satisfied.

    \item For each marking $i = \{1, \dots, n\}$, set $\vs(i) = v \in V(\Gamma)$ such that $\fx_i \subseteq \cC_v$. Then $\fx_i$ is mapped by $u$ to a point $(\fp_{\sigma_v}, k)$ in $\cZ_{j_i}$, where $k \in G_v$ maps to $j_i$ under $G_v \cong \Box(\sigma_v) \to \Box(\cZ)$. We set $k_i = k$. Then for each $v \in V^S(\vGa)$, $u|_{\cC_v}$ represents a point in $\Mbar_{0, \vk_v}(\cB G_v)$.

\end{itemize}
The above assignment gives a decomposition
$$
    \Mbar_{0,\vj}(\cZ, \beta)^{(\bC^*)^r} = \bigsqcup_{\vGa \in \Gamma_{0,\vj}(\cZ, \beta)} \cF_{\vGa}
$$
into connected components, where $\cF_{\vGa}$ denotes the component corresponding to $\vGa \in \Gamma_{0,\vj}(\cZ, \beta)$. Up to a finite morphism, $\cF_{\vGa}$ can be identified with
$$
    \cM_{\vGa} := \prod_{v \in V^S(\vGa)} \Mbar_{0, \vk_v}(\cB G_v),
$$
and in $A_*(\cM_{\vGa})$, we have\footnote{The coefficient below differs from $c_{\vGa}$ (see \eqref{eqn:cGamma}) by factors associated to unstable vertices of $\Gamma$, and such difference is accounted for in \cite[Section 9.3.3]{Liu13} by the integration conventions \eqref{eqn:UnstableIntegral} at unstable vertices.}
\begin{equation}\label{eqn:FGammaFiniteMorphism}
    [\cF_{\vGa}] = \frac{1}{|\Aut(\vGa)| \cdot \prod_{e \in E(\Gamma)}(d_e|G_e|)} \cdot \prod_{v \in V^S(\vGa), e \in E_v} \frac{|G_v|}{r_{(e,v)}} \cdot \prod_{v \in V^2(\vGa)} \frac{|G_v|}{r_v} \cdot [\cM_{\vGa}].
\end{equation}

Finally, we set
$$
    \Gamma_{0, n}(\cZ, \beta) := \bigsqcup_{\vj \in \Box(\cZ)^n} \Gamma_{0, \vj}(\cZ, \beta)
$$
to be the set of all genus-zero, $n$-pointed, degree-$\beta$ decorated graphs for $\cZ$. We have
$$
    \Mbar_{0,n}(\cZ, \beta)^{(\bC^*)^r} = \bigsqcup_{\vGa \in \Gamma_{0,n}(\cZ, \beta)} \cF_{\vGa}.
$$

\begin{remark}\rm{
For our toric Calabi-Yau $3$-orbifold $\cX$, we will use $T'$-equivariant localization on the moduli spaces of stable maps. Note that the $T'$-fixed points and $T'$-invariant lines of $\cX$ are the same as the $T$-fixed points and $T$-invariant lines. Therefore, the $T'$-fixed loci of the moduli spaces can be identified with the $T$-fixed loci and described by decorated graphs in the same way as above. Similarly, for our toric Calabi-Yau $4$-orbifold $\tcX$, we will use $\tT'$-equivariant localization on the moduli spaces of stable maps and describe the $\tT'$-fixed loci by decorated graphs.
}\end{remark}

\subsection{Disk invariants of \texorpdfstring{$(\cX, \cL, f)$}{X}}
In this section, we give the definition of the disk invariants of $(\cX, \cL, f)$, which are \emph{open Gromov-Witten invariants} \cite{CP14,FL13,FLT12,KL01} that encode twisted stable maps from genus-zero domains with a single boundary component. We follow Fang-Liu-Tseng \cite{FLT12} and refer the reader to there for additional details. 

\subsubsection{Twisted open stable maps and their moduli}
Open stable maps to symplectic orbifolds with Lagrangian boundary conditions are defined by Cho-Poddar \cite{CP14}, generalizing the manifold case defined by Katz-Liu \cite{KL01}; see also \cite{FOOO10,Liu02}. The domain of such a map is a prestable bordered orbifold Riemann surface which allows stacky points at interior nodes and interior marked points. For our target $(\cX, \cL)$, we follow the definition of \cite{FLT12}, which is closely related to the topological vertex \cite{LLLZ09,FL13}. As observed by \cite{FLT12}, the definition of \cite{CP14} assumes that the Lagrangian suborbifold is a smooth manifold. In the more general setting where the Lagrangian contains stacky points, as is the case for our $\cL$, in order to obtain compactness of the moduli when the target orbifold and the Lagrangian are both compact, one needs to allow orbifold structures at boundary nodes and boundary marked points of the domain. However, for $(\cX, \cL)$, since the open Gromov-Witten invariants are defined by $T'$-equivariant localization and $\cL$ does not contain any $T'$-fixed points, there is no need to allow orbifold structures on the boundary of the domain. 

Let $n \in \bZ_{\ge 0}$, $\beta' \in H_2(X,L; \bZ)$, and $(d, \lambda) \in H_1(\cL; \bZ) \cong \bZ \times G_{\tau_0}$. Let
$$
    \Mbar_{n}(\cX, \cL \mid \beta', (d, \lambda)) = \Mbar_{(0,1),n}(\cX, \cL \mid \beta', (d, \lambda))
$$
be the moduli space of maps
$$
    u: ((\cC, \fx_1, \dots, \fx_n), \partial \cC) \to (\cX, \cL)
$$
where
\begin{itemize}
    \item $(\cC, \fx_1, \dots, \fx_n)$ is a prestable bordered orbifold Riemann surface of topological type $(0,1)$ with $n$ interior marked points $\fx_1, \dots, \fx_n$. Here, if $\pi: \cC \to C$ is the projection to the coarse moduli space and $x_i := \pi(\fx_i)$, then $(C, x_1, \dots, x_n)$ is a prestable bordered Riemann surface of topological type $(0,1)$ with $n$ interior marked points. The topological type of $(0,1)$ means that topologically, $C$ is a nodal Riemann surface of arithmetic genus $0$ with a single open disk removed. In particular, $\partial \cC$ is connected and topologically a circle, and contains no orbifold points.

    \item Let $\nu: \hcC \to \cC$ be the normalization map, so that $\hcC$ is a possibly disconnected bordered orbifold Riemann surface with no nodes. Then the map $\nu \circ u: \hcC \to \cX$ is holomorphic.

    \item The automorphism group of $u$ is finite.

    \item Let $\bar{u}: (C, \partial C) \to (X,L)$ be the induced map between coarse moduli spaces. Then $\bar{u}_*[C] = \beta' \in H_2(X,L; \bZ)$.

    \item $u_*[\partial \cC] = (d, \lambda) \in H_1(\cL; \bZ)$, where $d \in \bZ$ is the winding number and $\lambda \in G_{\tau_0}$ is the monodromy.
\end{itemize}
$\Mbar_{n}(\cX, \cL \mid \beta', (d, \lambda))$ is a possibly singular stack with corners, equipped with a virtual tangent bundle which is a virtual real vector bundle. For $i = 1, \dots, n$, there is an evaluation map $\ev_i: \Mbar_{n}(\cX, \cL \mid \beta', (d, \lambda)) \to \cI \cX$ associated to the $i$-th interior marked point $\fx_i$. Given any $\vj = (j_1, \dots, j_n) \in \Box(\cX)^n$, we define
$$
    \Mbar_{\vj}(\cX, \cL \mid \beta', (d, \lambda)) := \bigcap_{i=1}^n \ev_i^{-1}(\cX_{j_i}).
$$
Then $\Mbar_{\vj}(\cX, \cL \mid \beta', (d, \lambda))$ is a union of connected components of $\Mbar_{n}(\cX, \cL \mid \beta', (d, \lambda))$, and the (real) rank of the virtual tangent bundle over $\Mbar_{\vj}(\cX, \cL \mid \beta', (d, \lambda))$ is
$$
    2\sum_{i = 1}^n (1-\age(j_i)).
$$

\subsubsection{Definition of disk invariants}
The action of the compact Calabi-Yau $2$-torus $T_{\bR}'$ on $(\cX, \cL)$ induces a $T_{\bR}'$-action on $\Mbar_{n}(\cX, \cL \mid \beta', (d, \lambda))$ and makes the virtual tangent bundle of $\Mbar_{n}(\cX, \cL \mid \beta', (d, \lambda))$ and the evaluation maps $\ev_1, \dots, \ev_n$ $T_{\bR}'$-equivariant. Let $\cF := \Mbar_{n}(\cX, \cL \mid \beta', (d, \lambda))^{T_{\bR}'}$ be the $T_{\bR}'$-fixed locus. Then each connected component of $\cF$ is a compact orbifold, on which the virtual tangent bundle agrees with the tangent bundle. We have
$$
    [\cF]^\vir = [\cF].
$$

\begin{definition}\rm{
Given $\gamma_1, \dots, \gamma_n \in H^*_{\CR,T'_{\bR}}(\cX; \bQ) = H^*_{\CR,T'}(\cX; \bQ)$, we define
$$
    \inner{\gamma_1, \dots, \gamma_n}_{\beta', (d, \lambda)}^{\cX, \cL} := \int_{[\cF]^\vir} \frac{\iota^* \left(\prod_{i = 1}^n \ev_i^*(\gamma_i)\right)}{e_{T_{\bR}'}(N^\vir)} \in \cQ_{T'},
$$
where $\iota: \cF \to \Mbar_{n}(\cX, \cL \mid \beta', (d, \lambda))$ is the inclusion and $N^\vir$ is the virtual normal bundle of $\cF$ in $\Mbar_{n}(\cX, \cL \mid \beta', (d, \lambda))$.
}
\end{definition}

Suppose that for each $i = 1, \dots, n$, we have $\gamma_i \in H^{2a_i}_{T'_{\bR}}(\cX_{j_i}; \bQ)$ (viewed as a $\bQ$-vector subspace of $H^{2(a_i + \age(j_i))}_{\CR, T'_{\bR}}(\cX; \bQ)$) for some $j_i \in \Box(\cX)$ and $a_i \in \bZ_{\ge 0}$. Set $\vj = (j_1, \dots, j_n)$. Then only the connected components of $\cF$ contained in $\Mbar_{\vj}(\cX, \cL \mid \beta', (d, \lambda))$ contribute to $\inner{\gamma_1, \dots, \gamma_n}_{\beta', (d, \lambda)}^{\cX, \cL}$. Therefore, $\inner{\gamma_1, \dots, \gamma_n}_{\beta', (d, \lambda)}^{\cX, \cL}$ is a homogenous rational function in $\su_1, \su_2$ of degree
$$
    \sum_{i = 1}^n (\age(j_i)-1+a_i).
$$

\begin{definition}\rm{
Let $\gamma_1, \dots, \gamma_n \in H^2_{\CR}(\cX; \bQ)$ and choose equivariant lifts in $H^2_{\CR,T'_{\bR}}(\cX; \bQ) = H^2_{\CR,T'}(\cX; \bQ)$ as in Convention \ref{conv:LiftOuter}. We define the \emph{disk} invariant
$$
    \inner{ \gamma_1, \dots, \gamma_n }^{\cX,(\cL,f)}_{\beta', (d,\lambda)} := \inner{\gamma_1, \dots, \gamma_n}_{\beta', (d, \lambda)}^{\cX, \cL} \big|_{\su_2 - f\su_1 = 0} \in \cQ_{T_f}.
$$
}
\end{definition}

Observe that $\inner{ \gamma_1, \dots, \gamma_n }^{\cX, (\cL,f)}_{\beta', (d,\lambda)}$ is homogeneous of degree $0$ and thus in $\bQ$. We will confirm, as a consequence of the numerical open/closed correspondence, that $\inner{\gamma_1, \dots, \gamma_n}_{\beta', (d, \lambda)}^{\cX, \cL}$ has no pole along $\su_2 - f\su_1 = 0$ and thus $\inner{ \gamma_1, \dots, \gamma_n }^{\cX,(\cL,f)}_{\beta', (d,\lambda)}$ is defined. See Remark \ref{rem:DiskWellDefined}.

We note that $\inner{\gamma_1, \dots, \gamma_n}_{\beta', (d, \lambda)}^{\cX, \cL}$ and $\inner{\gamma_1, \dots, \gamma_n}^{\cX, (\cL,f)}_{\beta', (d,\lambda)}$ are defined up to a sign depending on a choice of orientation on $\Mbar_{n}(\cX, \cL \mid \beta', (d, \lambda))$. Our choice will be specified by the computation result given in Proposition \ref{prop:DiskLocalResultOuter}. See also Remark \ref{rem:DiskSignOuter}.

\subsection{Localization computations of disk invariants}\label{sect:DiskLocalOuter}
In this section, we summarize the localization computations of the disk invariants of $(\cX, \cL, f)$ by Fang-Liu-Tseng \cite{FLT12}. We note that although \cite{FLT12} considered the case $f \in \bZ$, their results directly generalize to the case $f \in \bQ$; see also \cite{FZ19}.

\subsubsection{Tangent $T'$-weights}\label{sect:XTanWtsOuter}
Given any flag $(\tau, \sigma) \in F(\Sigma)$, we define
$$
    \bw(\tau, \sigma):= e_{T'}(T_{\fp_\sigma}\fl_\tau) \in H^2_{T'}(\pt; \bQ) = \bQ\su_1 \oplus \bQ \su_2,
$$
and $w(\tau, \sigma) \in \bQ$ such that
$$
    \bw(\tau, \sigma)|_{\su_2 - f\su_1=0} = w(\tau, \sigma) \su_1.
$$
In particular, for the three flags associated to the cone $\sigma_0$, we have
\begin{equation}\label{eqn:Sigma0WtsOuter}
    \begin{aligned}
        &\bw_0 := \bw(\tau_0, \sigma_0) = \frac{1}{\fr}\su_1, && w_0:= w(\tau_0, \sigma_0) = \frac{1}{\fr},\\
        &\bw_2 := \bw(\tau_2, \sigma_0) = \frac{\fs}{\fr\fm}\su_1 + \frac{1}{\fm}\su_2, && w_2:= w(\tau_2, \sigma_0) = \frac{\fs + \fr f}{\fr\fm},\\
        &\bw_3 := \bw(\tau_3, \sigma_0) = -\frac{\fm + \fs}{\fr\fm}\su_1 - \frac{1}{\fm}\su_2, && w_3:= w(\tau_3, \sigma_0) = -\frac{\fm + \fs + \fr f}{\fr\fm}.
    \end{aligned}
\end{equation}

\subsubsection{The disk factor}\label{sect:DiskFactorOuter}
Let $(d, \lambda) \in H_1(\cL; \bZ) \cong \bZ \times G_{\tau_0}$ such that $d>0$. Define
$$
    h(d, \lambda):= \pi_{(\tau_0, \sigma_0)}(d, \lambda) \in G_{\sigma_0}
$$
(see \eqref{eqn:FlagFundGroup}). As a key ingredient in the localization computation of the disk invariants, the \emph{disk factor} \cite{BC11,Ross14} is defined as 
$$
    D_{d, \lambda} := \inner{\one_{h(d,\lambda)}}_{d[B], (d, \lambda)}^{\cX, \cL} \in \cQ_{T'},
$$
which is homogeneous of degree $\age(h(d,\lambda))-1$. Here, the disk $B$ is defined in Section \ref{sect:HomologyOuter}. 
Let $\epsilon_2, \epsilon_3 \in \bQ \cap [0,1)$ such that $h(d,\lambda)$ acts on $T_{\fp_{\sigma_0}}\fl_{\tau_2}$, $T_{\fp_{\sigma_0}}\fl_{\tau_3}$ by multiplication by $e^{2\pi\sqrt{-1}\epsilon_2}$, $e^{2\pi\sqrt{-1}\epsilon_3}$ respectively. Then
$$
    \inner{dw_0} + \epsilon_2 + \epsilon_3 = \age(h(d,\lambda)).
$$

By choosing an orientation on the moduli space $\Mbar_1(\cX, \cL \mid d[B], (d, \lambda))$, we have the following formula of \cite[Section 3.11]{FLT12}, which is based on \cite{Ross14}:
\begin{equation}\label{eqn:DiskFactor}
\begin{aligned}
    D_{d, \lambda} &= (-1)^{\floor{dw_3-\epsilon_3}+ \ceil{\frac{d}{\fa}}}\left(\frac{\fr\bw_0}{d}\right)^{\age(h(d,\lambda))-1} \cdot \frac{1}{d\fm \cdot \floor{dw_0}!} \cdot \prod_{a = 1}^{\floor{dw_0}+\age(h(d,\lambda))-1} \left(\frac{d\bw_2}{\fr\bw_0} + a - \epsilon_2 \right)\\
        & = (-1)^{\floor{dw_3-\epsilon_3}+ \ceil{\frac{d}{\fa}}}\left(\frac{\su_1}{d}\right)^{\age(h(d,\lambda))-1} \cdot \frac{1}{d\fm \cdot \floor{dw_0}!} \cdot \prod_{a = 1}^{\floor{dw_0}+\age(h(d,\lambda))-1} \left(\frac{d\bw_2}{\su_1} + a - \epsilon_2 \right).
\end{aligned}
\end{equation}

\begin{remark}\label{rem:DiskSignOuter} \rm{
We note that when $f \in \bZ$, i.e. $\fa=1$, the sign convention of formula \eqref{eqn:DiskFactor} above differs from that in \cite{FLT12}, yet agrees with that in \cite{FL13} (and \cite{LY21}) in the smooth case. Our choice of sign (and orientation on the moduli space of open stable maps) ensures that the numerical open/closed correspondence holds without a sign difference. See Theorem \ref{thm:NumericalOuter}.
}\end{remark}

\subsubsection{Localization computations}
Given any $n \in \bZ_{\ge 0}$ and effective class $\beta \in H_2(X; \bZ)$, the $T'$-action on $\cX$ induces a $T'$-action on the moduli space $\Mbar_{0, n+1}(\cX, \beta)$ of stable maps to $\cX$. This makes the virtual tangent bundle and the evaluation maps $\ev_1, \dots, \ev_{n+1}$ $T'$-equivariant. \cite{FLT12} directly relates the $T_{\bR}'$-fixed loci of the moduli spaces of open stable maps to $(\cX, \cL)$ and the $T'$-fixed loci of the moduli spaces of stable maps to $\cX$ and compares their tangent-obstruction theories. As a consequence, via the disk factor, \cite[Proposition 3.3]{FLT12} relates the disk invariants of $(\cX, \cL, f)$ to the Gromov-Witten invariants of $\cX$, as follows:
\begin{theorem}[\cite{FLT12}]\label{thm:FLTDiskLocalOuter}
Let $n \in \bZ_{\ge 0}$, $\beta \in H_2(X; \bZ)$ be an effective class, $(d, \lambda) \in H_1(\cL; \bZ) \cong \bZ \times G_{\tau_0}$ such that $d>0$, and $\gamma_1, \dots, \gamma_n \in H^2_{\CR,T'_{\bR}}(\cX; \bQ) = H^2_{\CR,T'}(\cX; \bQ)$. Set
$$
    \beta' = \beta + d[B].
$$
Then
\begin{equation}\label{eqn:FLTDiskLocalOuter}
\inner{ \gamma_1, \dots, \gamma_n }^{\cX, (\cL,f)}_{\beta', (d,\lambda)} = \fr \fm D_{d, \lambda} \cdot \int_{[\Mbar_{0, n+1}(\cX, \beta)^{T'}]^\vir} \frac{\iota^* \left( \ev_{n+1}^* (\phi_{\sigma_0, h(d, \lambda)^{-1}}) \cdot \prod_{i = 1}^n \ev_i^* (\gamma_i) \right)}{e_{T'}(N^\vir) \cdot (\frac{\su_1}{d} - \bar{\psi}_{n+1})} \bigg|_{\su_2 - f\su_1 = 0},
\end{equation}
where $\iota: \Mbar_{0, n+1}(\cX, \beta)^{T'} \to \Mbar_{0, n+1}(\cX, \beta)$ is the inclusion, $N^\vir$ is the virtual normal bundle of $\Mbar_{0, n+1}(\cX, \beta)^{T'}$, and
$$
    \phi_{\sigma_0,  h(d, \lambda)^{-1}} = \iota_{\sigma_0, *}(\one_{h(d, \lambda)^{-1}}) \in H^*_{\CR, T'}(\cX; \bQ)
$$
is the $T'$-equivariant Poincar\'e dual of the point $(\fp_{\sigma_0},  h(d, \lambda)^{-1})$.
\end{theorem}

Recall from Section \ref{sect:DecGraphs} that connected components of $\Mbar_{0, n+1}(\cX, \beta)^{T'}$ are indexed by decorated graphs in $\Gamma_{0,n+1}(\cX, \beta)$. Using the study of the tangent-obstruction theory by \cite[Theorem 137]{Liu13} (which is also used in the proof of \cite[Proposition 3.3]{FLT12}), we can rewrite \eqref{eqn:FLTDiskLocalOuter} in terms of contributions from decorated graphs as
\begin{equation}\label{eqn:LiuDiskLocalOuter}
\begin{aligned}
 \inner{\gamma_1,\ldots, \gamma_n}_{\beta', (d, \lambda)}^{\cX, (\cL, f)} 
= & \fr \fm D_{d, \lambda} \cdot \sum_{\vGa \in \Gamma_{0,n+1}(\cX, \beta)} c_{\vGa} \cdot \prod_{e \in E(\Gamma)} \bh(e) \cdot \prod_{(e,v) \in F(\Gamma)} \bh(e,v) \cdot \prod_{v \in V(\Gamma)} \left( \prod_{i \in S_v} \iota_{\sigma_v}^* (\gamma_i) \right)\\
   & \cdot \prod_{v \in V(\Gamma)} \int_{\Mbar_{0, \vk_v}(\cB G_v)} \frac{\bh(v)}{\left(\frac{\su_1}{d} - \bar{\psi}_{n+1}\right)^{\delta_{v, n+1}} \cdot \prod_{e \in E_v} \left(\bw_{(e,v)} - \frac{\bar{\psi}_{(e,v)}}{r_{(e,v)}}\right)} \bigg|_{\su_2 - f\su_1 = 0},
\end{aligned}
\end{equation}
where:
\begin{itemize}
    \item For each $\vGa \in \Gamma_{0,n+1}(\cX, \beta)$, the coefficient $c_{\vGa}$ is defined in \eqref{eqn:cGamma}. To give the definitions of the other quantities in \eqref{eqn:LiuDiskLocalOuter}, we pick a stable map $u: (\cC, \fx_1, \dots, \fx_{n+1}) \to \cX$ whose associated decorated graph is $\vGa$ (see Section \ref{sect:DecGraphs}). The definitions do not depend on the choice of $u$. 

    \item For each $e \in E(\Gamma)$, define
        $$
            \bh(e):= \frac{e_{T'}(H^1(\cC_e,(u|_{\cC_e})^*T\cX)^m)}{e_{T'}(H^0(\cC_e,(u|_{\cC_e})^*T\cX)^m)},
        $$
        which is explicitly computed in \cite[Lemma 130]{Liu13} in terms of tangent $T'$-weights. Here and throughout the paper, the superscript ``$m$'' represents the \emph{moving part}: Any complex representation $V$ of a torus $(\bC^*)^r$ decomposes into a direct sum of $1$-dimensional representations and can thus be written as $V = V^f \oplus V^m$, where the \emph{fixed part} $V^f$ is the direct sum of all trivial $1$-dimensional representations and the moving part $V^m$ is the direct sum of all non-trivial ones.

    \item For each $(e,v) \in F(\Gamma)$, define
       $$
           \bh(e,v):=  e_{T'}\left((T_{\fp_{\sigma_v}}\cX)^{k_{(e,v)}}\right) = \prod_{\substack{ (\tau, \sigma_v) \in F(\Sigma)\\ k_{(e,v)} \in G_\tau } } \bw(\tau, \sigma_v).
      $$
        Here, $(T_{\fp_{\sigma_v}}\cX)^{k_{(e,v)}}$ is the maximal subspace of $T_{\fp_{\sigma_v}}\cX$ that is invariant under the action of $k_{(e,v)}$.

    \item We set
        $$
            \gamma_{n+1}:= \phi_{\sigma_0,  h(d, \lambda)^{-1}}
        $$
        for convenience.

    \item For each $v \in V(\Gamma)$, the marked points and corresponding descendant classes of $\Mbar_{0, \vk_v}(\cB G_v)$ are indexed by $E_v \cup S_v$. The integral over $\Mbar_{0, \vk_v}(\cB G_v)$ is a Hurwitz-Hodge integral (see Section \ref{sect:HHIntegrals}), and we adopt the integration convention \eqref{eqn:UnstableIntegral} for unstable vertices.

    \item For each stable vertex $v \in V^S(\vGa)$, define
        $$
                    \bh(v):= \frac{ e_{T'}\left( H^1(\cC_v,(u|_{\cC_v})^*T\cX)^m\right)}{e_{T'}\left( H^0(\cC_v,(u|_{\cC_v})^*T\cX)^m \right)},
        $$
        which is explicitly computed in \cite[Lemma 126]{Liu13} in terms of Hurwitz-Hodge classes and tangent $T'$-weights. For each unstable vertex $v \not \in V^S(\vGa)$, define
        $$
            \bh(v):= \begin{cases}
                \bh(e,v)^{-1} & \text{if } v \in V^1(\vGa) \cup V^{1,1}(\vGa), E_v = \{e\},\\
                \bh(e_1,v)^{-1} = \bh(e_2,v)^{-1} & \text{if } v \in V^2(\vGa), E_v = \{e_1, e_2\}.
            \end{cases}
        $$

    \item $\delta_{v, n+1}$ is the indicator function
        $$
            \delta_{v, n+1} := \begin{cases}
                1 & \text{if } n+1 \in S_v\\
                0 & \text{otherwise.}
            \end{cases}
        $$

    \item For each $(e,v) \in F(\Gamma)$, define
        $$
            \bw_{(e,v)}:= e_{T'}(T_{\fn(e,v)}\cC_e) = \frac{\fr(\tau_e, \sigma_v)\bw(\tau_e, \sigma_v)}{r_{(e,v)}d_e},
        $$
        where the $T'$-action on $T_{\fn(e,v)}\cC_e$ is induced from that on $T_{\fp_{\sigma_v}}\fl_{\tau_e}$.
\end{itemize}

To simplify \eqref{eqn:LiuDiskLocalOuter}, we note that for any $\sigma \in \Sigma(3)$, $\sigma \neq \sigma_0$,
$$
    \iota_{\sigma}^* (\gamma_{n+1}) = \iota_{\sigma}^*(\phi_{\sigma_0,  h(d, \lambda)^{-1}}) = 0.
$$
Thus only the decorated graphs in the subset
\begin{equation}\label{eqn:DiskGamma00Outer}
    \Gamma_{0,n+1}^{0,(d, \lambda)}(\cX, \beta) := \{\vGa \in \Gamma_{0,n+1}(\cX, \beta): \vf \circ \vs(n+1) = \sigma_0, \vk(n+1) = h(d,\lambda)^{-1} \}
\end{equation}
can contribute. We simplify \eqref{eqn:LiuDiskLocalOuter} as follows:
\begin{proposition}\label{prop:DiskLocalResultOuter}
Let $n, \beta, (d, \lambda), \beta'$, and $\gamma_1, \dots, \gamma_n$ be as in Theorem \ref{thm:FLTDiskLocalOuter}. Then
\begin{equation}\label{eqn:DiskLocalResultOuter}
\begin{aligned}
    \inner{\gamma_1,\ldots, \gamma_n}_{\beta', (d, \lambda)}^{\cX, (\cL, f)} =& \fr \fm D_{d, \lambda} \cdot \sum_{\vGa \in \Gamma_{0,n+1}^{0,(d, \lambda)}(\cX, \beta)} c_{\vGa} \cdot \prod_{e \in E(\Gamma)} \bh(e) \cdot \prod_{(e,v) \in F(\Gamma)} \bh(e,v) \cdot \prod_{v \in V(\Gamma)} \left( \prod_{i \in S_v} \iota_{\sigma_v}^* (\gamma_i) \right)\\
   & \cdot \prod_{v \in V(\Gamma)} \int_{\Mbar_{0, \vk_v}(\cB G_v)} \frac{\bh(v)}{\left(\frac{\su_1}{d} - \bar{\psi}_{n+1}\right)^{\delta_{v, n+1}} \cdot \prod_{e \in E_v} \left(\bw_{(e,v)} - \frac{\bar{\psi}_{(e,v)}}{r_{(e,v)}}\right)} \bigg|_{\su_2 - f\su_1 = 0}.
\end{aligned}
\end{equation}
\end{proposition}

\subsection{Closed invariants of \texorpdfstring{$\tcX$}{X}}\label{sect:ClosedOuter}
In this section, we define closed Gromov-Witten invariants of $\tcX$ and compute them by localization following \cite{Liu13}.


\subsubsection{Definition of closed Gromov-Witten invariants}
Let $n \in \bZ_{\ge 0}$ and $\tbeta \in H_2(\tX; \bZ)$ be an effective class. The action of the Calabi-Yau $3$-torus $\tT'$ on $\tcX$ induces a $\tT'$-action on the moduli space $\Mbar_{0, n}(\tcX, \tbeta)$. This makes the virtual tangent bundle and the evaluation maps $\ev_1, \dots, \ev_{n}$ $\tT'$-equivariant. The $\tT'$-fixed locus $\Mbar_{0, n}(\tcX, \tbeta)^{\tT'}$ can be identified with the $\tT$-fixed locus.

\begin{definition}\rm{
Given $\tgamma_1, \dots, \tgamma_n \in H^*_{\CR, \tT'}(\tcX; \bQ)$, we define
$$
    \inner{\tgamma_1, \dots, \tgamma_n}_{\tbeta}^{\tcX,\tT'}:= \int_{[\Mbar_{0, n}(\tcX, \tbeta)^{\tT'}]^\vir} \frac{\iota^*\left(\prod_{i=1}^{n} \ev_i^*(\tgamma_i)\right)}{e_{\tT'}(N^\vir)} \in \cQ_{\tT'},
$$
where $\iota: \Mbar_{0, n}(\tcX, \tbeta)^{\tT'} \to \Mbar_{0, n}(\tcX, \tbeta)$ is the inclusion and $N^\vir$ is the virtual normal bundle of $\Mbar_{0, n}(\tcX, \tbeta)^{\tT'}$ in $\Mbar_{0, n}(\tcX, \tbeta)$.
}
\end{definition}

Suppose that for each $i = 1, \dots, n$, we have $\tgamma_i \in H^{2a_i}_{\tT'}(\tcX_{j_i}; \bQ)$ (viewed as a $\bQ$-vector subspace of $H^{2(a_i + \age(j_i))}_{\CR, \tT'}(\tcX; \bQ)$) for some $j_i \in \Box(\tcX)$ and $a_i \in \bZ_{\ge 0}$. Then $\inner{\tgamma_1, \dots, \tgamma_n}_{\tbeta}^{\tcX}$ is a homogenous rational function in $\su_1, \su_2, \su_4$ of degree
$$
    -1 + \sum_{i = 1}^n (\age(j_i)-1+a_i).
$$

\begin{definition}\label{def:ClosedOuter}
\rm{
Let $\tgamma_1, \dots, \tgamma_n \in H^2_{\CR}(\tcX; \bQ)$ and choose equivariant lifts in $H^2_{\CR,\tT'}(\tcX; \bQ)$ as in Convention \ref{conv:LiftOuter}. For $\tk \in G_{\tsi_0}$, let $\tgamma_{\tk}$ be the class in $H^4_{\CR, \tT'}(\tcX; \bQ)$ defined by
\begin{equation}\label{eqn:ExtraInsertionOuter}
    \tgamma_{\tk} := \begin{dcases}
        \displaystyle{ \frac{\tcD_2^{\tT'}\tcD_3^{\tT'}\tcD_{R+1}^{\tT'}}{\frac{f}{\fm}\su_1 - \frac{1}{\fm} \su_2 - \su_4}  } & \text{if } \tk = 1,\\  
        \tcD_{R+1}^{\tT'}\one_{\tk^{-1}}  & \text{if } \age(\tk)=1,\\
        \one_{\tk^{-1}}  & \text{if } \age(\tk)=2,
        \end{dcases}
\end{equation}
where recall $\tcD_i^{\tT'} = [\cV(\trho_i)]$ is the $\tT'$-equivariant Poincar\'e dual of the divisor $\cV(\trho_i)$. Then, we define the \emph{closed} Gromov-Witten invariant
$$
  \inner{\tgamma_1,\ldots, \tgamma_n, \tgamma_{\tk}}_{\tbeta}^{\tcX, T_f} := \inner{\tgamma_1, \dots, \tgamma_n, \tgamma_{\tk}}_{\tbeta}^{\tcX,\tT'} \bigg|_{\su_4 = 0, \su_2 - f\su_1 = 0} \in \cQ_{T_f}.
$$
}
\end{definition}

Observe that $\inner{\tgamma_1,\ldots, \tgamma_n, \tgamma_{\tk}}_{\tbeta}^{\tcX,T_f}$ is homogeneous of degree $0$ and thus in $\bQ$.

\begin{lemma}\label{lem:ClosedWellDefined}
Under the setup of Definition \ref{def:ClosedOuter}, $\inner{\tgamma_1, \dots, \tgamma_n, \tgamma_{\tk}}_{\tbeta}^{\tcX,\tT'}$ has no pole along $\su_4 = 0, \su_2 - f\su_1 = 0$. In particular, $\inner{\tgamma_1,\ldots, \tgamma_n, \tgamma_{\tk}}_{\tbeta}^{\tcX, T_f}$ is defined.
\end{lemma}

We defer the proof, which is based on mirror symmetry and our B-model correspondence, to Section \ref{sect:Mirror}.



\subsubsection{Tangent $\tT'$-weights}
Given any flag $(\ttau, \tsi) \in F(\tSi)$, we define
$$
    \tbw(\ttau, \tsi):= e_{\tT'}(T_{\fp_{\tsi}} \fl_{\ttau}) \in H^2_{\tT'}(\pt; \bQ) = \bQ\su_1 \oplus \bQ \su_2 \oplus \bQ \su_4.
$$

For any $(\tau, \sigma) \in F(\Si)$, we have
\begin{equation}\label{eqn:TanWtU4RestrictOuter}
    \tbw(\iota(\tau), \iota(\sigma)) \big|_{\su_4 = 0} = \bw(\tau, \sigma).
\end{equation}
Moreover, for any $\sigma \in \Sigma(3)$, we have
$$
    \tbw(\sigma, \iota(\sigma)) = \su_4.
$$
Specifically for the cone $\iota(\sigma_0)$, we have
\begin{align*}
    &\tbw(\iota(\tau_0), \iota(\sigma_0)) = \frac{1}{\fr}\su_1, && \tbw(\sigma_0, \iota(\sigma_0)) = \su_4,\\
    &\tbw_2:= \tbw(\iota(\tau_2), \iota(\sigma_0)) = \frac{\fs}{\fr\fm}\su_1 + \frac{1}{\fm}\su_2, && \tbw_3:= \tbw(\iota(\tau_3), \iota(\sigma_0)) = -\frac{\fm + \fs}{\fr\fm}\su_1 - \frac{1}{\fm}\su_2 - \su_4.
\end{align*}

Tangent $\tT'$-weights at a fixed point $\fp_{\tsi}$ for $\tsi \in \tSi(4) \setminus \iota(\Sigma(3))$ are given in Section \ref{sect:TanWtsExtraConesOuter}. Specifically for the cone $\tsi_0$, we have
\begin{align*}
    & \tbw(\iota(\tau_0), \tsi_0) = -\frac{1}{\fa}\su_1, && \tbw(\delta_4(\tsi_0), \tsi_0) = \frac{1}{\fa}\su_1 + \su_4,\\
    & \tbw(\delta_2(\tsi_0), \tsi_0) = -\frac{f}{\fm}\su_1 + \frac{1}{\fm} \su_2, && \tbw(\delta_3(\tsi_0), \tsi_0) = \frac{f}{\fm}\su_1 - \frac{1}{\fm} \su_2 - \su_4,
\end{align*}
where $\delta_2(\tsi_0), \delta_3(\tsi_0), \delta_4(\tsi_0)$ are facets of $\tsi_0$ described by
$$
    I_{\delta_2(\tsi_0)}' = \{3, R+1, R+2\}, \quad I_{\delta_3(\tsi_0)}' = \{2, R+1, R+2\}, \quad I_{\delta_4(\tsi_0)}' = \{2, 3, R+1\}.
$$

\subsubsection{Localization computations}
Given any $n \in \bZ_{\ge 0}$ and effective class $\tbeta \in H_2(\tX; \bZ)$, the $\tT'$-action on $\tcX$ induces a $\tT'$-action on the moduli space $\Mbar_{0, n+1}(\tcX, \tbeta)$ of stable maps to $\tcX$. Recall from Section \ref{sect:DecGraphs} that connected components of $\Mbar_{0, n+1}(\tcX, \tbeta)^{\tT'}$ are indexed by decorated graphs in $\Gamma_{0,n+1}(\tcX, \tbeta)$. Using the study of the tangent-obstruction theory by \cite[Theorem 137]{Liu13}, we have the following computation:

\begin{proposition}\label{prop:ClosedLocalResultOuter}
Let $n \in \bZ_{\ge 0}$, $\tbeta \in H_2(\tX; \bZ)$ be an effective class, $\tk \in G_{\tsi_0}$, and $\tgamma_1, \dots, \tgamma_n \in H^2_{\CR, \tT'}(\tcX; \bQ)$. Then
\begin{equation}\label{eqn:ClosedLocalResultOuter}
\begin{aligned}
   \inner{\tgamma_1,\ldots, \tgamma_n, \tgamma_{\tk}}_{\tbeta}^{\tcX, T_f} = & \sum_{\vGa \in \Gamma_{0,n+1}(\tcX, \tbeta)} c_{\vGa} \cdot \prod_{e \in E(\Gamma)} \tbh(e) \cdot \prod_{(e,v) \in F(\Gamma)} \tbh(e,v) \cdot \prod_{v \in V(\Gamma)} \left( \prod_{i \in S_v} \iota_{\sigma_v}^* (\tgamma_i) \right)\\
   & \cdot \prod_{v \in V(\Gamma)} \int_{\Mbar_{0, \vk_v}(\cB G_v)} \frac{\tbh(v)}{\prod_{e \in E_v} \left(\tbw_{(e,v)} - \frac{\bar{\psi}_{(e,v)}}{r_{(e,v)}}\right)} \bigg|_{\su_4 = 0, \su_2 - f\su_1 = 0}.
\end{aligned}
\end{equation}
\end{proposition}

We explain the notations used in \eqref{eqn:ClosedLocalResultOuter} above:
\begin{itemize}
    \item For each $\vGa \in \Gamma_{0,n+1}(\tcX, \tbeta)$, the coefficient $c_{\vGa}$ is defined in \eqref{eqn:cGamma}. To give the definitions of the other quantities in \eqref{eqn:LiuDiskLocalOuter}, we pick a stable map $u: (\cC, \fx_1, \dots, \fx_{n+1}) \to \tcX$ whose associated decorated graph is $\vGa$ (see Section \ref{sect:DecGraphs}). The definitions do not depend on the choice of $u$. 

    \item For each $e \in E(\Gamma)$, define
        $$
            \tbh(e):= \frac{e_{\tT'}(H^1(\cC_e,(u|_{\cC_e})^*T\tcX)^m)}{e_{\tT'}(H^0(\cC_e,(u|_{\cC_e})^*T\tcX)^m)},
        $$
        which is explicitly computed in \cite[Lemma 130]{Liu13} in terms of tangent $\tT'$-weights.

    \item For each $(e,v) \in F(\Gamma)$, define
        $$
            \tbh(e,v):= e_{\tT'}\left((T_{\fp_{\sigma_v}}\tcX)^{k_{(e,v)}}\right) = \prod_{ \substack{ (\ttau, \sigma_v) \in F(\tSi) \\ k_{(e,v)} \in G_{\ttau} }} \tbw(\ttau, \sigma_v).
        $$
        Here, $(T_{\fp_{\sigma_v}}\tcX)^{k_{(e,v)}}$ is the maximal subspace of $T_{\fp_{\sigma_v}}\tcX$ that is invariant under the action of $k_{(e,v)}$.

    \item We set
        $$
            \tgamma_{n+1}:= \tgamma_{\tk}
        $$
        for convenience.

    \item For each stable vertex $v \in V^S(\vGa)$, define
        $$
            \tbh(v):= \frac{e_{\tT'}(H^1(\cC_v,(u|_{\cC_v})^*T\tcX)^m)}{e_{\tT'}(H^0(\cC_v,(u|_{\cC_v})^*T\tcX)^m)},
        $$
        which is explicitly computed in \cite[Lemma 126]{Liu13} in terms of Hurwitz-Hodge classes and tangent $\tT'$-weights. For each unstable vertex $v \not \in V^S(\vGa)$, define
        $$
            \tbh(v):= \begin{cases}
                \tbh(e,v)^{-1} & \text{if } v \in V^1(\vGa) \cup V^{1,1}(\vGa), E_v = \{e\},\\
                \tbh(e_1,v)^{-1} = \tbh(e_2,v)^{-1} & \text{if } v \in V^2(\vGa), E_v = \{e_1, e_2\}.
            \end{cases}
        $$

    \item For each $(e,v) \in F(\Gamma)$, define
        $$
            \tbw_{(e,v)}:= e_{\tT'}(T_{\fn(e,v)}\cC_e) = \frac{\fr(\tau_e, \sigma_v)\tbw(\tau_e, \sigma_v)}{r_{(e,v)}d_e},
        $$
        where the $\tT'$-action on $T_{\fn(e,v)}\cC_e$ is induced from that on $T_{\fp_{\sigma_v}}\fl_{\tau_e}$.
\end{itemize}

We note that for any $\tsi \in \tSi(4)$, $\tsi \neq \tsi_0$, and any $\lambda \in G_{\tau_0}$,
$$
    \iota_{\tsi}^* (\tgamma_{n+1}) = 0.
$$
Thus only the decorated graphs satisfying $\vf \circ \vs(n+1) = \tsi_0$ can contribute.

\section{Numerical open/closed correspondence}\label{sect:Numerical}
In this section, we establish the open/closed correspondence at the numerical level (Theorem \ref{thm:NumericalOuter}). That is, we identify the disk invariants of $(\cX, \cL, f)$ and the closed Gromov-Witten invariants of $\tcX$ in corresponding curve classes.


\subsection{The statement}\label{sect:NumericalOuter}

Recall from Section \ref{sect:ComparisonOuter} the inclusion
$$
    \iota_*: H_2(X, L; \bZ) \to H_2(\tX; \bZ)
$$
which maps $[B]$ to $[l_{\iota(\tau_0)}]$. Moreover, Convention \ref{conv:LiftOuter} specifies equivariant lifts of classes in $H^2_{\CR}(\cX;\bQ)$ taken as insertions in the Gromov-Witten invariants.

\begin{theorem}\label{thm:NumericalOuter}
Let $n \in \bZ_{\ge 0}$, $\beta \in H_2(X;\bZ)$ be an effective class, and $(d, \lambda) \in H_1(\cL;\bZ) \cong \bZ \times G_{\tau_0}$ such that $d>0$. Set
$$
    \beta' = \beta + d[B], \quad \tbeta = \iota_*(\beta'), \quad \tk:= \pi_{(\iota(\tau_0), \tsi_0)}(d, \lambda) \in G_{\tsi_0},
$$
(see \eqref{eqn:FlagFundGroup}). Let $\gamma_1, \dots, \gamma_n \in H^2_{\CR}(\cX;\bQ)$ and, by an abuse of notation, $\gamma_1, \dots, \gamma_n \in H^2_{\CR, T'}(\cX; \bQ)$, $\tgamma_1, \dots, \tgamma_n \in H^2_{\CR, \tT'}(\tcX; \bQ)$ be the equivariant lifts chosen as in Convention \ref{conv:LiftOuter}. Then
$$
    \inner{\gamma_1,\ldots,\gamma_n}_{\beta', (d, \lambda)}^{\cX, (\cL, f)} = \inner{\tgamma_1,\ldots,\tgamma_n, \tgamma_{\tk}}_{\tbeta}^{\tcX, T_f}.
$$
\end{theorem}

We note that by the description of the map $\pi_{(\iota(\tau_0), \tsi_0)}$ in \eqref{eqn:FlagMapDiagram} and the discussion in Section \ref{sect:stablizer}, $\age(\tk) \le 1$ if and only if $\fa \mid d$, and $\age(\tk)=2$ in the other case.

For the rest of this section, we give an outline of the proof and set up some notations. Recall that the localization computations in Propositions \ref{prop:DiskLocalResultOuter} and \ref{prop:ClosedLocalResultOuter} express 
$\inner{\gamma_1,\dots,\gamma_n}_{\beta', (d, \lambda)}^{\cX, (\cL, f)}$ as a sum of contributions from decorated graphs in $\Gamma_{0, n+1}^{0,(d, \lambda)}(\cX, \beta)$ (see \eqref{eqn:DiskGamma00Outer}) and $\inner{\tgamma_1,\dots,\tgamma_n,
\tgamma_{\tk}}_{\tbeta}^{\tcX, T_f}$ as a sum of contributions from decorated graphs in $\Gamma_{0,n+1}(\tcX, \tbeta)$. There are two main steps in our proof of Theorem \ref{thm:NumericalOuter}. First, we set up a one-to-one correspondence between $\Gamma_{0, n+1}^{0,(d, \lambda)}(\cX, \beta)$ and a subset
$$
   \Gamma_{0, n+1}^{0,\tk}(\tcX, \tbeta)
$$
of decorated graphs in $\Gamma_{0,n+1}(\tcX, \tbeta)$, and directly relate the contributions from corresponding decorated graphs. Second, we show that there is no contribution from any decorated graph in $\Gamma_{0,n+1}(\tcX, \tbeta)$ outside the subset 
$\Gamma_{0, n+1}^{0,\tk}(\tcX, \tbeta)$.

To define $\Gamma_{0, n+1}^{0,\tk}(\tcX, \tbeta)$, we consider a map
$$
    \epsilon: \Gamma_{0, n+1}^{0,(d,\lambda)}(\cX, \beta) \to \Gamma_{0, n+1}(\tcX, \tbeta),
$$
defined as follows: Given $\vGa \in \Gamma_{0, n+1}^{0,(d,\lambda)}(\cX, \beta)$, to obtain $\epsilon(\vGa)$, we first post-compose the label map $\vf$ with the map $\iota$ of cones (see \eqref{eqn:ConeMapIotaOuter}). Then, we replace the $(n+1)$-th marked point in $\vGa$, with marking $v_0 = v_0(\vGa) := \vs(n+1)$, by
\begin{itemize}
    \item a new vertex $\tv_0 = \tv_0(\epsilon(\vGa))$ with label $\vf(\tv_0) = \tsi_0$;

    \item a new edge $e_0 = e_0(\epsilon(\vGa))$ between $v_0$ and $\tv_0$, with label $\vf(e_0) = \iota(\tau_0)$ and degree $\vd(e_0) = (d, \lambda) \in H_{\iota(\tau_0)} = H_{\tau_0} \cong \bZ \times G_{\tau_0}$;

    \item a new $(n+1)$-th marked point with marking $\vs(n+1) = \tv_0$ and twisting $\vk(n+1) = \tk \in G_{\tsi_0}$.

\end{itemize}
See Figure \ref{fig:NewDecGraphOuter}. The twisting at the new flags are $k_{(e_0, v_0)} = h(d, \lambda)$, $k_{(e_0, \tv_0)} = \tk$ and the compatibility at vertices $v_0, \tv_0$ in $\epsilon(\vGa)$ are satisfied. It is straightforward to check that $\epsilon(\vGa)$ indeed belongs to $\Gamma_{0, n+1}(\tcX, \tbeta)$ and that $\epsilon$ is injective. We define $\Gamma_{0, n+1}^{0,\tk}(\tcX, \tbeta)$ to be the image of $\epsilon$.

\begin{figure}[h]
    \begin{tikzpicture}[scale=0.7]
        \coordinate (v0) at (0, 0);
        \coordinate (v1) at (-1.5, 1);        
        \coordinate (v2) at (-1.5, -1);
                
        \node at (v0) {$\bullet$};
        \node at (v1) {$\bullet$};
        \node at (v2) {$\bullet$};

        \draw (-2.5,1) -- (v1) -- (v0) -- (v2) -- (-2.5,-1);
        \draw (v0) -- (0.5, 0);        

        \node at (0, -0.5) {$v_0$};
        \node at (-3, 1) {$\cdots$};
        \node at (-3, -1) {$\cdots$};
        \node at (1.3, 0) {$n+1$};
        \node at (-0.5, -2) {$\vGa$};

        \node at (3.5, 0) {$\Rightarrow$};

        \coordinate (v00) at (8, 0);
        \coordinate (tv0) at (10, 0);
        \coordinate (v11) at (6.5, 1);        
        \coordinate (v22) at (6.5, -1);
                
        \node at (v00) {$\bullet$};
        \node at (tv0) {$\bullet$};
        \node at (v11) {$\bullet$};
        \node at (v22) {$\bullet$};

        \draw (5.5,1) -- (v11) -- (v00) -- (v22) -- (5.5,-1);
        \draw (v00) -- (tv0) -- (10.5, 0);        

        \node at (8, -0.5) {$v_0$};
        \node at (10, -0.5) {$\tv_0$};
        \node at (9, 0.4) {$e_0$};
        \node at (5, 1) {$\cdots$};
        \node at (5, -1) {$\cdots$};
        \node at (11.3, 0) {$n+1$};
        \node at (8.5, -2) {$\epsilon(\vGa)$};
    \end{tikzpicture}

    \caption{Newly added vertex $\tv_0$ and edge $e_0$ in the graph $\epsilon(\vGa)$.}
    \label{fig:NewDecGraphOuter}
\end{figure}

We set up some additional notations. For each $\vGa \in \Gamma_{0, n+1}^{0,(d,\lambda)}(\cX, \beta)$, let
\begin{align*}
    C_{\vGa}:= &\fr \fm D_{d, \lambda} \cdot c_{\vGa} \cdot \prod_{e \in E(\Gamma)} \bh(e) \cdot \prod_{(e,v) \in F(\Gamma)} \bh(e,v) \cdot \prod_{v \in V(\Gamma)} \left( \prod_{i \in S_v} \iota_{\sigma_v}^* (\gamma_i) \right)\\
   & \cdot \prod_{v \in V(\Gamma)} \int_{\Mbar_{0, \vk_v}(\cB G_v)} \frac{\bh(v)}{\left(\frac{\su_1}{d} - \bar{\psi}_{n+1}\right)^{\delta_{v, n+1}} \cdot \prod_{e \in E_v} \left(\bw_{(e,v)} - \frac{\bar{\psi}_{(e,v)}}{r_{(e,v)}}\right)} \in \cQ_{T'}
\end{align*}
denote the contribution of $\vGa$ to $\inner{\gamma_1,\dots,\gamma_n}_{\beta', (d, \lambda)}^{\cX, (\cL, f)}$ as in \eqref{eqn:DiskLocalResultOuter} in Proposition \ref{prop:DiskLocalResultOuter} before the weight restriction $\su_2 - f\su_1 = 0$. Moreover, for an effective class $\tbeta \in H_2(\tX;\bZ)$ and $\vGa \in \Gamma_{0, n+1}(\tcX, \tbeta)$, let
\begin{align*}
    \tC_{\vGa} := & c_{\vGa} \cdot \prod_{e \in E(\Gamma)} \tbh(e) \cdot \prod_{(e,v) \in F(\Gamma)} \tbh(e,v) \cdot \prod_{v \in V(\Gamma)} \left( \prod_{i \in S_v} \iota_{\sigma_v}^* (\tgamma_i) \right)\\
    & \cdot \prod_{v \in V(\Gamma)} \int_{\Mbar_{0, \vk_v}(\cB G_v)} \frac{\tbh(v)}{\prod_{e \in E_v} \left(\tbw_{(e,v)} - \frac{\bar{\psi}_{(e,v)}}{r_{(e,v)}}\right)} \in \cQ_{\tT'}
\end{align*}
denote the contribution of $\vGa$ to $\inner{\tgamma_1,\ldots, \tgamma_n,\tgamma_{\tk}}_{\tbeta}^{\tcX,T_f}$ as in \eqref{eqn:ClosedLocalResultOuter} in Proposition \ref{prop:ClosedLocalResultOuter} before the weight restriction $\su_4 = 0, \su_2 - f\su_1 = 0$.

\subsection{Matching contributions}
As the first step in proving Theorem \ref{thm:NumericalOuter}, we show the following lemma matching the contributions from decorated graphs in $\Gamma_{0, n+1}^{0,(d, \lambda)}(\cX, \beta)$ and $\Gamma_{0, n+1}^{0,\tk}(\tcX, \tbeta) = \epsilon(\Gamma_{0, n+1}^{0,(d, \lambda)}(\cX, \beta))$.

\begin{lemma}\label{lem:NumericalOuterContribute}
For the quantities defined as in Theorem \ref{thm:NumericalOuter}, we have that for each $\vGa \in \Gamma_{0, n+1}^{0,(d,\lambda)}(\cX, \beta)$,
$$
    \tC_{\epsilon(\vGa)} \big|_{\su_4 = 0} = C_{\vGa}.
$$
\end{lemma}

In particular, by Proposition \ref{prop:DiskLocalResultOuter}, we have
$$
    \inner{\gamma_1,\ldots,\gamma_n}_{\beta', (d, \lambda)}^{\cX, (\cL, f)}= \sum_{\vGa \in \Gamma_{0, n+1}^{0,\tk}(\tcX, \tbeta)} \tC_{\vGa} \big|_{\su_4 = 0, \su_2 - f\su_1 = 0}.
$$

\begin{proof}
Let $\vGa \in \Gamma_{0, n+1}^{0,(d,\lambda)}(\cX, \beta)$ and $\vGa' = \epsilon(\vGa) \in \Gamma_{0, n+1}^{0,\tk}(\tcX, \tbeta)$. The underlying graph $\Gamma = (V(\Gamma), E(\Gamma))$ of $\vGa$ is a subgraph of the underlying graph $\Gamma' = (V(\Gamma'), E(\Gamma'))$ of $\vGa'$. Let $v_0 = v_0(\vGa) \in V(\Gamma)$, $\tv_0 = \tv_0(\vGa') \in V(\Gamma')$, $e_0 = e_0(\vGa') \in E(\Gamma')$. We carry out the comparison between $\tC_{\vGa'} \big|_{\su_4 = 0}$ and $C_{\vGa}$ piece by piece, while supplying computational details in Section \ref{sect:NumericalLocalDetail} in the appendix.
\begin{itemize}
    \item (Coefficients) Note that $\Aut(\vGa) \cong \Aut(\vGa')$. By definition \eqref{eqn:cGamma}, we have
    $$
        c_{\vGa'} = c_{\vGa} \cdot \frac{\frac{|G_{v_0}|}{r_{(e_0, v_0)}} \cdot \frac{|G_{\tv_0}|}{r_{(e_0, \tv_0)}}}{d_{e_0}|G_{e_0}|} = c_{\vGa} \cdot \frac{\fr\fa\fm}{d r_{(e_0, v_0)} r_{(e_0, \tv_0)}}.
    $$

    \item (Edges) For each $e \in E(\Gamma)$, by Lemma \ref{lem:EdgeComparison},  we have
    $$
        \left(\su_4 \tbh(e) \right) \big|_{\su_4 = 0} = \bh(e).  
    $$
    For the edge $e_0$, by the computation of $\tbh(e_0)$ in Lemma \ref{lem:EdgeIotaTau0Outer}, we have 
    $$
  \tbh(e_0) \big|_{\su_4 = 0} = \begin{dcases}
            \displaystyle{     - d \fm \left(\frac{\su_1}{\fa} \right)^{-1} \left( \frac{\su_2-f\su_1}{\fm} \right)^{-1} \cdot D_{d,\lambda} }& \text{if } \tk = 1,\\
            \displaystyle{     - d \fm \left(\frac{\su_1}{\fa} \right)^{-1} \cdot D_{d,\lambda}  }& \text{if } \age(\tk)=1,\\
            \displaystyle{     d \fm \cdot D_{d,\lambda}  }& \text{if } \age(\tk)=2.\\    
        \end{dcases}
    $$

    \item (Flags) For each $(e,v) \in F(\Gamma)$, by Lemma \ref{lem:FlagComparison}, we have
    $$
        \frac{\tbh(e,v)}{\su_4} \bigg|_{\su_4 = 0} = \bh(e,v).
    $$
    For the flag $(e_0, v_0)$, we have
      $$ 
        \tbh(e_0, v_0) = \prod_{ \substack{ \ttau \in \{\iota(\tau_0), \iota(\tau_2), \iota(\tau_3), \sigma_0\} \\ h(d,\lambda) \in G_{\ttau} } } \tbw(\ttau, \iota(\sigma_0)) 
        = \su_4 \cdot \prod_{ \substack{ \ttau \in \{\iota(\tau_0), \iota(\tau_2), \iota(\tau_3)\}\\  h(d,\lambda) \in G_{\ttau} } } \tbw(\ttau, \iota(\sigma_0)).
$$ 
    Therefore,
    \begin{equation}\label{eqn:He0v0PointClass}
        \frac{\tbh(e_0,v_0)}{\su_4} \bigg|_{\su_4 = 0} = \prod_{ \substack{ \tau \in \{\tau_0, \tau_2, \tau_3\} \\ h(d,\lambda) \in G_{\tau} } } \bw(\tau, \sigma_0) = \iota_{\sigma_0}^*(\phi_{\sigma_0, h(d, \lambda)^{-1}}).
    \end{equation}
    The term $\tbh(e_0,\tv_0)$ for the flag $(e_0, \tv_0)$ will be accounted for by the term $\tbh(\tv_0)$ for the unstable vertex $\tv_0 \in V^{1,1}(\vGa')$ below.

    \item (Insertions) For any $i = 1, \dots, n$, with $i \in S_v$ for $v \in V(\Gamma)$, since $\iota^*(\tgamma_i) = \gamma_i$, we have
    $$
        \iota_{\tsi_v}^*(\tgamma_i) \big|_{\su_4} = \iota_{\sigma_v}^*(\gamma_i).
    $$
    The insertion $\iota_{\sigma_0}^*(\gamma_{n+1}) = \iota_{\sigma_0}^*(\phi_{\sigma_0, h(d, \lambda)^{-1}})$ is related by \eqref{eqn:He0v0PointClass} above to $\tbh(e_0, v_0)$. Moreover, we have
    \begin{equation}\label{eqn:ClosedInsertionRestrictOuter}
        \iota_{\tsi_0}^*(\tgamma_{n+1}) = \begin{cases}
                \displaystyle{      -\frac{\su_1}{\fa} \cdot \frac{\su_2-f\su_1}{\fm} } & \text{if } \tk = 1,\\  
                \noalign{\vskip5pt}
            \displaystyle{    -\frac{\su_1}{\fa} \one_{\tk^{-1}} } & \text{if } \age(\tk) =1,\\
            \noalign{\vskip3pt}
            \displaystyle{  \one_{\tk^{-1}} } & \text{if } \age(\tk) =2.
                     \end{cases}
    \end{equation}
  
    \item (Vertices) For any $v \in V(\Gamma) \setminus \{v_0\}$, by Lemmas \ref{lem:FlagComparison} (when $v$ is unstable), \ref{lem:VertexComparison}, and \ref{lem:IntegralComparison}, we have
    $$
        \int_{\Mbar_{0, \vk_v}(\cB G_v)} \frac{\su_4\tbh(v)}{\prod_{e \in E_v} \left(\tbw_{(e,v)} - \frac{\bar{\psi}_{(e,v)}}{r_{(e,v)}}\right)} \bigg|_{\su_4 = 0} = \int_{\Mbar_{0, \vk_v}(\cB G_v)} \frac{\bh(v)}{\prod_{e \in E_v} \left(\bw_{(e,v)} - \frac{\bar{\psi}_{(e,v)}}{r_{(e,v)}}\right)}.
    $$

    For the vertex $v_0$, note that the $(n+1)$-th marked point in $\vGa$ is replaced by the flag $(e_0, v_0)$ in $\vGa'$, and the twisting $k_{n+1}$ in $\vGa$ is identified with $k_{(e_0, v_0)}^{-1} = h(d, \lambda)^{-1}$. Therefore, the integral
    $$
         \int_{\Mbar_{0, \vk_{v_0}}(\cB G_{v_0})} \frac{\su_4\tbh(v_0)}{\prod_{e \in E_{v_0}} \left(\tbw_{(e,v_0)} - \frac{\bar{\psi}_{(e,v_0)}}{r_{(e,v_0)}}\right)} \bigg|_{\su_4 = 0}
    $$
    in $\tC_{\vGa'}$ (after multiplied by $\su_4$ and restricted to $\su_4 = 0$) is identified with
    $$
        r_{(e_0, v_0)} \cdot \int_{\Mbar_{0, \vk_{v_0}}(\cB G_{v_0})} \frac{\bh(v_0)}{\left(\frac{\su_1}{d} - \bar{\psi}_{n+1}\right) \cdot \prod_{e \in E_{v_0}} \left(\bw_{(e,v_0)} - \frac{\bar{\psi}_{(e,v_0)}}{r_{(e,v_0)}}\right)}
    $$
    where the integral is the one in $C_{\vGa}$.

    Finally, for the unstable vertex $\tv_0 \in V^{1,1}(\vGa')$, we have by definition that $\tbh(\tv_0) = \tbh(e_0, \tv_0)^{-1}$. By \eqref{eqn:UnstableIntegral}, we have
    $$
        \int_{\Mbar_{0, \vk_{\tv_0}}(\cB G_{\tv_0})} \frac{\tbh(\tv_0)}{\tbw_{(e_0,\tv_0)} - \frac{\bar{\psi}_{(e_0,\tv_0)}}{r_{(e_0,\tv_0)}}} \bigg|_{\su_4 = 0} = \frac{r_{(e_0, \tv_0)}}{|G_{\tv_0}|} \tbh(\tv_0) \big|_{\su_4 = 0} = \frac{r_{(e_0, \tv_0)}}{\fa\fm}  \cdot \left( \tbh(e_0, \tv_0)\big|_{\su_4= 0} \right)^{-1}.
    $$
\end{itemize}

We now piece together the above comparisons. First, we collect the powers of $\su_4$ appearing in the pieces of $\tC_{\vGa'}$. Each edge, flag, vertex in $\Gamma$ contributes a power of $-1, 1, -1$ respectively. Since the graph $\Gamma$ is a tree, we have
$$
    -|E(\Gamma)| + |F(\Gamma)| - |V(\Gamma)| = -1.
$$
In addition, the extra flag $(e_0, v_0)$ in $\Gamma'$ contributes a power of $1$, while there is no contribution from $e_0$, $(e_0, \tv_0)$, or $\tv_0$. Therefore, the total power of $\su_4$ in $\tC_{\vGa'}$ is $0$. Summarizing the comparisons of the pieces above, we have that when $\tk = 1$,
$$
    \tC_{\vGa'} \big|_{\su_4 = 0} = C_{\vGa} \cdot \frac{\fa}{d r_{(e_0, v_0)} r_{(e_0, \tv_0)}} \cdot (-1) d \fm \left(\frac{\su_1}{\fa} \right)^{-1} \left( \frac{\su_2-f\su_1}{\fm} \right)^{-1} \cdot \left( -\frac{\su_1}{\fa} \cdot \frac{\su_2-f\su_1}{\fm} \right) \cdot r_{(e_0, v_0)} \cdot \frac{r_{(e_0, \tv_0)}}{\fa\fm} = C_{\vGa}.
$$
When $\age(\tk) = 1$, we have
$$
    \tC_{\vGa'} \big|_{\su_4 = 0} = C_{\vGa} \cdot \frac{\fa}{d r_{(e_0, v_0)} r_{(e_0, \tv_0)}} \cdot (-1) d \fm \left(\frac{\su_1}{\fa} \right)^{-1} \cdot \left( -\frac{\su_1}{\fa} \right) \cdot r_{(e_0, v_0)} \cdot \frac{r_{(e_0, \tv_0)}}{\fa\fm} = C_{\vGa}.
$$
When $\age(\tk) = 2$, we have
$$
    \tC_{\vGa'} \big|_{\su_4 = 0} = C_{\vGa} \cdot \frac{\fa}{d r_{(e_0, v_0)} r_{(e_0, \tv_0)}} \cdot d \fm  \cdot r_{(e_0, v_0)} \cdot \frac{r_{(e_0, \tv_0)}}{\fa\fm} = C_{\vGa}.
$$
\end{proof}

\subsection{Vanishing arguments}
In this section, we study the contributions from decorated graphs in $\Gamma_{0,n+1}(\tcX, \tbeta) \setminus \Gamma_{0, n+1}^{0,\tk}(\tcX, \tbeta)$. Recall first that we made the following observation at the end of Section \ref{sect:ClosedOuter}.

\begin{observation}\label{obs:AModelLVanish}
For any effective class $\tbeta \in H_2(\tX;\bZ)$ and $\vGa \in \Gamma_{0,n+1}(\tcX, \tbeta)$, $\tC_{\vGa} = 0$ unless $\vs(n+1) = \tv_0$ and $\vf(\tv_0) = \tsi_0$. In particular, if $\tbeta \in H_2(X; \bZ)$, then $\tC_{\vGa} = 0$ for all $\vGa \in \Gamma_{0,n+1}(\tcX, \tbeta)$.
\end{observation}

We note that in the above, there is a special case when $\tbeta = 0$, $n = 2$, and $\vGa$ represents a 3-pointed constant map to $\fp_{\tsi_0}$, but in this case $\iota_{\tsi_0}^* (\tgamma_1) = \iota_{\tsi_0}^* (\tgamma_2) = 0$ and thus $\tC_{\vGa} = 0$.

Now we introduce some notations for an effective class $\tbeta \in H_2(\tX;\bZ) \setminus H_2(X;\bZ)$ and $\vGa \in \Gamma_{0,n+1}(\tcX, \tbeta)$. We partition $V(\Gamma)$ into two subsets
$$
    V_0 := \{v \in V(\Gamma) : \vf(v) \in \iota(\Sigma(3)) \}, \quad  V_1 := \{v \in V(\Gamma) : \vf(v) \in \tSi(4) \setminus \iota(\Sigma(3)) \}
$$
and let $\Gamma_0 = (V_0, E_0)$, $\Gamma_1 = (V_1, E_1)$ be induced subgraphs of $\Gamma$ on them, where
$$
    E_0 := \{e \in E(\Gamma): \vf(e) \in \iota(\Sigma(2)_c) \}, \quad E_1 := \{e \in E(\Gamma): \vf(e) \in \tSi(3)_c \setminus \iota(\Sigma(2)) \}.
$$
The two subgraphs can be disconnected, but the components are connected by edges in
$$
    E_{01} := E(\Gamma) \setminus (E_0 \cup E_1) = \{e \in E(\Gamma): \vf(e) \in \iota(\Sigma(2) \setminus \Sigma(2)_c) \}.
$$
See Figure \ref{fig:SubgraphsOuter}. 
Let
\begin{equation}\label{eqn:c0}
c_0= |V_0|- |E_0| 
\end{equation}
be the number of connected components of the graph $\Gamma_0$. 
Restricting the decorations of $\vGa$ to a component of $\Gamma_0$ yields a decorated graph whose degree is a curve class in $H_2(X;\bZ)$. Since $\tbeta \not \in H_2(X; \bZ)$, the graph $\Gamma_1$ must be non-empty. We further partition $F(\Gamma)$ into two subsets
$$
    F_0 := \{(e,v) \in F(\Gamma): v \in V_0\}, \qquad F_1 := \{(e,v) \in F(\Gamma): v \in V_1\}.
$$
Then
\begin{equation}\label{eqn:F0}
|F_0| = 2|E_0|+|E_{01}|.
\end{equation}  
Moreover, we denote
$$
    V_1' := \{ v \in V_1 \cap V^2(\vGa) : \vf(E_v) = \{\delta_2(\vf(v)), \delta_3(\vf(v))\}\}.
$$

\begin{figure}[h]
    \begin{tikzpicture}[scale=0.5]
        \draw (10,0) ellipse (1.2 and 1.4);
        \draw (10,5.6) ellipse (1.2 and 1.4);

        \draw (2,-0.8) ellipse (2.2 and 1);
        \draw (2,2) ellipse (2.2 and 1);
        \draw (2,6) ellipse (2.2 and 1);

        \node at (10,2.8) {$\bullet$};
      
        \draw (10,4.7) -- (10,2.8) -- (10,0.9);
        \draw (3.8,6) -- (9.3,5.7);
        \draw (3.8,2) -- (9.3,0.5);
        \draw (3.8,-0.8) -- (9.3,-0.5);

        \draw[dashed] (-0.7,-2.5) -- (4.7,-2.5) -- (4.7, 7.7) -- (-0.7,7.7) -- (-0.7,-2.5);
        \draw[dashed] (8.3,-2.5) -- (11.7,-2.5) -- (11.7, 7.7) -- (8.3,7.7) -- (8.3,-2.5);

        \node at (11,2.8) {$\in V_1'$};
        \node at (10, -3.5) {$\Gamma_1 = (V_1, E_1)$};
        \node at (2, -3.5) {$\Gamma_0 = (V_0, E_0)$};
        \node at (6.5, 2.8) {$E_{01}$};
    \end{tikzpicture}

    \caption{Subgraphs $\Gamma_0$, $\Gamma_1$ connected by $E_{01}$, and a vertex in $V_1'$.}
    \label{fig:SubgraphsOuter}
\end{figure}

We now prove the following preparatory lemma.

\begin{lemma}\label{lem:NumericalOuterU4Vanish}
Let $\tbeta \in H_2(\tX;\bZ) \setminus H_2(X;\bZ)$ be an effective class and $\vGa \in \Gamma_{0,n+1}(\tcX, \tbeta)$. Unless
\begin{equation}\label{eqn:U4VanishCriterion}
    |E_1| - |V_1'| = |E_{01}| - c_0 = 0,
\end{equation}
we have
$$
    \tC_{\vGa} \big|_{\su_4 = 0} = 0.
$$
\end{lemma}

Note that condition \eqref{eqn:U4VanishCriterion} is only possible if $\Gamma_1$ consists of a single vertex, denoted by $\tv_0$, and each edge in $E_{01}$ connects a component of $\Gamma_0$ to $\tv_0$. In this case, all edges in $E_{01}$ are labeled by $\iota(\tau_0)$.

\begin{proof}
We consider the power of $\su_4$ in $\tC_{\vGa}$. By Lemmas \ref{lem:EdgeComparison} -- \ref{lem:ExtraVertexOuter}, \ref{lem:ExtraIntegralOuter} -- \ref{lem:ExtraEdge2Outer}, each edge in $E_0$ or vertex in $V_0 \cup V_1'$ contributes a power of $-1$, each flag in $F_0$ or edge in $E_1$ contributes a power of $1$, and there are no other contributions. Since each connected component of $\Gamma_0$ is a tree, we see that the total power of $\su_4$ is
\begin{equation}\label{eqn:U4PowerOuter}
   |F_0| - |E_0| - |V_0| + |E_1| - |V_1'| = (|E_{01}| - c_0) + (|E_1| - |V_1'|),
\end{equation}
where we use \eqref{eqn:c0} and \eqref{eqn:F0}. Since the subgraph $\Gamma_1$ is non-empty, we have $|E_{01}| \ge c_0$. Moreover, since each vertex in $V_1'$ is incident to two distinct edges in $E_1$, we have $|E_1| \ge |V_1'|$. Therefore, the power \eqref{eqn:U4PowerOuter} is always non-negative, and is zero only if $|E_1| - |V_1'| = |E_{01}| - c_0 = 0$. This implies the lemma.
\end{proof}

Lemma \ref{lem:NumericalOuterU4Vanish} implies the following result, which we will use in Section \ref{sect:Generating}.

\begin{lemma}\label{lem:OtherClassVanish}
For any effective class $\tbeta \in H_2(\tX;\bZ) \setminus \iota_*(H_2(X, L; \bZ))$ and $\vGa \in \Gamma_{0,n+1}(\tcX, \tbeta)$, we have
$$
    \tC_{\vGa} \big|_{\su_4 = 0} = 0.
$$
\end{lemma}

\begin{proof}
It suffices to observe that such $\vGa$ cannot satisfy condition \eqref{eqn:U4VanishCriterion}, and thus Lemma \ref{lem:NumericalOuterU4Vanish} applies.
\end{proof}

Now we show the following lemma as the second main ingredient in proving Theorem \ref{thm:NumericalOuter}.

\begin{lemma}\label{lem:NumericalOuterVanish}
Consider the setup of Theorem \ref{thm:NumericalOuter}. Assume that $f \in \bQ$ is \emph{generic} (with respect to $\tbeta$), i.e. avoiding a finite set of rational numbers. Then for any $\vGa \in \Gamma_{0,n+1}(\tcX, \tbeta) \setminus \Gamma_{0, n+1}^{0,\tk}(\tcX, \tbeta)$, we have
$$
    \tC_{\vGa} \big|_{\su_4 = 0, \su_2 - f\su_1 = 0} = 0.
$$
\end{lemma}

\begin{proof}
By Lemma \ref{lem:NumericalOuterU4Vanish}, it suffices to consider the case where condition \eqref{eqn:U4VanishCriterion} holds. We again denote the only vertex in $\Gamma_1$ by $\tv_0$. By Observation \ref{obs:AModelLVanish}, it suffices to consider the case $\vs(n+1) = \tv_0$ and $\vf(\tv_0) = \tsi_0$. Note that all edges in $E_{01}$ are incident to $\tv_0$. Since $\vGa \not \in \Gamma_{0, n+1}^{0,\tk}(\tcX, \tbeta)$, either $|E_{01}|>1$, or $|E_{01}|=1$ and $S_{\tv_0}$ contains a marking other than $n+1$. In either case, $\tv_0$ is a stable vertex.

We assume without loss of generality that for each $i = 1, \dots, n$, $\tgamma_i$ is contained in $H^*_{\tT'}(\tcX_{j_i}; \bQ)$ for some $j_i \in \Box(\tcX)$. If $j_i = \vzero$, then the way the lift $\tgamma_i$ is chosen (Convention \ref{conv:LiftOuter}) implies that we must have $\vs(i) \neq \tv_0$ in order for $\tC_{\vGa}$ to be non-zero. Therefore, $S_{\tv_0}$ consists of $n+1$ and a subset of $\{i \in \{1, \dots, n\} \mid j_i \neq \vzero\}$. Recall our notation of the vector $\vk_{\tv_0} = (k_{(e, \tv_0)}^{-1}, k_i) \in G_{\tsi_0}^{E_{01} \cup S_{\tv_0}}$. In particular, $k_{n+1} = \tk$ and $k_i \neq 1$ for any other $i \in S_{\tv_0}$.

Now we consider the power of $\su_2 - f\su_1$ in $\tC_{\vGa} \big|_{\su_4 = 0}$ for a generic choice of $f \in \bQ$. We consider several cases below and show that in each case, the total power is positive. This implies that $\tC_{\vGa} \big|_{\su_4 = 0, \su_2 - f\su_1 = 0} = 0$. 
\begin{itemize}
    \item \textit{Case I: $\tk = 1$.} Note from \eqref{eqn:ClosedInsertionRestrictOuter} that $\iota_{\tsi_0}^*(\tgamma_{n+1})$ contributes a power of $1$ in the case. By Lemmas \ref{lem:GenFNoPower} -- \ref{lem:ExtraVertexOuter}, \ref{lem:ExtraIntegralOuter}, \ref{lem:EdgeIotaTau0Outer}, the only additional powers come from the vertex $\tv_0$, edges in $E_{01}$, and flags consisting of $\tv_0$ and an incident edge. First suppose $\vk_{\tv_0} = (1)^{E_{01} \cup S_{\tv_0}}$, that is, $k_{(e, \tv_0)} = 1$ for all $e \in E_{01}$ and $S_{\tv_0} = \{n+1\}$. Then $\tv_0$ contributes a power of $-2$, each edge $e \in E_{01}$ contributes a power of $-1$, and the corresponding flag $(e, \tv_0)$ contributes a power of $2$. Thus the total power of $\su_2 - f\su_1$ is
    $$
        1 - 2 + |E_{01}| = |E_{01}| - 1.
    $$
    Since $\vGa \not \in \Gamma_{0, n+1}^{0,\tk}(\tcX, \tbeta)$, we have $|E_{01}| > 1$, which makes the above power strictly positive. 

    On the other hand, suppose $\vk_{\tv_0} \neq (1)^{E_{01} \cup S_{\tv_0}}$, that is, not all $k_{(e, \tv_0)}$ are $1$, or $S_{\tv_0} \neq \{n+1\}$. Then $\tv_0$ has no contribution, any edge $e \in E_{01}$ with $k_{(e, \tv_0)} = 1$ contributes a power of $-1$, and the corresponding flag $(e, \tv_0)$ contributes a power of $2$. Thus the total power of $\su_2 - f\su_1$ is
    $$
        1 + |\{e \in E_{01} : k_{(e, \tv_0)} = 1\}|>0.
    $$
    
    \item \textit{Case II: $\tk \neq 1$.} In this case, there is no contribution from $\iota_{\tsi_0}^*(\tgamma_{n+1})$. Moreover, $k_i \neq 1$ for all $i \in S_{\tv_0}$. Similar to above, there is no contribution from the denominator of $\tbh(\tv_0)$, and the total contribution from edges in $E_{01}$ and the associated flags is
    $$
        |\{e \in E_{01} : k_{(e, \tv_0)} = 1\}|.
    $$
    This is positive if there exists some $e \in E_{01}$ such that $k_{(e, \tv_0)} = 1$. In the remaining case where $k_{(e, \tv_0)} \neq 1$ for all $e \in E_{01}$, Lemma \ref{lem:VertexTsi0Outer} implies that $\su_2 - f\su_1$ is a power of the numerator of $\tbh(\tv_0)$, which makes the total power positive.
\end{itemize}
\end{proof}

\subsection{Completing the proof}

\begin{proof}[Proof of Theorem \ref{thm:NumericalOuter}]
For generic $f \in \bQ$ (with respect to $\beta'$ and $\tbeta$), Theorem \ref{thm:NumericalOuter} directly follows from Lemmas \ref{lem:NumericalOuterContribute} and \ref{lem:NumericalOuterVanish}. We now prove the theorem for any arbitrary $f$. Let
$$
    C = C(\su_1, \su_2) : = \inner{\gamma_1, \dots, \gamma_n}_{\beta', (d, \lambda)}^{\cX, \cL} \in \bQ(\su_1, \su_2),
$$
which for generic $f \in \bQ$ restricts to $\inner{ \gamma_1, \dots, \gamma_n }^{\cX,(\cL,f)}_{\beta', (d,\lambda)}$ under the weight restriction $\su_2 - f\su_1 = 0$. 
On the other hand, we consider the dependence of the toric Calabi-Yau 4-orbifold $\tcX$ and its closed invariants on $f \in \bQ$. Fixing $\cX$, $\cL$, $(d,\lambda)$, $\beta' = \beta + d[B]$, and $\gamma_1, \dots, \gamma_n$ as in Theorem \ref{thm:NumericalOuter}, we let $\tcX_f$ denote the 4-orbifold that we construct for $(\cX, \cL, f)$ and $\tbeta_f \in H_2(\tcX_f;\bZ)$ be the image of $\beta'$. Let
$$
    \tC = \tC(f, \su_1, \su_2, \su_4) : = \inner{\tgamma_1, \dots, \tgamma_n, \tgamma_{\tk}}_{\tbeta_f}^{\tcX_f,\tT'} \in \bQ(f, \su_1, \su_2, \su_4).
$$
By Lemma \ref{lem:ClosedWellDefined}, $\tC$ has no pole along $\su_4 = 0, \su_2 - f\su_1 = 0$, and thus restricts to
$$
    N = N(f) := \inner{\tgamma_1,\ldots,\tgamma_n, \tgamma_{\tk}}_{\tbeta_f}^{\tcX_f, T_f} \in \bQ(f).
$$
In fact, by Lemma \ref{lem:NumericalOuterU4Vanish}, only decorated graphs satisfying condition \eqref{eqn:U4VanishCriterion} contribute to $\tC$ and $N$. For such graphs, the dependence on $f$ only comes from edges labeled by $\iota(\tau_0)$, which is \emph{polynomial} by Lemma \ref{lem:EdgeIotaTau0Outer}. Thus $N \in \bQ[f]$.

We now consider the difference $D: = C(\su_1, \su_2) - N(f)$, as well as the smooth affine hypersurface $H = \{\su_2 - f\su_1 = 0\} \subset \bC^3 = \Spec(\bC[f, \su_1, \su_2])$. By Lemmas \ref{lem:NumericalOuterContribute} and \ref{lem:NumericalOuterVanish}, $D$ does not have a pole along $H$ and thus restricts to a rational function on $H$. Moreover, $D|_H$ vanishes on an infinite collection of divisors $\{\{f = 0\} \subset H : f \in \bQ \text{ generic}\}$ in $H$. Therefore, $D|_H$ is constantly zero, which implies that the weight restriction
$$
    C \big|_{\su_2 - f\su_1 = 0} = D \big|_{\su_2 - f\su_1 = 0} + N
$$
is defined for all $f \in \bQ$ and identified with $N = \inner{\tgamma_1,\ldots,\tgamma_n, \tgamma_{\tk}}_{\tbeta_f}^{\tcX_f, T_f}$. 
\end{proof}

\begin{remark}\label{rem:DiskWellDefined}
\rm{
The above proof confirms that the disk invariant $\inner{ \gamma_1, \dots, \gamma_n }^{\cX,(\cL,f)}_{\beta', (d,\lambda)}$ is defined for any arbitrary framing $f \in \bQ$.
}
\end{remark}

\section{Generating functions}\label{sect:Generating}
In this section, we use the numerical open/closed correspondence (Theorem \ref{thm:NumericalOuter}) to obtain a correspondence at the level of the generating functions (Theorem \ref{thm:GenCorr}). With this, we further show that the generating function of disk invariants of $(\cX, \cL, f)$ can be recovered from the $\tT'$-equivariant $J$-function of $\tcX$ (Theorem \ref{thm:JPairing}).

\subsection{Definitions}
We choose a basis $\{ u_1,\ldots, u_{R-3}\}$ of  $H^2_{\CR}(\cX;\bQ)\cong \bQ^{R-3}$ such that $u_1,\ldots, u_{R'-3}$  is 
a basis of $H^2(\cX;\bQ)\cong \bQ^{R'-3}$; we further assume that $u_1,\ldots, u_{R-3} $ satisfy the conditions
specified in Section \ref{sec:I}.  For $a = 1, \dots, R-3$, let $\tu_a \in H^2_{\CR}(\tcX; \bQ)$ be the lift of $u_a$ chosen as in Convention \ref{conv:LiftOuter}, and let
$\tu_{R-2} :=\tD_{R+1}$.  Then $\{ \tu_1,\dots, \tu_{R-2}\}$ is a basis of $H^2_{\CR}(\tcX;\bQ)\cong \bQ^{R-2}$, 
$\{ \tu_1,\ldots, \tu_{R'-3}, \tu_{R-2} \}$ is a basis of $H^2(\tcX;\bQ)\cong \bQ^{R'-2}$.

\subsubsection{Generating functions of closed invariants of $\tcX$} \label{sect:GenClosed}
Let
$$
    \tbtau_2 = \sum_{a=1}^{R-2} \ttau_a  \tu_a = \tbtau_2' + \tbtau_2''
$$
where  $\ttau_1,\dots, \ttau_{R-2}$ are complex variables and
$$
\tbtau_2' = \sum_{a=1}^{R'-3} \ttau_a \tu_a + \ttau_{R-2} \tu_{R-2}  \in H^2(\tcX;\bC),\quad  \tbtau_2'' =\sum_{a=R'-2}^{R-3}\ttau_a \tu_a. 
$$
We choose $\tT'$-equivariant lifts of $\tu_1, \dots, \tu_{R-2}$ as in Convention \ref{conv:LiftOuter}. Let 
$$
\cR_{\tT'}^{\bC} : = H_{\tT'}^*(\pt;\bC)=\bC[\su_1,\su_2, \su_4],
$$
and let  $\cQ_{\tT'}^{\bC} = \bC(\su_1,\su_2, \su_4) = \cQ_{\tT'}\otimes_{\bQ}\bC$ be the fractional field of $\cR_{\tT'}^{\bC}$. Define
$$
\tQ_a = e^{\ttau_a}, \quad a=1,\ldots, R'-3, R-2.
$$
Given any commutative ring $\bS$, define
$$
\bS \llbracket \tQ, \tbtau_2''\rrbracket : = \bS \llbracket \tQ_1,\ldots,\tQ_{R'-3}, \tQ_{R-2}, \ttau_{R'-2},\ldots,\ttau_{R-3}\rrbracket. 
$$
Let $\NE(\tX) \subset H_2(\tX;\bR)$ be the Mori cone generated by effective curve classes in the coarse moduli space $\tX$ of $\tcX$.
Let $E(\tX)$ denote the semigroup $\NE(\tX)\cap H_2(\tX;\bZ)$.
\begin{definition} \label{def:Correlator} \rm{
Given $a_1,\ldots, a_n\in \bZ_{\geq 0}$,
$\tgamma_1,\ldots,\tgamma_n\in H^*_{\CR,\tT'}(\cX;\bC)\otimes_{\cR^\bC_{\tT'}} \cQ^\bC_{\tT'}$, we define
the following  generating function of genus zero $\tT'$-equivariant closed Gromov-Witten invariants of $\tcX$: 
\begin{eqnarray*}
\llangle \gamma_1\bar{\psi}^{a_1},  \cdots, \gamma_n\bar{\psi}^{a_n} \rrangle^{\tcX,\tT'} (\tbtau_2)
:=  \sum_{\tbeta\in E(\tX)} \sum_{l=0}^\infty \frac{1}{l!}\langle
\tgamma_1\bar{\psi}^{a_1}, \cdots, \tgamma_n\bar{\psi}^{a_n}, \tbtau_2^l \rangle^{\tcX,\tT'}_{\tbeta}.
\end{eqnarray*}
In particular, 
\begin{eqnarray*} 
    \llangle \tgamma_1,\dots, \tgamma_n \rrangle^{\tcX, \tT'}(\tbtau_2) & :=  &
     \sum_{\tbeta\in E(\tX)} \sum_{l \in \bZ_{\geq 0}} \frac{1}{l!}\langle
\tgamma_1 , \cdots, \tgamma_n , \tbtau_2^l \rangle^{\cX,\tT'}_{\tbeta} \\ 
&=&   \sum_{\tbeta\in E(\tX)} e^{\int_{\tbeta} \tbtau'_2}  \sum_{l \in \bZ_{\geq 0}} \frac{1}{l!}\langle
\tgamma_1 , \cdots, \tgamma_n , (\tbtau''_2)^l \rangle^{\cX,\tT'}_{\tbeta}   
\in \cQ_{\tT'}^{\bC} \llbracket \tQ,\tbtau_2'' \rrbracket
\end{eqnarray*}
where the second equality follows from the divisor equation. 
}
\end{definition} 
\begin{definition}\rm{
For any $\tk \in G_{\tsi_0} \cong \mu_{\fa\fm}$, we define
$$
    \llangle \tgamma_{\tk}\rrangle^{\tcX,T_f} (\tbtau_2) 
     := \sum_{\tbeta \in E(\tX)}  \sum_{l \in \bZ_{\ge 0}}  \frac{1}{l!}  \inner{ \tbtau_2^l ,\tgamma_{\tk}}^{\tcX, T_f}_{\tbeta}
     = \sum_{\tbeta\in E(\tX)}  e^{\int_{\tbeta}\tbtau_2'} \sum_{l \in \bZ_{\ge 0}}  \frac{1}{l!}  \inner{(\tbtau''_2)^l, \tgamma_{\tk}}^{\tcX, T_f}_{\tbeta} 
     \in \bC \llbracket \tQ, \tbtau_2''\rrbracket.
$$
}
\end{definition}
By Observation \ref{obs:AModelLVanish} and Lemma \ref{lem:OtherClassVanish}, the closed invariant $\inner{\tbtau_2^l, \tgamma_{\tk}}^{\tcX, T_f}_{\tbeta}$ 
vanishes for any $\tbeta \in E(\tX)$ that is not of the form $\iota_*(\beta+d[B])$ for some $\beta \in E(X)$ and $d \in \bZ_{>0}$, where $E(X)$ is the semigroup $\NE(X)\cap H_2(X;\bZ)$. (Here again in the case $\tbeta = 0$ and $l = 2$, since $\iota_{\tsi_0}^*(\tbtau_2) = 0$ by our choice of equivariant lifts in Convention \ref{conv:LiftOuter}, we have $\inner{\tbtau_2, \tbtau_2, \tgamma_{\tk}}^{\tcX, \tT'}_{0} = 0$.) Thus we have
$$
  \llangle \tgamma_{\tk}\rrangle^{\tcX,T_f} (\tbtau_2)     = \sum_{\beta \in E(X)} \sum_{d\in \bZ_{>0}}
   \sum_{l \in \bZ_{\ge 0}} \frac{1}{l !} \inner{\tbtau_2^l, \tgamma_{\tk}}^{\tcX, T_f}_{\iota_*(\beta+d[B])}.
$$

\subsubsection{Generating functions of disk invariants of $(\cX,\cL,f)$} \label{sect:GenOpen} 
Now we give the definition of the generating function of disk invariants of $(\cX, \cL, f)$ following \cite[Section 3.13]{FLT12}. Let
$$
    \btau_2 = \sum_{a = 1}^{R-3} \tau_a u_a
$$
where  $\tau_1,\ldots, \tau_{R-3}$ are complex variables.
We choose $T'$-equivariant lifts of $u_1, \dots, u_{R-3}$ and thus of $\btau_2$ as in Convention \ref{conv:LiftOuter}. Let $\sX$ be an additional variable. For $(d, \lambda) \in H_1(\cL; \bZ) \cong \bZ \times G_{\tau_0}$ with $d>0$, we write
$$
    \th(d, \lambda):= \pi_{(\iota(\tau_0), \tsi_0)}(d, \lambda) \in G_{\tsi_0} \cong \mu_{\fa\fm}
$$
(see \eqref{eqn:FlagFundGroup}). Note from \eqref{eqn:FlagMapDiagram} that the map $\th = \pi_{(\iota(\tau_0), \tsi_0)}$ is surjective. 

\begin{definition}\label{def:FOuter}\rm{
For any $\tk \in \mu_{\fa\fm}$, define
$$
    F^{\cX, (\cL, f)}_{\tk}(\btau_2, \sX):= \sum_{\beta\in E(X)} \sum_{\substack{(d, \lambda) \in \bZ_{>0} \times G_{\tau_0} \\ \tk = \th(d,\lambda)}} \sum_{l \in \bZ_{\ge 0}}  
    \frac{\inner{ \btau_2^l }^{\cX, (\cL, f)}_{\beta+d[B], (d, \lambda)}}{l!} \sX^d,
$$
which takes value in $\bC$.
}
\end{definition}

\begin{remark}\label{rem:AmodelGrouping}
\rm{
In \cite{FLT12}, the generating functions of disk invariants are formed by grouping classes $(\beta, d, \lambda)$ according to $\lambda$. This is consistent with Definition \ref{def:FOuter} in the case $f \in \bZ$, i.e. $\fa=1$, where $\th(d,\lambda) = \lambda$ for all $d, \lambda$.
}
\end{remark}

\subsection{Open/closed correspondence for generating functions}\label{sect:GenCorr}
We now identify the generating function of genus-zero closed Gromov-Witten invariants of $\tcX$ with that of the disk invariants of
$(\cX, \cL,f)$. 
\begin{theorem}\label{thm:GenCorr}
For any $\tk \in \mu_{\fa\fm}$,  the correspondence
$$
    F^{\cX, (\cL, f)}_{\tk} (\btau_2, \sX) = \llangle \tgamma_{\tk}\rrangle^{\tcX,T_f}(\tbtau_2)
$$
holds under the relation $\ttau_a = \tau_a$ for $a = 1, \dots, R-3$ and $\ttau_{R-2} = \log \sX$.
\end{theorem}

\begin{proof}
Fix an effective class $\beta \in E(X)$ and $(d, \lambda) \in H_1(\cL; \bZ)$ with $d>0$. We denote $ \hat{\btau}_2 := \sum_{a = 1}^{R-3} \ttau_a \tu_a = \tbtau_2 - \ttau_{R-2}\tu_{R-2}$, which is the lift of $\btau_2$ chosen as in Convention \ref{conv:LiftOuter} under the identification $\ttau_a = \tau_a$ for $a = 1, \dots, R-3$. We have 
\begin{align*}
\sum_{l \in \bZ_{\ge 0}} \frac{1}{l!} \langle \tbtau_2^l, \tgamma_{\tk} \rangle^{\tcX, T_f}_{\iota_*(\beta+d[B])}  
& = \sum_{l \in \bZ_{\ge 0}}  \frac{1}{l!}  \inner{(\hat{\btau}_2+\ttau_{R-2}\tu_{R-2})^l, \tgamma_{\tk} }^{\tcX, T_f}_{\iota_*(\beta+d[B])} \\
    &= \sum_{l \in \bZ_{\ge 0}} \sum_{k = 0}^l \frac{(d\ttau_{R-2})^k \inner{ (\hat{\btau}_2 )^{l-k}, \tgamma_{\tk}}^{\tcX, T_f}_{\iota_*(\beta+d[B])} }{k!(l-k)!} \\
    &= \sum_{l \in \bZ_{\ge 0}} \sum_{k = 0}^l \frac{(d\ttau_{R-2})^k \inner{\btau_2^{l-k} }^{\cX, (\cL, f)}_{\beta+d[B], (d, \lambda)} }{k!(l-k)!} \\
    &= \sum_{k,m \in \bZ_{\ge 0}} \frac{(d\ttau_{R-2})^k}{k!} \cdot \frac{ \inner{\btau_2^m}^{\cX, (\cL, f)}_{\beta+d[B], (d, \lambda)}}{m!}  \\
    &= \sum_{m \in \bZ_{\ge 0}} \frac{1}{m!} \inner{\btau_2^m}^{\cX, (\cL, f)}_{\beta+d[B], (d, \lambda)} \exp(d\ttau_{R-2})
\end{align*}
where the second equality follows from the divisor equation applied with $\inner{\fa\tD_{R+1}, \iota_*(\beta+d[B])} = d$, and the third equality follows from the numerical correspondence (Theorem \ref{thm:NumericalOuter}). The final line is then identified with
$$
    \sum_{m \in \bZ_{\ge 0}}  \frac{1}{m!} \inner{\btau_2^m}^{\cX, (\cL, f)}_{\beta+d[B], (d, \lambda)} \sX^d
$$
under $\ttau_{R-2} = \log \sX$, or equivalently $\tQ_{R-2}=\sX$. 
\end{proof}

\subsection{Equivariant small quantum cohomology of \texorpdfstring{$\tcX$}{X}}
As an application of Theorem \ref{thm:GenCorr}, we show in Section \ref{sect:JPairing} that $F^{\cX, (\cL, f)}_{\tk}$ can be recovered from the $\tT'$-equivariant small $J$-function of $\tcX$. In preparation, we review the $\tT'$-equivariant small quantum cohomology of $\tcX$ in this section. We refer to \cite[Section 2]{CIJ18} for the equivariant quantum cohomology of a general smooth Deligne-Mumford stack that has a semi-projective coarse moduli space and admits a torus action, as well as \cite{Givental96, Pandharipande98, Iritani09} for additional details.

\subsubsection{Quantum product}  

Let $(-,-)_{\tcX}^{\tT'}$ denote the $\tT'$-equivariant orbifold Poincar\'e pairing of $\tcX$, defined as
$$
    (\tu,\tv)_{\tcX}^{\tT'} := \int_{\cI \tcX} \tu \cup \inv^*(\tv), \quad \tu, \tv \in H^*_{\CR, \tT'}(\tcX; \bQ).
$$
Here, the integral is defined by $\tT'$-equivariant localization \cite{AB84} on $\cI \tcX$ and takes value in $\cQ_{\tT'}$. 

The $\tT'$-equivariant small \emph{quantum product} of $\tcX$ at $\tbtau_2$ is an associative, commutative product 
$\star_{\tbtau_2}$ on
$$
    H^*_{\CR, \tT'}(\tcX; \bC) \otimes_{\cR^{\bC}_{\tT'}} \cR^{\bC}_{\tT'}\llbracket \tQ,  \tbtau_2'' \rrbracket
$$
defined by 
$$
    a \star_{\tbtau_2} b := \sum_{\tbeta \in E(\tX)} \sum_{n \in \bZ_{\ge 0}} \frac{1}{n!} \inv^* \ev_{3,*} \left(\ev_1^*(a)\ev_2^*(b) \prod_{i=4}^{n+3} \ev_i^*(\tbtau_2) \cap [\Mbar_{0,n+3}(\tcX, \tbeta)]^\vir \right).
$$
Note that the semi-projectivity of $\tX$ ensures that $\ev_3: \Mbar_{0,n+3}(\tcX, \tbeta) \to \cI \tcX$ is proper. Equivalently, $\star_{\tbtau_2}$ is defined by
$$
    (a \star_{\tbtau_2} b, c)_{\tcX}^{\tT'} =  \llangle a, b, c\rrangle^{\tcX,\tT'} (\tbtau_2)
$$
where the right hand side is defined in Definition  \ref{def:Correlator}. 

\subsubsection{Quantum differential equations}

Let $A=\dim_{\bQ} H^*_{\CR}(\tX;\bQ)$. Then 
$$
H^*_{\CR,\tT'}(\tcX;\bC)\otimes_{\cR^{\bC}_{\tT'}} \cR^{\bC}_{\tT'} \llbracket \tQ, \tbtau_2''\rrbracket
$$
is a free module of rank $A$ over the ring $\cR_{\tT'}^{\bC}\llbracket \tQ, \tbtau_2''\rrbracket$, so it defines a vector
bundle $E$  of rank $A$ over the formal scheme $\hat{H}:= \Spec \cR_{\tT'}^{\bC} \llbracket \tQ, \tbtau_2''\rrbracket$.  The $\tT'$-equivariant small \emph{quantum connection} is a family
of flat connections $\{ \nabla^z: z\in \bC^*\}$ on $E$ defined by 
$$
    \nabla^z_{\frac{\partial}{\partial \ttau_a} } s = \frac{\partial s}{\partial \ttau_a}  - \frac{1}{z} \tu_a\star_{\tbtau_2}s, \quad a=1,\ldots, R-2.
$$
Flat sections of $\nabla^z$ are solutions of $\tT'$-equivariant small \emph{quantum differential equations}
\begin{equation} \label{eqn:QDE}
\frac{\partial s}{\partial \ttau_a} =   \frac{1}{z}\tu_a\star_{\tbtau_2} s, \quad a=1,\ldots, R-2,
\end{equation} 
which can be rewritten as 
\begin{equation} \label{eqn:QDE2}
\begin{aligned}
\tQ_a \frac{\partial s}{\partial \tQ_a} = &  \frac{1}{z}\tu_a\star_{\tbtau_2} s, \quad a=1,\ldots,  R'-3, R-2; \\
\frac{\partial s}{\partial \ttau_a} = &  \frac{1}{z}\tu_a\star_{\tbtau_2} s, \quad a=R'-2,\ldots, R-3. 
\end{aligned} 
\end{equation} 

\subsubsection{The fundamental solution} 
The $\tT'$-equivariant small quantum differential equations \eqref{eqn:QDE2}  are defined over the ring $\cR_{\tT'}^{\bC} \llbracket \tQ, \tbtau_2''\rrbracket$, and
have regular singularities along $\tQ_a=0$. We now describe a fundamental solution  to \eqref{eqn:QDE2} which is defined over the larger ring 
$$
\cR_{\tT'}^{\bC}[\tbtau_2']\llbracket z^{-1}\rrbracket \llbracket \tQ, \tbtau_2''\rrbracket 
:= \cR_{\tT'}^{\bC}[\ttau_1,\dots, \ttau_{R'-3}, \ttau_{R-2}] \llbracket z^{-1}\rrbracket \llbracket \tQ_1,\dots, \tQ_{R'-3}, \tQ_{R-2}, \ttau_{R'-2},\dots, \ttau_{R-3} \rrbracket
$$
(see \cite[Section 3.2]{CIJ18}). The \emph{$S$-operator}
$$
S(\tbtau_2,z) \in \End(H^*_{\CR,\tT'}(\cX;\bC))\otimes_{\cR_{\tT'}^\bC}   \cR_{\tT'}^{\bC}[\tbtau_2']\llbracket z^{-1}\rrbracket \llbracket \tQ, \tbtau_2''\rrbracket 
$$
is defined as follows. For any $\tu,\tv\in H^*_{\CR,\tT'}(\cX;\bC)$, 
$$
\left(\tu, S(\tbtau_2,z \right) \tv )^{\tT'}_{\tcX} =(\tu,\tv)^{\tT'}_{\tcX} +\llangle \tu,\frac{\tv}{z-\bar{\psi}}\rrangle^{\tcX,\tT'}(\tbtau_2)
$$
where 
$$
\llangle \tu,\frac{\tv}{z-\bar{\psi}}\rrangle^{\tcX,\tT'} =\sum_{k=0}^\infty z^{-k-1} \llangle \tu, \tv \bar{\psi}^k \rrangle^{\tcX,\tT'} (\tbtau_2).
$$
More explicitly,  we complete $\{\tu_1, \dots, \tu_{R-2}\}$ into a homogeneous basis $\{\tu_1, \dots, \tu_{A}\}$ of $H^*_{\CR, \tT'}(\tcX; \bQ)$ over $\cQ_{\tT'}$, and
let $\{\tu^1, \dots, \tu^A\}$ be the basis dual to $\{\tu_1, \dots, \tu_{A}\}$ under the $\tT'$-equivariant Poincar\'{e} pairing  $(-,-)_{\tcX}^{\tT'}$. Then
$$
    S(\tbtau_2, z) \tv  =  \tv +  \sum_{a=1}^A \llangle \tu_a, \frac{\tv}{z - \bar{\psi}}\rrangle_{\tbeta}^{\tcX,\tT'}(\tbtau_2)  \tu^a. 
$$

As observed by \cite[Proposition 2.4]{CIJ18}, the following is a straightforward equivariant generalization of results in \cite{Givental96, Pandharipande98, Iritani09}:
\begin{proposition}[\cite{Givental96, Pandharipande98, Iritani09, CIJ18}]\label{prop:QDESolution}
For any $\tv \in H^*_{\CR, \tT'}(\tcX; \bC)$, $S(\tbtau_2, z) \tv$ is a solution to the $\tT'$-equivariant quantum differential equation \eqref{eqn:QDE}: 
$$
\frac{\partial}{\partial \ttau_a} \cS(\tbtau_2,z) \tv =  \frac{1}{z} \tu_a \star_{\tbtau_2} S(\tbtau_2, z) \tv,  \qquad a = 1, \dots, R-2.
$$
\end{proposition}

Note that in the large radius limit $\tQ\to 0$, we have
$$
\lim_{\tQ\to 0} S(\tbtau_2,z)  \tv = e^{\tbtau_2/z}\tv,
$$
which is a solution to the following classical limit of the quantum differential equations
$$
\frac{\partial s}{\partial \ttau_a} = \frac{1}{z} \tu_a s, \qquad a=1,\ldots, R-2. 
$$


\subsection{Equivariant small \texorpdfstring{$J$}{J}-function of \texorpdfstring{$\tcX$}{X} and disk potential of \texorpdfstring{$(\cX,\cL,f)$}{X}}\label{sect:JPairing}
The $\tT'$-equivariant small $J$-function \cite{Tseng10, CG07, Givental98} of $\tcX$
$$
J_{\tcX}^{\tT'}(\btau_2,z)  
$$
is characterized by 
$$
(J_{\tcX}^{\tT'}(\btau_2,z),\tv)^{\tT'}_{\tcX}  = (\one,S(\btau_2,z)\tv )^{\tT'}_{\tcX} 
$$
for any $\tv\in H_{\tT'}^*(\tcX;\bC)$; in other words, it is defined by 
$$
    J_{\tcX}^{\tT'}(\tbtau_2, z) :=  \one + \sum_{a=1}^A \llangle \one,  \frac{\tu_a}{z - \bar{\psi}} \rrangle^{\tcX,\tT'}(\tbtau_2) \tu^a,
$$
which takes value in $H^*_{\CR, \tT'}(\tcX; \bC)$. 

\begin{notation}\label{not:ZCoefficient}
\rm{
Let $\bS$ be a commutative ring,
and let
$$
h(z) = \sum_{n=0}^\infty h_n z^{-n} \in \bS\llbracket z^{-1} \rrbracket 
$$
where $h_n\in \bS$.  For any $n\in \bZ_{\geq 0}$, define
$$
[z^{-n}] h(z) = h_n.
$$
}\end{notation}

The following lemma follows immediately from the definition and the string equation. 
\begin{lemma}\label{lem:JPairing}
For any $\tv\in H^*_{\CR,\tT'}(\tcX;\bC)$, 
$$
(J_{\tcX}^{\tT'}(\tbtau_2,z), \tv)^{\tT'}_{\tcX} = (1,\tv)^{\tT'}_{\tcX} + z^{-1} (\tbtau_2, \tv)^{\tT'}_{\tcX}  + \sum_{k=0}^\infty z^{-2-k} \llangle \tv \bar{\psi}^k\rrangle^{\tcX,\tT'}(\tbtau_2). 
$$
In particular, 
\begin{equation}\label{eqn:primary} 
[z^{-2}] (J_{\tcX}^{\tT'}(\tbtau_2,z), \tv)^{\tT'}_{\tcX}  = \llangle \tv\rrangle^{\tcX,\tT'}(\tbtau_2).
\end{equation} 
\end{lemma}


Under the open/closed correspondence for generating functions (Theorem \ref{thm:GenCorr}), \eqref{eqn:primary}  immediately implies that the generating function $F^{\cX, (\cL, f)}(\btau_2, \sX)$ of disk invariants of $(\cX, \cL, f)$ can be retrieved from the $\tT'$-equivariant small $J$-function of $\tcX$.

\begin{theorem}\label{thm:JPairing}
For any $\tk \in \mu_{\fa\fm}$, 
$$
    F^{\cX, (\cL, f)}_{\tk} (\btau_2, \sX) = [z^{-2}]  \left(J_{\tcX}^{\tT'}(\tbtau_2, z), \tgamma_{\tk} \right)_{\tcX}^{\tT'} \bigg|_{\su_4 = 0, \su_2 - f\su_1 = 0}
$$
under the relation $\ttau_a = \tau_a$ for $a = 1, \dots, R-3$ and $\ttau_{R-2} = \log \sX$.
\end{theorem}

\color{black}


\section{B-model correspondence and mirror symmetry}\label{sect:BModel}
In this section, we develop the open/closed correspondence on the B-model side. We first establish the counterpart of Theorem \ref{thm:JPairing} that the B-model disk function of $(\cX, \cL, f)$ \cite{FL13, FLT12} can be recovered from the $\tT'$-equivariant $I$-function of $\tcX$ (Theorem \ref{thm:IPairing}). Furthermore, we show that the A- and B-model correspondences are compatible with the closed mirror symmetry for $\tcX$ and the open mirror symmetry of $(\cX, \cL, f)$.


\subsection{Extended Nef cones and Mori cones}

Recall that we defined the extended nef cone of $\cX$ in Section \ref{sect:AVBrane}. We now define the extended Mori cone of $\cX$. For each maximal cone $\sigma \in \Sigma(3)$, define
$$
   \bK_{\sigma}^{\vee} := \sum_{i \in I_{\sigma}} \bZ D_i,
$$
which is a sublattice of $\bL^\vee$ of finite index. Let
$$
    \bK_{\sigma} := \{ \beta \in \bL_{\bQ} : \inner{D, \beta} \in \bZ \text{ for all } D \in \bK_{\sigma}^\vee\}
$$
be the dual lattice of $\bK_{\sigma}^{\vee}$ viewed as an overlattice of $\bL$ in $\bL_{\bQ}$, where $\inner{-,-}$ denotes the pairing between $\bL_{\bQ}^\vee$ and $\bL_{\bQ}$. The map
$$
    v: \bK_{\sigma} \to N, \quad \beta \mapsto \sum_{i = 1}^R \ceil{\inner{D_i, \beta}} b_i
$$
induces a bijection $\bK_{\sigma} / \bL \to \Box(\sigma)$, which we also refer to as $v$ by abusive notation. Let
$$
    \bK(\cX) := \bigcup_{\sigma \in \Sigma(3)} \bK_{\sigma}.
$$

For each $\sigma \in \Sigma(3)$, define the \emph{extended $\sigma$-Mori cone} as
$$
    \tNE(\sigma) := \{ \beta \in \bL_{\bR} : \inner{D, \beta} \ge 0 \text{ for all } D \in \tNef(\sigma) \},
$$
which is the dual cone of $\tNef(\sigma)$. Let
$$
    \bK_{\eff, \sigma} := \bK_{\sigma} \cap \tNE(\sigma).
$$

The \emph{extended Mori cone} of $\cX$ is defined to be
$$
    \tNE(\cX) := \bigcup_{\tsi \in \tSi(4)} \tNE(\tsi).
$$
Moreover, define
$$
    \bK_{\eff}(\tcX) := \bK \cap \tNE(\cX) = \bigcup_{\sigma \in \Sigma(3)} \bK_{\eff, \sigma}.
$$

We now give an analog of the above definitions for $\tcX$. Let $\tD_i \in \tbL^\vee$, $i = 1, \dots, R+2$ be defined as in \eqref{eqn:tDiOuter}. For each maximal cone $\tsi \in \tSi(4)$, define the \emph{extended $\tsi$-nef cone} as
$$
    \tNef(\tsi):= \sum_{i \in I_{\tsi}} \bR_{\ge 0} \tD_i.
$$
The \emph{extended nef cone} of $\tcX$ is defined to be
$$
    \tNef(\tcX):= \bigcap_{\tsi \in \tSi(4)} \tNef(\tsi),
$$
which is an $(R-2)$-dimensional simplicial cone in $\tbL_{\bR}^\vee$. Under the projection $\tbL^\vee \to \bL^\vee$, the image of $\tNef(\tcX)$ is $\tNef(\cX)$.

For each $\tsi \in \tSi(4)$, define
$$
   \bK_{\tsi}^{\vee} := \sum_{i \in I_{\tsi}} \bZ \tD_i,
$$
which is a sublattice of $\tbL^\vee$ of finite index. Let
$$
    \bK_{\tsi} := \{ \tbeta \in \tbL_{\bQ} : \inner{\tD, \tbeta} \in \bZ \text{ for all } \tD \in \bK_{\tsi}^\vee\}
$$
be the dual lattice of $\bK_{\tsi}^{\vee}$ viewed as an overlattice of $\tbL$ in $\tbL_{\bQ}$, where $\inner{-,-}$ denotes the pairing between $\tbL_{\bQ}^\vee$ and $\tbL_{\bQ}$. The map
$$
    \tv: \bK_{\tsi} \to \tN, \quad \tbeta \mapsto \sum_{i = 1}^{R+2} \ceil{\inner{\tD_i, \tbeta}} \tb_i
$$
induces a bijection $\bK_{\tsi} / \tbL \to \Box(\tsi)$, which we also refer to as $\tv$ by abusive notation. Let
$$
    \bK(\tcX) := \bigcup_{\tsi \in \tSi(4)} \bK_{\tsi}.
$$

For each $\tsi \in \tSi(4)$, define the \emph{extended $\tsi$-Mori cone} as
$$
    \tNE(\tsi) := \{ \tbeta \in \tbL_{\bR} : \inner{\tD, \tbeta} \ge 0 \text{ for all } \tD \in \tNef(\tsi) \},
$$
which is the dual cone of $\tNef(\tsi)$. Let
$$
    \bK_{\eff, \tsi} := \bK_{\tsi} \cap \tNE(\tsi).
$$

The \emph{extended Mori cone} of $\tcX$ is defined to be
$$
    \tNE(\tcX) := \bigcup_{\tsi \in \tSi(4)} \tNE(\tsi).
$$
Moreover, define
$$
    \bK_{\eff}(\tcX) := \bK(\tcX) \cap \tNE(\tcX) =\bigcup_{\tsi \in \tSi(4)} \bK_{\eff, \tsi}.
$$

We now make a few observations to be used later.

\begin{observation}\label{obs:ClassInL}
Given $\tbeta \in \tbL_{\bQ}$ that is contained in $\bL_{\bQ}$, we have
$$
    \inner{\tD, \tbeta} = \inner{D, \tbeta}
$$
for any $\tD \in \tbL^\vee$ projecting to $D \in \bL^\vee$. Moreover, the following hold:
\begin{itemize}
    \item $\inner{\tD_i, \tbeta} = \inner{D_i, \tbeta}$, $i = 1, \dots, R$.

    \item $\inner{\tD_{R+1}, \tbeta} = \inner{\tD_{R+2}, \tbeta} = 0$.

    \item If $\tbeta \in \bK_{\eff}(\tcX)$, then $\tv(\tbeta) \in N \subset \tN$ and agrees with $v(\tbeta)$. In particular, $\age(\tv(\tbeta)) \le 1$.
\end{itemize}
\end{observation}

\begin{lemma}\label{lem:DRPlusOne}
We have
$$
    \tD_{R+1} \in \tNef(\tcX).
$$
\end{lemma}

\begin{proof}
For any $\tsi \in \iota(\Sigma(3))$, we have $R+1 \in I_{\tsi}$ and thus $\tD_{R+1} \in \tNef(\tsi)$. Now we consider cones in $\tSi(4) \setminus \iota(\Sigma(3))$. Note that the exactness of the first row of \eqref{eqn:tXDualSES} implies that
\begin{equation}\label{eqn:DRPlusOneLinComb}
    \fa\tD_{R+1}= \sum_{i = 1}^{R} m_i \tD_i = \sum_{\substack{i \in \{1, \dots, R\} \\ m_i >0}} m_i \tD_i.
\end{equation}
Moreover, for any $\tsi \in \tSi(4) \setminus \iota(\Sigma(3))$, we have
$$
    I_{\tsi} \supseteq \{i \in \{1, \dots, R\} : m_i >0\}.
$$
This implies that $\tD_{R+1} \in \tNef(\tsi)$.
\end{proof}

\subsection{Equivariant \texorpdfstring{$I$}{I}-function of \texorpdfstring{$\tcX$}{X} and B-model correspondence} \label{sec:I} 

We choose elements
$$
    H_1, \dots, H_{R-3} \in \bL^\vee \cap \tNef(\cX), \quad \tH_1, \dots, \tH_{R-2} \in \tbL^\vee \cap \tNef(\tcX)
$$
that satisfy the following conditions:
\begin{itemize}
    \item $\{H_1, \dots, H_{R-3}\}$ is a $\bQ$-basis for $\bL_{\bQ}^\vee$. $\{\tH_1, \dots, \tH_{R-2}\}$ is a $\bQ$-basis for $\tbL_{\bQ}^\vee$.

    \item The images of $H_1, \dots, H_{R'}$ under the \emph{Kirwan map}
    $$
        \kappa: \bL^\vee \cong H^2_G(\bC^R; \bZ) \to H^2(\cX; \bZ)
    $$
    form a $\bQ$-basis for $H^2(\cX; \bQ)$. The images of $\{\tH_1, \dots, \tH_{R'}, \tH_{R-2}\}$ under
    $$
        \tkappa: \tbL^\vee \cong H^2_{\tG}(\bC^{R+2}; \bZ) \to H^2(\cX; \bZ)
    $$
    is a $\bQ$-basis for $H^2(\tcX; \bQ)$.

    \item For each $a = 1, \dots, R'-3$, $\tkappa(\tH_a)$ is the lift of $\kappa(H_a)$ chosen as in Convention \ref{conv:LiftOuter}. In particular, $\tH_a$ projects to $H_a$ under $\tbL^\vee \to \tbL$.

    \item For each $a = R'-2, \dots, R-3$, $H_a = D_{3+a}$ and $\tH_a = \tD_{3+a}$.

    \item $\tH_{R-2} = \fa\tD_{R+1}$ (see Lemma \ref{lem:DRPlusOne}).
\end{itemize}
Recall from Section \ref{sect:GenCorr} that we used bases $\{u_1, \dots, u_{R-3}\}$ for $H^2_{\CR}(\cX; \bQ)$, $\{\tu_1, \dots, \tu_{R-2}\}$ for $H^2_{\CR}(\tcX; \bQ)$ when defining generating functions of Gromov-Witten invariants. We now fix the choices
\begin{align*}
    u_a &= \kappa(H_a), && a = 1, \dots, R'-3,\\
    \tu_a &= \tkappa(\tH_a), && a = 1, \dots, R'-3, R-2,\\
    u_a & = \tu_a = \one_{j(3+a)}, && a = R'-2, \dots, R-3
\end{align*}
where $j(3+a) \in \Box(\cX) \subseteq \Box(\tcX)$ is represented by $b_{3+a}$ (or $\tb_{3+a}$). As before, for $a = 1, \dots, R-2$, we choose $T'$-equivariant lifts of $u_a$ and $\tT'$-equivariant lifts of $\tu_a$ as in Convention \ref{conv:LiftOuter}.

Let $q = (q_1, \dots, q_{R-3})$, $\tq = (\tq_1, \dots, \tq_{R-2})$ be formal variables. For each $\beta \in \bK(\cX)$, $\tbeta \in \bK(\tcX)$, we set
$$
    q^\beta := q_1^{\inner{H_1, \beta}} \cdots q_{R-3}^{\inner{H_{R-3}, \beta}}, \quad 
    \tq^{\tbeta} := \tq_1^{\inner{\tH_1, \tbeta}} \cdots \tq_{R-2}^{\inner{\tH_{R-2}, \tbeta}}.
$$

\subsubsection{B-model disk function of $(\cX,\cL,f)$}
Recall from \eqref{eqn:Sigma0WtsOuter} that
$$
        w_0 = \frac{1}{\fr}, \quad w_2 = \frac{\fs + \fr f}{\fr\fm}, \quad w_3 = -\frac{\fm + \fs + \fr f}{\fr\fm}.
$$
Following \cite{FLT12}, we set
$$
    \bK_{\eff}(\cX, \cL) := \{ (\beta, d) \in  \bK_{\eff, \sigma_0} \times \bZ : \inner{D_1, \beta} + dw_0 \in \bZ_{\ge 0}, d \neq 0 \}.
$$

For each $(\beta, d) \in \bK_{\eff}(\cX, \cL)$, there exists a unique $\lambda(\beta, d) \in G_{\tau_0}$ such that
$$
    v(\beta) = h(d, \lambda(\beta,d)).
$$
We set
$$
    \tk(\beta, d) := \th(d, \lambda(\beta,d)).
$$

Let $x$ be a formal variable. The \emph{B-model disk function} of $(\cX, \cL, f)$ \cite{FL13, FLT12} is defined as follows.

\begin{definition}\label{def:WOuter}\rm{
For any $\tk \in \mu_{\fa\fm}$, define
$$
    W^{\cX, (\cL, f)}_{\tk}(q, x) = \sum_{\substack{(\beta, d) \in \bK_{\eff}(\cX, \cL)\\ \tk = \tk(\beta,d)}} q^\beta x^d \frac{(-1)^{\floor{dw_3-\epsilon_3}+ \ceil{\frac{d}{\fa}}}}{\fm d (\inner{D_1, \beta} + dw_0)! \prod_{i = 4}^R \inner{D_i, \beta}!} \cdot \frac{\prod_{m = 1}^\infty (-\inner{D_3, \beta} - dw_3 - m)}{\prod_{m = 0}^\infty (\inner{D_2, \beta} + dw_2 - m)},
$$
which takes value in $\bC$.
}\end{definition}



\begin{remark}\label{rem:WSign} \rm{
We note that when $f \in \bZ$, i.e. $\fa=1$, the sign convention of $W^{\cX, (\cL, f)}(q, x)$ above differs from that in \cite{FLT12}, yet agrees with that in \cite{FL13} in the smooth case. This is to be consistent with our sign convention of the disk invariants. See Remark \ref{rem:DiskSignOuter}.
}\end{remark}

\begin{remark}\label{rem:BmodelGrouping}
\rm{
Similar to Remark \ref{rem:AmodelGrouping}, we note that in \cite{FLT12}, the B-model disk functions are formed by grouping classes $(\beta, d)$ according to $\lambda(\beta, d)$. This is consistent with Definition \ref{def:WOuter} in the case $f \in \bZ$, i.e. $\fa=1$, where $\th(d,\lambda) = \lambda$ for all $d, \lambda$.
}
\end{remark}

\subsubsection{Equivariant $I$-function of $\tcX$}
Let $\tq_0$ be a formal variable. The $\tT'$-equivariant \emph{$I$-function} of $\tcX$ is defined as
\begin{align*}
    I_{\tcX}^{\tT'}&(\tq_0, \tq, z) :=  e^{\frac{1}{z}\left(\tq_0 + \sum_{a \in \{1, \dots, R'-3, R-2\}}\tu_a \log \tq_a \right)} \\
    & \cdot \sum_{\tbeta \in \bK_{\eff}(\tcX)} \tq^{\tbeta} \prod_{i \in \{1, \dots, R', R+1, R+2\}} \frac{\prod_{m = \ceil{\inner{\tD_i, \tbeta}}}^\infty (\tcD_i^{\tT'} + (\inner{\tD_i, \tbeta} - m)z)}{\prod_{m = 0}^\infty (\tcD_i^{\tT'} + (\inner{\tD_i, \tbeta} - m)z)}
    \cdot \prod_{i = R'+1}^R \frac{\prod_{m = \ceil{\inner{\tD_i, \tbeta}}}^\infty (\inner{\tD_i, \tbeta} - m)z}{\prod_{m = 0}^\infty (\inner{\tD_i, \tbeta} - m)z} \one_{\tv(\tbeta)}\\
    & = e^{\frac{1}{z}\left(\tq_0 + \sum_{a \in \{1, \dots, R'-3, R-2\}}\tu_a \log \tq_a \right)} \\
    & \cdot \sum_{\tbeta \in \bK_{\eff}(\tcX)} \tq^{\tbeta} \prod_{i \in \{1, \dots, R', R+1, R+2\}} \frac{\prod_{m = \ceil{\inner{\tD_i, \tbeta}}}^\infty (\frac{\tcD_i^{\tT'}}{z} + \inner{\tD_i, \tbeta} - m)}{\prod_{m = 0}^\infty (\frac{\tcD_i^{\tT'}}{z} + \inner{\tD_i, \tbeta} - m)}
    \cdot \prod_{i = R'+1}^R \frac{\prod_{m = \ceil{\inner{\tD_i, \tbeta}}}^\infty (\inner{\tD_i, \tbeta} - m)}{\prod_{m = 0}^\infty (\inner{\tD_i, \tbeta} - m)} \frac{\one_{\tv(\tbeta)}}{z^{\age(\tv(\tbeta))}},
\end{align*}
which takes value in $H^*_{\CR, \tT'}(\tcX; \bQ)$. Here the second equality follows from that $\tcX$ is Calabi-Yau, and thus $\sum_{i=1}^{R+2} \tD_i = 0$. We set
$$
    I_{\tcX}^{\tT'}(\tq, z) := I_{\tcX}^{\tT'}(0, \tq, z).
$$

\subsubsection{B-model correspondence}
We now establish the following B-model version of the open/closed correspondence, which says that the disk function $W^{\cX, (\cL, f)}(q, x)$ can be retrieved from the $\tT'$-equivariant $I$-function of $\tcX$.

\begin{theorem}\label{thm:IPairing}
For any $\tk \in \mu_{\fa\fm}$,  
$$
W^{\cX, (\cL, f)}_{\tk}(q, x) = [z^{-2}] \left(I_{\tcX}^{\tT'}(\tq, z), \tgamma_{\tk} \right)_{\tcX}^{\tT'} \bigg|_{\su_4 = 0, \su_2 - f\su_1 = 0}
$$
under the relation $\tq_a = q_a$ for $a = 1, \dots, R-3$ and $\tq_{R-2} = x$.
\end{theorem}

Note that Theorem \ref{thm:IPairing} is the B-model analog of Theorem \ref{thm:JPairing}. See Section \ref{sect:Mirror} below for additional discussions in the context of mirror symmetry.

We will prove Theorem \ref{thm:IPairing} in two steps. First, in Lemma \ref{lem:BModelContribute}, we identify $W^{\cX, (\cL, f)}_{\tk}(q, x)$ with the
pairing of $\tgamma_{\tk}$ with the terms in $[z^{-2}] I_{\tcX}^{\tT'}(\tq, z)$ corresponding to classes in
$$
    \bK_{\eff, \tsi_0} \setminus \bL_{\bQ}.
$$
Second, in Lemmas \ref{lem:BModelVanish} and \ref{lem:BModelLVanish}, we show that terms corresponding to classes outside $\bK_{\eff, \tsi_0} \setminus \bL_{\bQ}$ do not contribute after pairing with $\tgamma_{\tk}$ and taking weight restrictions.

\subsection{Proof of Theorem \ref{thm:IPairing}}

Recall that 
$\iota_{\tsi}^*(\tgamma_{\tk}) = 0$ for any $\tsi \in \tSi(4)$, $\tsi \neq \tsi_0$. For each $\tbeta$, we set
$$
    I_{\tbeta}(\tq, z) := \tq^{\tbeta} \prod_{i \in \{1, \dots, R', R+1, R+2\}} \frac{\prod_{m = \ceil{\inner{\tD_i, \tbeta}}}^\infty (\frac{\tcD_i^{\tT'}}{z} + \inner{\tD_i, \tbeta} - m)}{\prod_{m = 0}^\infty (\frac{\tcD_i^{\tT'}}{z} + \inner{\tD_i, \tbeta} - m)}
    \cdot \prod_{i = R'+1}^R \frac{\prod_{m = \ceil{\inner{\tD_i, \tbeta}}}^\infty (\inner{\tD_i, \tbeta} - m)}{\prod_{m = 0}^\infty (\inner{\tD_i, \tbeta} - m)} \frac{\one_{\tv(\tbeta)}}{z^{\age(\tv(\tbeta))}}.
$$
That is,
$$
    I_{\tcX}^{\tT'}(\tq, z) :=  e^{\frac{1}{z}\left(\sum_{a \in \{1, \dots, R'-3, R-2\}}\tu_a \log \tq_a \right)} \cdot \sum_{\tbeta \in \bK_{\eff}(\tcX)} I_{\tbeta}(\tq, z).
$$

\begin{lemma}\label{lem:BModelContribute}
For any $\tk \in \mu_{\fa\fm}$, 
$$
    W^{\cX, (\cL, f)}_{\tk}(q, x) = [z^{-2}] \left(\sum_{\tbeta \in \bK_{\eff, \tsi_0} \setminus \bL_{\bQ}} I_{\tbeta}(\tq, z), \tgamma_{\tk} \right)_{\tcX}^{\tT'} \bigg|_{\su_4 = 0, \su_2 - f\su_1 = 0}
$$
under the relation $\tq_a = q_a$ for $a = 1, \dots, R-3$ and $\tq_{R-2} = x$.
\end{lemma}

See Notation \ref{not:ZCoefficient} for the notation $[z^{-2}]$.

\begin{proof}
We first set up a one-to-one correspondence between $\bK_{\eff, \tsi_0} \setminus \bL_{\bQ}$ and $\bK_{\eff}(\cX, \cL)$. Let
\begin{equation}\label{eqn:lZero}
    l^{(0)} := \left(w_0, w_2, w_3, 0, \dots, 0, \frac{1}{\fa}, -\frac{1}{\fa}\right) \in \bQ^{R+2},
\end{equation}
which satisfies $\talpha(l^{(0)}) = 0$ (see \eqref{eqn:tXSES}) and does not belong to the span of $\{\tl^{(1)}, \dots, \tl^{(R-3)}\}$. We write
$$
    \tl^{(R-2)} = \fa\tl^{(R-2)}_{R+1} l^{(0)} + c_1 \tl^{(1)} + \cdots + c_{R-3} \tl^{(R-3)},
$$
where $\tl^{(R-2)}_{R+1}$ is the $(R+1)$-th component of $\tl^{(R-2)}$ and $c_1, \dots, c_{R-3} \in \bQ$. Then, we define a map
\begin{equation}\label{eqn:EffClassCorr}
    \tbL_{\bQ} \to \bL_{\bQ} \times \bQ,
\end{equation}
$$
    \tbeta = (\beta_1, \dots, \beta_{R-2}) \mapsto (\beta, d) = ((\beta_1 + \beta_{R-2}c_1, \dots, \beta_{R-3} + \beta_{R-2}c_{R-3}), \beta_{R-2}\fa\tl^{(R-2)}_{R+1}).
$$
Note that
\begin{equation}\label{eqn:dAsPairing}
    d = \inner{\fa\tD_{R+1}, \tbeta}.
\end{equation}
If $\tbeta \in \bK_{\eff, \tsi_0} \setminus \bL_{\bQ}$, we have $d \neq 0$, and $\inner{\tD_i, \tbeta} \in \bZ_{\ge 0}$ for all $i \in I_{\tsi_0} = \{1, 4, \dots, R\}$. Then \eqref{eqn:DRPlusOneLinComb} gives
$$
    d = \inner{\fa\tD_{R+1}, \tbeta} = \inner{m_1\tD_1 + \sum_{i=4}^R m_i \tD_i, \tbeta} \in \bZ_{\ge 0}.
$$
This verifies that the image of any class in $\bK_{\eff, \tsi_0} \setminus \bL_{\bQ}$ satisfies $d \in \bZ_{\neq 0}$. Now denote $l^{(0)} = (l^{(0)}_1, \dots, l^{(0)}_{R+2})$. It is straightforward to check that for any $i = 1, \dots, R$,
\begin{equation}\label{eqn:PairingCorr}
    \inner{\tD_i, \tbeta} = \inner{D_i, \beta} + dl^{(0)}_i.
\end{equation}
Therefore, \eqref{eqn:EffClassCorr} induces a bijection $\bK_{\eff, \tsi_0} \setminus \bL_{\bQ} \to \bK_{\eff}(\cX, \cL)$.  Moreover, from the definitions, we verify that for corresponding classes $\tbeta$ and $(\beta, d)$,
\begin{equation}\label{eqn:CompatibleTwist}
    -\tv(\tbeta) = \tk(\beta, d).
\end{equation}

From now on, we fix corresponding classes $\tbeta$ and $(\beta, d)$ under the above bijection such that both sides of \eqref{eqn:CompatibleTwist} are $\tk$. Otherwise $I_{\tbeta}(\tq, z)$ pairs to zero with $\tgamma_{\tk}$ and $(\beta, d)$ does not contribute to $W^{\cX, (\cL, f)}_{\tk}(q, x)$.
Note that for each $a = 1, \dots, R-3$, $H_a$ can be written as a linear combination of $D_4, \dots, D_R$ and $\tH_a$ can be written as a linear combination of $\tD_4, \dots, \tD_R$ with the same coefficients (Convention \ref{conv:LiftOuter}). Then \eqref{eqn:PairingCorr} and $l^{(0)}_4 = \cdots = l^{(0)}_R = 0$ imply that
$$
    \inner{\tH_a, \tbeta} = \inner{H_a, \beta}.
$$
This combined with \eqref{eqn:dAsPairing} gives that
$$
    \tq^{\tbeta} = q^\beta x^d
$$
under $\tq_a = q_a$ for $a = 1, \dots, R-3$ and $\tq_{R-2} = x$.

Now, we compute that
\begin{align*}
    \iota_{\tsi_0}^*(I_{\tbeta}(\tq, z)) &= \frac{\tq^{\tbeta}}{\prod_{i \in I_{\tsi_0}}\inner{\tD_i, \tbeta}!} \cdot
    \prod_{i \in \{2, 3, R+1, R+2\}} \frac{\prod_{m = \ceil{\inner{\tD_i, \tbeta}}}^\infty (\frac{\iota_{\tsi_0}^*(\tcD_i^{\tT'})}{z} + \inner{\tD_i, \tbeta} - m)}{\prod_{m = 0}^\infty (\frac{\iota_{\tsi_0}^*(\tcD_i^{\tT'})}{z} + \inner{\tD_i, \tbeta} - m)} \frac{\one_{\tv(\tbeta)}}{z^{\age(\tv(\tbeta))}}\\
    &= \frac{\tq^{\tbeta}}{(\inner{D_1, \beta} + dw_0)! \prod_{i = 4}^R \inner{D_i, \beta}!} \cdot
    \prod_{i \in \{2, 3, R+1, R+2\}} \frac{\prod_{m = \ceil{\inner{\tD_i, \tbeta}}}^\infty (\frac{\iota_{\tsi_0}^*(\tcD_i^{\tT'})}{z} + \inner{\tD_i, \tbeta} - m)}{\prod_{m = 0}^\infty (\frac{\iota_{\tsi_0}^*(\tcD_i^{\tT'})}{z} + \inner{\tD_i, \tbeta} - m)} \frac{\one_{\tv(\tbeta)}}{z^{\age(\tv(\tbeta))}}
\end{align*}
where the second equality follows from \eqref{eqn:PairingCorr}. Recall that
$$
    \iota_{\tsi_0}^*(\tcD_2^{\tT'}) = -\frac{f}{\fm}\su_1 + \frac{1}{\fm} \su_2, \quad \iota_{\tsi_0}^*(\tcD_3^{\tT'}) =  \frac{f}{\fm}\su_1 - \frac{1}{\fm} \su_2 - \su_4, \quad \iota_{\tsi_0}^*(\tcD_{R+1}^{\tT'}) = -\frac{1}{\fa}\su_1, \quad \iota_{\tsi_0}^*(\tcD_{R+2}^{\tT'}) = \frac{1}{\fa}\su_1 + \su_4.
$$
For $i = R+1, R+2$, since $\inner{\tD_{R+1}, \tbeta} = -\inner{\tD_{R+2}, \tbeta} = \frac{d}{\fa} > 0$, we have
\begin{align*}
    & \frac{\prod_{m = \ceil{\inner{\tD_{R+1}, \tbeta}}}^\infty (\frac{\iota_{\tsi_0}^*(\tcD_{R+1}^{\tT'})}{z} + \inner{\tD_{R+1}, \tbeta} - m)}{\prod_{m = 0}^\infty (\frac{\iota_{\tsi_0}^*(\tcD_{R+1}^{\tT'})}{z} + \inner{\tD_{R+1}, \tbeta} - m)} \cdot \frac{\prod_{m = \ceil{\inner{\tD_{R+2}, \tbeta}}}^\infty (\frac{\iota_{\tsi_0}^*(\tcD_{R+2}^{\tT'})}{z} + \inner{\tD_{R+2}, \tbeta} - m)}{\prod_{m = 0}^\infty (\frac{\iota_{\tsi_0}^*(\tcD_{R+2}^{\tT'})}{z} + \inner{\tD_{R+2}, \tbeta} - m)} \bigg|_{\su_4 = 0}\\
    & = \frac{\prod_{m = 1}^{\floor{\frac{d}{\fa}}} (\frac{\su_1}{\fa z} - \frac{d}{\fa} + m)}{\prod_{m = 0}^{\ceil{\frac{d}{\fa}}-1} (\frac{-\su_1}{\fa z} + \frac{d}{\fa} -m)}\\ 
    &  = \begin{cases}
        \displaystyle{(-1)^{\frac{d}{\fa}-1}\frac{\su_1}{\fa z} \left(-\frac{\su_1}{\fa z}+\frac{d}{\fa}\right)^{-1} }& \text{if } \fa \mid d,\\
        \displaystyle{ (-1)^{\ceil{\frac{d}{\fa}}-1} \left(-\frac{\su_1}{\fa z}+\frac{d}{\fa}\right)^{-1} }& \text{if } \fa \nmid d.
    \end{cases}
\end{align*}
For $i = 2,3$, if $\tk=1$ in which case $\inner{\tD_2, \tbeta}, \inner{\tD_3, \tbeta} \in \bZ$, we have
\begin{align*}
    & \frac{\prod_{m = \ceil{\inner{\tD_2, \tbeta}}}^\infty (\frac{\iota_{\tsi_0}^*(\tcD_2^{\tT'})}{z} + \inner{\tD_2, \tbeta} - m)}{\prod_{m = 0}^\infty (\frac{\iota_{\tsi_0}^*(\tcD_2^{\tT'})}{z} + \inner{\tD_2, \tbeta} - m)} \cdot \frac{\prod_{m = \ceil{\inner{\tD_3, \tbeta}}}^\infty (\frac{\iota_{\tsi_0}^*(\tcD_3^{\tT'})}{z} + \inner{\tD_3, \tbeta} - m)}{\prod_{m = 0}^\infty (\frac{\iota_{\tsi_0}^*(\tcD_3^{\tT'})}{z} + \inner{\tD_3, \tbeta} - m)} \bigg|_{\su_4 = 0}\\
    & = (-1)^{\inner{\tD_3, \tbeta}-1} \frac{f\su_2 - \su_1}{\fm z} \cdot \frac{\prod_{m = 1}^\infty (\frac{\su_2-f\su_1}{\fm z} - \inner{\tD_3, \tbeta} - m)}{\prod_{m = 0}^\infty (\frac{\su_2-f\su_1}{\fm z} + \inner{\tD_2, \tbeta} - m)}\\
    & = (-1)^{\inner{D_3, \beta} + dw_3-1}  \frac{f\su_2 - \su_1}{\fm z} \cdot \frac{\prod_{m = 1}^\infty (\frac{\su_2-f\su_1}{\fm z} - \inner{D_3, \beta} - dw_3 - m)}{\prod_{m = 0}^\infty (\frac{\su_2-f\su_1}{\fm z} + \inner{D_2, \beta} + dw_2 - m)},
\end{align*}
where the last equality follows from \eqref{eqn:PairingCorr}. If otherwise $\tk \neq 1$, we have
\begin{align*}
    & \frac{\prod_{m = \ceil{\inner{\tD_2, \tbeta}}}^\infty (\frac{\iota_{\tsi_0}^*(\tcD_2^{\tT'})}{z} + \inner{\tD_2, \tbeta} - m)}{\prod_{m = 0}^\infty (\frac{\iota_{\tsi_0}^*(\tcD_2^{\tT'})}{z} + \inner{\tD_2, \tbeta} - m)} \cdot \frac{\prod_{m = \ceil{\inner{\tD_3, \tbeta}}}^\infty (\frac{\iota_{\tsi_0}^*(\tcD_3^{\tT'})}{z} + \inner{\tD_3, \tbeta} - m)}{\prod_{m = 0}^\infty (\frac{\iota_{\tsi_0}^*(\tcD_3^{\tT'})}{z} + \inner{\tD_3, \tbeta} - m)} \bigg|_{\su_4 = 0}\\
    & = (-1)^{\floor{\inner{D_3, \beta} + dw_3}-1} \frac{\prod_{m = 1}^\infty (\frac{\su_2-f\su_1}{\fm z} - \inner{D_3, \beta} - dw_3 - m)}{\prod_{m = 0}^\infty (\frac{\su_2-f\su_1}{\fm z} + \inner{D_2, \beta} + dw_2 - m)}.
\end{align*}

Summarizing the above computations, we see that under $\tq_a = q_a$ for $a = 1, \dots, R-3$ and $\tq_{R-2} = x$, the coefficient of  $z^{-2}$ in $\iota_{\tsi_0}^*(I_{\tbeta}(\tq, z)) \big|_{\su_4 = 0}$ is as follows: If $\tk=1$, it is
$$
    q^{\beta}x^d \frac{\su_1 (f\su_1 - \su_2)}{\fm} \cdot \frac{(-1)^{\inner{D_3, \beta} + dw_3+ \frac{d}{\fa}}}{d(\inner{D_1, \beta} + dw_0)! \prod_{i = 4}^R \inner{D_i, \beta}!} \cdot \frac{\prod_{m = 1}^\infty (\frac{\su_2-f\su_1}{\fm z} - \inner{D_3, \beta} - dw_3 - m)}{\prod_{m = 0}^\infty (\frac{\su_2-f\su_1}{\fm z} + \inner{D_2, \beta} + dw_2 - m)}.
$$
If $\age(\tk)=1$, it is
$$
    q^{\beta}x^d \su_1 \one_{\tk^{-1}} \frac{(-1)^{\floor{\inner{D_3, \beta} + dw_3}+ \ceil{\frac{d}{\fa}}}}{d(\inner{D_1, \beta} + dw_0)! \prod_{i = 4}^R \inner{D_i, \beta}!} \cdot \frac{\prod_{m = 1}^\infty (\frac{\su_2-f\su_1}{\fm z} - \inner{D_3, \beta} - dw_3 - m)}{\prod_{m = 0}^\infty (\frac{\su_2-f\su_1}{\fm z} + \inner{D_2, \beta} + dw_2 - m)}.
$$
If $\age(\tk)=2$, it is
$$
    q^{\beta}x^d \fa\one_{\tk^{-1}} \frac{(-1)^{\floor{\inner{D_3, \beta} + dw_3}+ \ceil{\frac{d}{\fa}}}}{d(\inner{D_1, \beta} + dw_0)! \prod_{i = 4}^R \inner{D_i, \beta}!} \cdot \frac{\prod_{m = 1}^\infty (\frac{\su_2-f\su_1}{\fm z} - \inner{D_3, \beta} - dw_3 - m)}{\prod_{m = 0}^\infty (\frac{\su_2-f\su_1}{\fm z} + \inner{D_2, \beta} + dw_2 - m)}.
$$
Therefore,
$$
[z^{-2}]    \left(I_{\tbeta}(\tq, z), \tgamma_{\tk} \right)_{\tcX}^{\tT'} \bigg|_{\su_4 = 0, \su_2 - f\su_1 = 0}
$$
yields the term in $W^{\cX, (\cL, f)}_{\tk}(q, x)$ that corresponds to $(\beta, d)$. The lemma thus follows as we sum over all classes $\tbeta \in \bK_{\eff, \tsi_0} \setminus \bL_{\bQ}$ with $-\tv(\tbeta) = \tk$.
\end{proof}

\begin{lemma}\label{lem:BModelVanish}
For any $\tbeta \in \bK_{\eff}(\tcX) \setminus \bK_{\eff, \tsi_0}$ and any $\tk \in G_{\tsi_0}$,  
\begin{equation}\label{eqn:BModelVanish}
    \left(I_{\tbeta}(\tq, z), \tgamma_{\tk}\right)_{\tcX}^{\tT'} = 0.
\end{equation}
\end{lemma}

\begin{proof}
By definition, there exists $\tsi \in \tSi(4)$, $\tsi \neq \tsi_0$ such that $\tbeta \in \bK_{\eff, \tsi}$. We first consider the case $\tsi = \iota(\sigma_0)$. In this case, $\inner{\tD_i, \tbeta} \in \bZ_{\ge 0}$ for any $i \in I_{\iota(\sigma_0)} = \{4, \dots, R, R+1\}$. By \eqref{eqn:DRPlusOneLinComb}, we have
$$
    \inner{\tD_1, \tbeta} = \inner{\fa\tD_{R+1} - \sum_{i = 4}^R m_i\tD_i, \tbeta} \in \bZ.
$$
We cannot have $\inner{\tD_1, \tbeta} \in \bZ_{\ge 0}$ as that would imply $\tbeta \in \bK_{\eff, \tsi_0}$. Thus $\inner{\tD_1, \tbeta} \in \bZ_{<0}$. This implies that $I_{\tbeta}(\tq, z)$ contains $\tcD_1^{\tT'}$ as a factor. Then \eqref{eqn:BModelVanish} holds since $\iota_{\tsi_0}^*(\tcD_1^{\tT'}) = 0$.

Now we consider the case $\tsi \neq \iota(\sigma_0)$. In this case, if $\tv(\tbeta) \neq \vzero$, then $\tv(\tbeta)$ is a non-trivial element in $\Box(\tsi)$ and thus cannot represent any element of $\Box(\tsi_0)$. However, $\tgamma_{\tk}$ belongs to a sector corresponding to an element in $\Box(\tsi_0)$. Thus \eqref{eqn:BModelVanish} holds. Now suppose on the other hand that $\tv(\tbeta) = \vzero$, which means that $\tbeta \in \tbL$ and $\inner{\tD_i, \tbeta} \in \bZ$ for all $i = 1, \dots, R+2$. Since $\tbeta \not \in \bK_{\eff, \tsi_0}$, we have in particular that $\tbeta \not \in \tNE(\tsi_0)$, i.e. there exists $i \in I_{\tsi_0} = \{1, 4, \dots, R\}$ such that $\inner{\tD_i, \tbeta} < 0$. This implies that $I_{\tbeta}(\tq, z)$ contains $\tcD_i^{\tT'}$ as a factor. Then \eqref{eqn:BModelVanish} holds since $\iota_{\tsi_0}^*(\tcD_i^{\tT'}) = 0$.
\end{proof}

\begin{lemma}\label{lem:BModelLVanish}
For any $\tbeta \in \bK_{\eff}(\tcX) \cap \bL_{\bQ}$ and any $\tk \in G_{\tsi_0}$, we have
\begin{equation}\label{eqn:BModelLVanish}
[z^{-2}]    \left(I_{\tbeta}(\tq, z), \tgamma_{\tk} \right)_{\tcX}^{\tT'} \bigg|_{\su_4 = 0, \su_2 - f\su_1 = 0} = 0.
\end{equation}
\end{lemma}

\begin{proof}
Note from Observation \ref{obs:ClassInL} that $\inner{\tD_{R+1}, \tbeta} = \inner{\tD_{R+2}, \tbeta} = 0$ and $\age(\tv(\tbeta)) \le 1$, which implies that $[z^{-2}] I_{\tbeta}(\tq, z)$ is either $0$ or has form
$$
    c \tq^{\tbeta} \cdot \begin{cases}
        \tcD_{i}^{\tT'} \tcD_{i'}^{\tT'} & \text{if } \tv(\tbeta) = \vzero,\\
        \tcD_{i}^{\tT'} \one_{\tv(\tbeta)} & \text{if } \tv(\tbeta) \neq \vzero
    \end{cases}
$$
for some $i, i' \in \{1, \dots, R'\}$ and $c \in \bQ$, $c \neq 0$. The pullback of the above term to $\tsi_0$ is zero unless $i, i' \in \{2,3\}$ and $\tv(\tbeta) \in \Box(\tau_0)$.
Moreover, \eqref{eqn:BModelLVanish} is zero unless $\tk$ is contained in $G_{\tau_0}$ and corresponds to $-\tv(\beta)\in \Box(\tau_0)$. In the remainder of this proof we assume that $i, i' \in \{2,3\}$ and $\tv(\beta) \in \Box(\tau_0)$ and $\tk \in G_{\tau_0}$ is the element corresponding to $-\tv(\beta)$. 

If $\tv(\tbeta) = \vzero$, then
$$
    \left(\tcD_{i}^{\tT'} \tcD_{i'}^{\tT'} , \tgamma_{\tk} \right)_{\tcX}^{\tT'} \bigg|_{\su_4 = 0} = \left(\tcD_{i}^{\tT'} \tcD_{i'}^{\tT'} , \frac{\tcD_2^{\tT'}\tcD_3^{\tT'}\tcD_{R+1}^{\tT'}}{\frac{f}{\fm}\su_1 - \frac{1}{\fm} \su_2 - \su_4}\right)_{\tcX}^{\tT'} \bigg|_{\su_4 = 0} =  \frac{\iota_{\tsi_0}^*(\tcD_i^{\tT'}\tcD_{i'}^{\tT'})}{\fa\fm\left(\frac{f}{\fm}\su_1 - \frac{1}{\fm} \su_2 - \su_4 \right)\iota_{\tsi_0}^*(\tcD_{R+2}^{\tT'})} \bigg|_{\su_4 = 0} = \pm \frac{\su_2 - f\su_1}{\fa\fm^2 \su_1},
$$
which further restricts to $0$ under $\su_2 -f\su_1 = 0$. If on the other hand $\tv(\tbeta) \neq \vzero$ then
$$
    \left(\tcD_{i}^{\tT'} \one_{\tv(\tbeta)}, \tgamma_{\tk} \right)_{\tcX}^{\tT'} \bigg|_{\su_4 = 0} = \left(\tcD_{i}^{\tT'} \one_{\tv(\tbeta)}, \tcD_{R+1}^{\tT'}\one_{\lambda^{-1}}  \right)_{\tcX}^{\tT'} \bigg|_{\su_4 = 0} =  \frac{ \iota_{\tsi_0}^*(\tcD_i^{\tT'})}{\fa\fm\iota_{\tsi_0}^*(\tcD_{R+2}^{\tT'})} \bigg|_{\su_4 = 0} = \pm  \frac{\su_2 - f\su_1}{\fa\fm^2 \su_1},
$$
which again restricts to $0$ under $\su_2 -f\su_1 = 0$. Therefore \eqref{eqn:BModelLVanish} holds.
\end{proof}

\begin{proof}[Proof of Theorem \ref{thm:IPairing}]
By our choice of $\tT'$-equivariant lifts in Convention \ref{conv:LiftOuter}, we have $\iota_{\tsi_0}^*(\tu_a) = 0$ for all $a = 1, \dots, R-2$, and thus
$$
    \iota_{\tsi_0}^* \left( e^{\frac{1}{z}\left(\sum_{a \in \{1, \dots, R'-3, R-2\}}\tu_a \log \tq_a \right)} \right) = 1.
$$
Therefore, 
$$
    \left(I_{\tcX}^{\tT'}(\tq, z), \tgamma_{\tk} \right)_{\tcX}^{\tT'} = \left(\sum_{\tbeta \in \bK_{\eff}(\tcX)} I_{\tbeta}(\tq, z), \tgamma_{\tk} \right)_{\tcX}^{\tT'}.
$$
We can then conclude by Lemmas \ref{lem:BModelContribute}, \ref{lem:BModelVanish}, and \ref{lem:BModelLVanish}.
\end{proof}

\subsection{Toric mirror symmetry}\label{sect:Mirror}
In this section and the next, we situate our A- and B-model open/closed correspondences in the context of mirror symmetry. Recall the web of relations in Figure \ref{fig:Web}, where the vertical arrows in Figure \ref{fig:Web} are our open/closed correspondences.

For the bottom arrow, the \emph{mirror theorem} of \cite{Givental98, CCK15, CCIT15} relates the $\tT'$-equivariant $J$- and $I$-functions of $\tcX$ in the following way.

\begin{theorem}[\cite{Givental98, CCK15, CCIT15}] \label{thm:Mirror}
We have
$$
    e^{\frac{1}{z}\ttau_0(\tq_0, \tq)}J_{\tcX}^{\tT'}(\tbtau_2(\tq), z) = I_{\tcX}^{\tT'}(\tq_0, \tq, z),
$$
where the $\tT'$-equivariant \emph{closed mirror map} $\ttau_0 = \ttau_0(\tq_0, \tq)$, $\tbtau_2 = \tbtau_2(\tq)$ is determined by the first-order term in the expansion of the $I$-function in powers of $z^{-1}$:
$$
    I_{\tcX}^{\tT'}(\tq_0, \tq, z) = 1 + z^{-1}\left(\ttau_0(\tq_0, \tq) + \tbtau_2(\tq) \right) + o(z^{-1})
$$
where terms in $o(z^{-1})$ involves $z^{-k}$ for some $k \ge 2$.
\end{theorem}

We now give an explicit description of the closed mirror map for $\tcX$, starting with the following definitions.
\begin{itemize}
    \item For $i = 1, \dots, R', R+1, R+2$, let
        $$
            \tOmega_i := \{\tbeta \in \bK_{\eff}(\tcX) : \tv(\tbeta) = \vzero, \inner{\tD_i, \tbeta}<0, \inner{\tD_{i'}, \tbeta} \ge 0 \text{ for } i' \in \{1, \dots, R+2\} \setminus \{i\} \}, 
        $$
        and
        $$
            \tA_i(\tq) := \sum_{\tbeta \in \tOmega_i} \tq^{\tbeta} \frac{(-1)^{-\inner{\tD_i, \tbeta}-1}(-\inner{\tD_i, \tbeta}-1)!}{\prod_{i' \in \{1, \dots, R+2\} \setminus \{i\}} \inner{\tD_{i'}, \tbeta}!}.
        $$

    \item For $i = R'+1, \dots, R$, let
        $$
            \tOmega_i := \{\tbeta \in \bK_{\eff}(\tcX): \tv(\tbeta) = \tb_i, \inner{\tD_{i'}, \tbeta} \not \in \bZ_{<0} \text{ for } i' = 1, \dots, R+2\},
        $$
        and
        $$
            \tA_i(\tq) := \sum_{\tbeta \in \tOmega_i} \tq^{\tbeta} \prod_{i' = R'+1}^R \frac{\prod_{m = \ceil{\inner{\tD_{i'}, \tbeta}}}^\infty (\inner{\tD_{i'}, \tbeta} - m)}{\prod_{m = 0}^\infty (\inner{\tD_{i'}, \tbeta} - m)}.
        $$

    \item For $i = 1, \dots, R+2$, we write
        $$
            \tD_i = \sum_{a = 1}^{R-2} \tm_i^{(a)}\tH_a
        $$    
        for $\tm_i^{(a)} \in \bQ$. Moreover, set
        $$
            \tlambda_i: = \iota_{\tsi_0}^*(\tcD_i^{\tT'}),
        $$
        which is zero unless $i \in \{2, 3, R+1, R+2\}$. Since $\iota_{\tsi_0}^*(\tu_a) = 0$ for all $a$, we have
        $$
            \tcD_i^{\tT'} = \sum_{a = 1}^{R-2} \tm_i^{(a)}\tu_a + \tlambda_i.
        $$

    \item For $a = 1, \dots, R'-3, R-2$, let
        $$
            \tS_a(\tq) := \sum_{i \in \{1, \dots, R', R+1, R+2\}} \tm_i^{(a)} \tA_i(\tq).
        $$
\end{itemize}
With the above definitions, the $I$-function of $\tcX$ can be written as
$$
    I_{\tcX}^{\tT'}(\tq_0, \tq, z) = 1 + z^{-1}\left(\tq_0 + \sum_{i=1}^{R+2} \tlambda_i \tA_i(\tq) + \sum_{a \in \{1, \dots, R'-3, R-2\}} \left( \log \tq_a + \tS_a(\tq) \right) \tu_a  + \sum_{i = R'+1}^{R} \tA_i(\tq) \one_{\tb_i} \right) + o(z^{-1}).
$$
As stated in Theorem \ref{thm:Mirror}, the closed mirror map for $\tcX$ is thus given by
\begin{align*}
    \ttau_0(\tq_0, \tq) & = \tq_0 + \sum_{i=1}^{R+2} \tlambda_i \tA_i(\tq), \\
    \ttau_a(\tq) & = \begin{cases}
            \log \tq_a + \tS_a(\tq), & a \in \{1, \dots, R'-3, R-2\},\\
            \tA_{a+3}(\tq), & a = R'-2, \dots, R-3.
        \end{cases} 
\end{align*}

We now make two observations that will simplify the closed mirror map above.

\begin{lemma}\label{lem:tmVanish}
We have
\begin{enumerate}
    \item $(\tm_{R+1}^{(1)}, \dots, \tm_{R+1}^{(R-2)}) = - (\tm_{R+2}^{(1)}, \dots, \tm_{R+2}^{(R-2)}) = (0, \dots, 0, 1)$.

    \item $(\tm_{1}^{(R-2)}, \dots, \tm_{R+2}^{(R-2)}) = l^{(0)} = (w_0, w_2, w_3, 0, \dots, 0, \frac{1}{\fa}, -\frac{1}{\fa})$ (see \eqref{eqn:lZero}).
\end{enumerate}
\end{lemma}

\begin{proof}
(1) follows from our choice that $\tH_{R-2} = \fa\tD_{R+1} = -\fa\tD_{R+2}$. For (2), as we observed in the proof of Lemma \ref{lem:BModelContribute}, $\{\tH_1, \dots, \tH_{R-3}\}$ gives a $\bQ$-basis for the span of $\{\tD_4, \dots, \tD_R\}$, which implies that $\tm_4^{(R-2)} = \dots = \tm_{R}^{(R-2)} = 0$. Then (2) follows from that $\talpha(\tm_{1}^{(R-2)}, \dots, \tm_{R+2}^{(R-2)}) = 0$ (see \eqref{eqn:tXSES}) and $\tm_{R+1}^{(R-2)} = \frac{1}{\fa}$.
\end{proof}

\begin{lemma}\label{lem:RPlusVanish}
We have
$$
    \tOmega_{R+1} = \tOmega_{R+2} = \emptyset, \quad \tA_{R+1}(\tq) = \tA_{R+2}(\tq) = 0.
$$
\end{lemma}

\begin{proof}
Note that $\tD_1 + \cdots + \tD_R = 0$ and $\{\tD_1, \dots, \tD_R\}$ contains a $\bQ$-basis for $\tbL_{\bQ}^\vee$. Then, any $\tbeta \in \tbL_{\bQ}$ that satisfies $\inner{\tD_i, \tbeta} \ge 0$ for all $i = 1, \dots, R$ must be zero. This implies that $\tOmega_{R+1} = \tOmega_{R+2} = \emptyset$.
\end{proof}

With Lemmas \ref{lem:tmVanish} and \ref{lem:RPlusVanish}, we simplify the mirror map as follows:
\begin{equation}\label{eqn:ClosedMirrorMap}
\begin{aligned}
    \ttau_0(\tq_0, \tq) & = \tq_0 + \tlambda_2 \tA_2(\tq) + \tlambda_3 \tA_3(\tq)\\
        & = \tq_0 + \left( -\frac{f}{\fm}\su_1 + \frac{1}{\fm} \su_2 \right) \tA_2(\tq) + \left( \frac{f}{\fm}\su_1 - \frac{1}{\fm} \su_2 - \su_4 \right) \tA_3(\tq),\\
    \ttau_a(\tq) & = \log \tq_a + \tS_a(\tq) = \log \tq_a + \sum_{i =1}^{R'} \tm_i^{(a)} \tA_i(\tq), \quad a = 1, \dots, R'-3,\\
    \ttau_a(\tq) & = \tA_{a+3}(\tq), \quad a = R'-2, \dots, R-3,\\
    \ttau_{R-2}(\tq) &= \log \tq_{R-2} + w_0 \tA_1(\tq) + w_2 \tA_2(\tq) + w_3 \tA_3(\tq).
\end{aligned}
\end{equation}

Now we prove Lemma \ref{lem:ClosedWellDefined} that the closed Gromov-Witten invariants under consideration are defined.

\begin{proof}[Proof of Lemma \ref{lem:ClosedWellDefined}]
By Lemma \ref{lem:JPairing} and Theorem \ref{thm:Mirror}, we have
$$
    \llangle \tgamma_{\lambda} \rrangle^{\tcX,\tT'}(\tbtau_2) = [z^{-2}] (J_{\tcX}^{\tT'}(\tbtau_2,z), \tgamma_{\lambda})^{\tT'}_{\tcX} = [z^{-2}] \left(e^{-\frac{1}{z}\ttau_0(0, \tq)}I_{\tcX}^{\tT'}(\tq, z), \tgamma_{\tk} \right)_{\tcX}^{\tT'}
$$
under the mirror map \eqref{eqn:ClosedMirrorMap}. Note that
$$
    \ttau_0(0, \tq) \big|_{\su_4 = 0} =  \frac{\su_2 - f\su_1}{\fm}(-\tA_2(\tq) + \tA_3(\tq)).
$$
By a similar vanishing argument as in the proof of Lemma \ref{lem:BModelLVanish}, we have
\begin{equation}\label{eqn:NoTau0}
   [z^{-2}] \left(e^{-\frac{1}{z}\ttau_0(0, \tq)}I_{\tcX}^{\tT'}(\tq, z), \tgamma_{\tk} \right)_{\tcX}^{\tT'} \bigg|_{\su_4 = 0, \su_2 - f\su_1 = 0} 
   = [z^{-2}]\left(I_{\tcX}^{\tT'}(\tq, z)
, \tgamma_{\tk} \right)_{\tcX}^{\tT'} \bigg|_{\su_4 = 0, \su_2 - f\su_1 = 0},
\end{equation}
which is identified by Theorem \ref{thm:IPairing} with $W^{\cX, (\cL, f)}_{\tk}(q, x)$ under the relation $q_a = \tq_a$ for $a = 1, \dots, R-3$ and $x = \tq_{R-2}$. By definition (Definition \ref{def:WOuter}), $W^{\cX, (\cL, f)}_{\tk}(q, x)$ is a power series in $q, x$ with $\bQ$-coefficients. Therefore, considering $[z^{-2}] \left(e^{-\frac{1}{z}\ttau_0(0, \tq)}I_{\tcX}^{\tT'}(\tq, z), \tgamma_{\tk} \right)_{\tcX}^{\tT'}$ as a power series in $\tq$ with coefficients in $\cQ_{\tT'}$, we see that none of such coefficients have a pole along $\su_4 = 0, \su_2 - f\su_1 = 0$. Since the mirror map $\tbtau_2 = \tbtau_2(\tq)$ \eqref{eqn:ClosedMirrorMap} is non-equivariant, if we consider the generating function $\llangle \tgamma_{\lambda} \rrangle^{\tcX,\tT'}(\tbtau_2)$ as a power series in $\tQ,\tbtau_2''$ with coefficients in $\cQ_{\tT'}^{\bC}$ (Definition \ref{def:Correlator}), then none of such coefficients have a pole along $\su_4 = 0, \su_2 - f\su_1 = 0$ either. The lemma thus follows.
\end{proof}

\subsection{Open mirror symmetry and compatibility}
The mirror theorem of \cite{Givental98, CCK15, CCIT15} also applies to the Calabi-Yau $3$-orbifold $\cX$, which relates the ($T'$-equivariant) $J$- and $I$-functions of $\cX$ via the closed mirror map $\btau_2 = \btau_2(q)$. As conjectured by \cite{AV00} and proven by \cite{FL13, FLT12}, such mirror symmetry can be extended to the open sector in the sense that the A- and B-model disk functions can be identified via the closed mirror map and an additional \emph{open} mirror map $\sX = \sX(q,x)$. This gives the top arrow in Figure \ref{fig:Web}.

\begin{theorem}[\cite{FL13, FLT12}]\label{thm:OpenMirror}
For all $\tk \in \mu_{\fa\fm}$, 
$$
F^{\cX, (\cL, f)}_{\tk}(\btau_2(q), \sX(q,x)) = W^{\cX, (\cL, f)}_{\tk}(q, x) 
$$ 
under the open-closed mirror map $\btau_2 = \btau_2(q)$, $\sX = \sX(q,x)$, given in \eqref{eqn:OpenMirrorMap} below.
\end{theorem}

We describe the open-closed mirror map of $(\cX, \cL, f)$, again starting with some definitions.
\begin{itemize}
    \item For $i = 1, \dots, R'$, let
        $$
            \Omega_i := \{\beta \in \bK_{\eff}(\cX) : v(\beta) = \vzero, \inner{D_i, \beta}<0, \inner{D_{i'}, \beta} \ge 0 \text{ for } i' \in \{1, \dots, R\} \setminus \{i\} \}, 
        $$
        and
        $$
            A_i(q) := \sum_{\beta \in \Omega_i} q^{\beta} \frac{(-1)^{-\inner{D_i, \beta}-1}(-\inner{D_i, \beta}-1)!}{\prod_{i' \in \{1, \dots, R\} \setminus \{i\}} \inner{D_{i'}, \beta}!}.
        $$

    \item For $i = R'+1, \dots, R$, let
        $$
            \Omega_i := \{\beta \in \bK_{\eff}(\cX): v(\beta) = b_i, \inner{D_{i'}, \beta} \not \in \bZ_{<0} \text{ for } i' = 1, \dots, R\},
        $$
        and
        $$
            A_i(q) := \sum_{\beta \in \Omega_i} q^{\beta} \prod_{i' = R'+1}^R \frac{\prod_{m = \ceil{\inner{D_{i'}, \beta}}}^\infty (\inner{D_{i'}, \beta} - m)}{\prod_{m = 0}^\infty (\inner{D_{i'}, \beta} - m)}.
        $$

    \item For $i = 1, \dots, R$, we write
        $$
            D_i = \sum_{a = 1}^{R-3} m_i^{(a)}H_a
        $$    
        for $m_i^{(a)} \in \bQ$.

    \item For $a = 1, \dots, R'-3$, let
        $$
            S_a(q) := \sum_{i =1}^{R'} m_i^{(a)} A_i(q).
        $$
\end{itemize}
Then, the open-closed mirror map of $\cX$ is given by
\begin{equation}\label{eqn:OpenMirrorMap}
    \tau_a(q)  = \begin{cases}
            \log q_a + S_a(q), & a \in \{1, \dots, R'-3\},\\
            A_{a+3}(q), & a = R'-2, \dots, R-3,
        \end{cases}
\end{equation}
$$
    \log \sX = \log x + w_0A_1(q) + w_2A_2(q) + w_3A_3(q).
$$

For the remainder of this section, we show that our A- and B-model open/closed correspondences can be used to provide an alternative proof of the open-closed mirror symmetry of $(\cX, \cL, f)$ \cite{FL13, FLT12} (Theorem \ref{thm:OpenMirror}) in the more general case where the framing $f \in \bQ$ is not necessarily an integer. In other words, we show that the diagram in Figure \ref{fig:Web} is ``commutative'' by showing that the top arrow can be recovered from the other three arrows. We first identify the open-closed mirror map of $(\cX, \cL, f)$ with the closed mirror map of $\tcX$.

\begin{proposition}\label{prop:MirrorMapCorr}
Under the relation $\tq_a = q_a$ for $a = 1, \dots, R-3$ and $\tq_{R-2} = x$, we have
$$
    \ttau_a(\tq) = \tau_a(q), \quad a = 1, \dots, R-3,
$$
$$
    \ttau_{R-2}(\tq) = \log \sX(q, x). 
$$
\end{proposition}

\begin{proof}
By construction, under the projection $\tbL^\vee \to \bL^\vee$, $\tD_i$ projects to $D_i$ for $i = 1, \dots, R$, $\tD_{R+1}$ and $\tD_{R+2}$ projects to $0$, and $\tH_a$ projects to $H_a$ for $a = 1, \dots, R-3$. By Lemma \ref{lem:tmVanish}, we have
\begin{equation}\label{eqn:mMatches}
    \tm_i^{(a)} = m_i^{(a)}, \quad i = 1, \dots, R, a = 1, \dots, R-3.
\end{equation}

Now observe that if $\tbeta \in \tbL$ satisfies that both $\inner{\tD_{R+1}, \tbeta}$ and $\inner{\tD_{R+2}, \tbeta}$ are non-negative (or both are non-positive), then both are in fact zero and $\tbeta \in \bL$. Thus, by Observation \ref{obs:ClassInL}, it is straightforward to verify that $\tq^{\tbeta}$ does not involve $\tq_{R-2}$ and equals $q^{\tbeta}$, and that for any $i = 1, \dots, R$,
\begin{equation}\label{eqn:OmegaAMatches}
    \tOmega_i = \Omega_i, \quad \tA_i(\tq) = A_i(q),
\end{equation}
under the relation $\tq_a = q_a$, $a = 1, \dots, R-3$. This combined with \eqref{eqn:mMatches} gives
$$
    \tS_a(\tq) = S_a(q)
$$
for $a = 1, \dots, R-3$. We then conclude by the descriptions \eqref{eqn:ClosedMirrorMap}, \eqref{eqn:OpenMirrorMap} of the mirror maps.
\end{proof}

Now we give our alternative proof of Theorem \ref{thm:OpenMirror}.

\begin{proof}[Proof of Theorem \ref{thm:OpenMirror} via Theorems \ref{thm:JPairing} and \ref{thm:IPairing}]
We have the following:
\begin{align*}
    & F^{\cX, (\cL, f)}_{\tk}(\btau_2, \sX) &&\\
    & = [z^{-2}] \left(J_{\tcX}^{\tT'}(\tbtau_2, z), \tgamma_{\tk} \right)_{\tcX}^{\tT'} \bigg|_{\su_4 = 0, \su_2 - f\su_1 = 0} && (\text{by Theorem \ref{thm:JPairing} under $\tau_a = \ttau_a$, $\log \sX = \ttau_{R-2}$})\\
    & = [z^{-2}]\left(e^{-\frac{1}{z}\ttau_0(0, \tq)}I_{\tcX}^{\tT'}(\tq, z)
, \tgamma_{\tk} \right)_{\tcX}^{\tT'} \bigg|_{\su_4 = 0, \su_2 - f\su_1 = 0} && (\text{by Theorem \ref{thm:Mirror} under closed mirror map \eqref{eqn:ClosedMirrorMap}})\\
    & = [z^{-2}]\left(I_{\tcX}^{\tT'}(\tq, z)
, \tgamma_{\tk} \right)_{\tcX}^{\tT'} \bigg|_{\su_4 = 0, \su_2 - f\su_1 = 0} && (\text{by \eqref{eqn:NoTau0}})\\
    & = W^{\cX, (\cL, f)}_{\tk}(q, x) && (\text{by Theorem \ref{thm:IPairing} under $q_a = \tq_a$, $x = \tq_{R-2}$}).
\end{align*}
Moreover, Proposition \ref{prop:MirrorMapCorr} implies that the dependence of $\btau_2$ and $\sX$ on $q$ and $x$ through the above chain of correspondences agrees with the open-closed mirror map \eqref{eqn:OpenMirrorMap} of $(\cX, \cL, f)$.
\end{proof}

\begin{remark}\rm{
As a consequence of the ``commutativity'' of the diagram in Figure \ref{fig:Web}, we may deduce any one of Theorems \ref{thm:JPairing}, \ref{thm:IPairing}, and \ref{thm:OpenMirror} (i.e. the left, right, and top arrows) from the other two. This yields alternative proofs of both Theorems \ref{thm:JPairing} and \ref{thm:IPairing}.
}
\end{remark}








\appendix

\section{Preliminaries of orbifold Gromov-Witten theory}\label{sect:GWPrelim}
In this section, we collect additional premliminaries of orbifold Gromov-Witten theory, mainly to supplement Section \ref{sect:GW}.

\subsection{Hurwitz-Hodge integrals}\label{sect:HHIntegrals}
In this section, we briefly review \emph{Hurwitz-Hodge integrals}, which are intersection numbers on moduli spaces of twisted stable maps to the classifying stack of a finite group. We restrict our attention to the genus zero case and the case where the finite group is abelian, which we will need in our localization computations. We fix a finite abelian group $G$ in this section. The classifying stack $\cB G = [\pt / G]$ is a smooth Deligne-Mumford stack, and
$$
    \cI \cB G = \bigsqcup_{k \in G} (\cB G)_k.
$$

Let $n \in \bZ_{\ge 0}$ and $\vk = (k_1, \dots, k_n) \in G^n$. The definitions in Section \ref{sect:Moduli} can be applied to define the moduli space $\Mbar_{0,\vk}(\cB G):= \Mbar_{0,\vk}(\cB G, 0)$ of genus-zero, $\vk$-twisted stable maps to $\cB G$. We assume that $n \ge 3$ and 
$$
    k_1 \cdots  k_n = 1 \in G.
$$
Otherwise, $\Mbar_{0,\vk}(\cB G)$ is empty.


Let $\pi: \cU \to \Mbar_{0,\vk}(\cB G)$ be the universal curve, $u: \cU \to \cB G$ be the universal map, and $\epsilon: \Mbar_{0,\vk}(\cB G) \to \Mbar_{0,n}$ be the natural forgetful map. Let $\rho: G \to \GL(V)$ be an irreducible representation of $G$, where $V$ is a $1$-dimensional vector space over $\bC$. Then $\cE_\rho:= [V/G]$ is a line bundle over $\cB G$, and
$$
    \pi_* u^* \cE_\rho = \begin{cases}
        \cO_{\Mbar_{0,\vk}(\cB G)} & \text{if $\rho$ is the trivial representation} \\
        0 & \text{otherwise.} 
    \end{cases}
$$
The $\rho$-twisted \emph{Hurwitz-Hodge bundle} $\bE_\rho$ is the dual of the vector bundle $R^1\pi_*f^* \cE_\rho \to \Mbar_{0, \vk}(\cB G)$. If $\rho = 1$ is the trivial representation, then $\bE_1$ is the pullback of the Hodge bundle over $\Mbar_{0, n}$ under the map $\epsilon$, and thus $\bE_1 = 0$. If $\rho$ is non-trivial, then we have
\begin{equation}\label{eqn:HHBundleRank}
    \rank(\bE_\rho) = -1 + \sum_{i = 1}^n c_i,
\end{equation}
where for each $i = 1, \dots, n$, $c_i \in [0,1) \cap \bQ$ is such that the eigenvalue of $\rho(k_i) \in \GL(V)$ is $\exp(2\pi\sqrt{-1}c_i)$.

We define the following classes on $\Mbar_{0,\vk}(\cB G)$:
\begin{itemize}
    \item \emph{Hodge classes:} Given an irreducible representation $\rho$ of $G$, and $j = 0, \dots, \rank(\bE_\rho)$, let
    $$
        \lambda_j^\rho:= c_j(\bE_\rho) \in A^j(\Mbar_{0,\vk}(\cB G)).
    $$

    \item \emph{Descendant classes:} For $i = 1, \dots, n$, let
    $$
        \bar{\psi}_i := \epsilon^* \psi_i \in A^1(\Mbar_{0,\vk}(\cB G)).
    $$
\end{itemize}
Moreover, we define
$$
    \Lambda_\rho^\vee(\bw) = \sum_{j = 0}^{\rank(\bE_\rho)} (-1)^j \lambda_j^\rho \bw^{\rank(\bE_\rho)-j},
$$
where $\bw$ is a formal variable. We will use the following version of Mumford's relation \cite[Proposition]{BGP08}:
\begin{lemma}[\cite{BGP08}]\label{lem:Mumford}
Let $G$ be a finite abelian group, $n \in \bZ_{\ge 3}$, and $\vk = (k_1, \dots, k_n) \in G^n$ such that $k_1 \cdots k_n = 1 \in G$. Let $\rho$ be an irreducible representation of $G$, $\rho^{\vee}$ denote its dual representation, and $\bw$ be a formal variable. Then 
$$
    \Lambda_\rho^\vee(\bw) \Lambda_{\rho^\vee}^\vee(-\bw) = (-1)^{\rank(\bE_{\rho^\vee})}\bw^{\rank(\bE_\rho) + \rank(\bE_{\rho^{\vee}})},
$$
where $\bE_\rho, \bE_{\rho^\vee}$ are Hurwitz-Hodge bundles over $\Mbar_{0,\vk}(\cB G)$ defined above.
\end{lemma}

\emph{Hurwitz-Hodge integrals} are integrals of form
$$
    \int_{\Mbar_{0,\vk}(\cB G)} \bar{\psi}_1^{a_1} \cdots \bar{\psi}_n^{a_n}\lambda_{j_1}^{\rho_1} \cdots \lambda_{j_m}^{\rho_m},
$$
where $a_1, \dots, a_n \in \bZ_{\ge 0}$, $\rho_1, \dots, \rho_m$ are (not necessarily distinct) irreducible representations of $G$, and each $j_i \in \{0, \dots, \rank(\bE_{\rho_i})\}$. Zhou \cite{Zhou07} gave an algorithm for computing these integrals, as follows: By Tseng's orbifold Riemann-Roch theorem \cite{Tseng10}, the Hurwitz-Hodge integrals can be reconstructed from descendant integrals, i.e. integrals of form
$$
    \int_{\Mbar_{0,\vk}(\cB G)} \bar{\psi}_1^{a_1} \cdots \bar{\psi}_n^{a_n},
$$
where $a_1, \dots, a_n \in \bZ_{\ge 0}$. By Jarvis-Kimura \cite[Proposition 3.4]{JK02}, we have
$$
    \int_{\Mbar_{0,\vk}(\cB G)} \bar{\psi}_1^{a_1} \cdots \bar{\psi}_n^{a_n} = \frac{1}{|G|} \int_{\Mbar_{0,n}} \psi_1^{a_1} \cdots \psi_n^{a_n} = \begin{cases}
        \frac{(n-3)!}{a_1! \cdots a_n!} & \text{if } a_1 + \cdots + a_n = n-3\\
        0 & \text{otherwise.}
    \end{cases}
$$
As a consequence, it is straightforward to derive the following identity (see e.g. \cite[Lemmas 61 and 123]{Liu13} for some special cases):
\begin{lemma}\label{lem:DescendantIntegral}
Let $G$ be a finite abelian group, $n \in \bZ_{\ge 3}$, and $\vk = (k_1, \dots, k_n) \in G^n$ such that $k_1 \cdots k_n = 1 \in G$. Let $S \subseteq \{1, \dots, n\}$, and for each $i \in S$, let $\bw_i$ be a formal variable. Then
$$
    \int_{\Mbar_{0, \vk}(\cB G)} \prod_{i \in S} \frac{1}{\bw_i - \bar{\psi}_i} = \frac{1}{|G| \cdot \prod_{i \in S} \bw_i} \left( \sum_{i \in S} \frac{1}{\bw_i} \right)^{n-3}.
$$
\end{lemma}

In our localization computations, we adopt the following integration conventions for
\begin{itemize}
    \item $\Mbar_{0, (1)}(\cB G)$, viewed as a $(-2)$-dimensional space, and
    \item $\Mbar_{0, (k, k^{-1})}(\cB G)$, viewed as a $(-1)$-dimensional space, where $k \in G$.
\end{itemize}
Let $\bw_1, \bw_2$ be formal variables. We set
\begin{equation}\label{eqn:UnstableIntegral}
    \begin{aligned}
         &\int_{\Mbar_{0, (1)}(\cB G)} \frac{1}{\bw_1 - \bar{\psi}_1} = \frac{\bw_1}{|G|}, \quad \int_{\Mbar_{0, (k, k^{-1})}(\cB G)} \frac{1}{\bw_1 - \bar{\psi}_1} = \frac{1}{|G|}, \\
        & \int_{\Mbar_{0, (k, k^{-1})}(\cB G)} \frac{1}{(\bw_1 - \bar{\psi}_1)(\bw_2 - \bar{\psi}_2)} = \frac{1}{|G| \cdot (\bw_1 + \bw_2)}.
    \end{aligned}
\end{equation}
Note that this is consistent with Lemma \ref{lem:DescendantIntegral}.

\subsection{Twisted covers of proper torus-invariant lines}\label{sect:TwistedCovers}
In this section, we characterize non-constant representable morphisms from irreducible twisted curves to proper torus-invariant lines in toric orbifolds. Such maps are studied in detail by Johnson \cite{Jo09}.

The domain of such a map has form $\cC_{r_1, r_2}$ for some $r_1, r_2 \in \bZ_{>0}$, which is the $1$-dimensional toric orbifold defined by the stacky fan 
$$
    \bSi_{r_1, r_2} = (\bZ, \Sigma, \alpha_{r_1,r_2}).
$$
Here, $\Sigma$ is the complete fan in $\bR$ and $\alpha_{r_1,r_2}: \bZ^2 \to \bZ$ is determined by $(r_1, -r_2)$. The coarse moduli space of $\cC_{r_1, r_2}$ is isomorphic to $\bP^1$. There are two $(\bC^*)$-fixed points $\fp_1, \fp_2$ of $\cC_{r_1, r_2}$, with generic stablizer groups $\mu_{r_1}, \mu_{r_2}$ respectively. Let $p_1, p_2 \in \bP^1$ be their images in the coarse moduli space, which are the two $(\bC^*)$-fixed points.

Now let $\cZ$ be an $r$-dimensional toric orbifold specified by an extended stacky fan $\bXi = (\bZ^r, \Xi, \alpha)$ as in Section \ref{sect:PrelimToric}, with Deligne-Mumford torus $(\bC^*)^r$. Let $\tau \in \Xi(r-1)_c$ and $\sigma_1, \sigma_2 \in \Xi(r)$ be the two $r$-dimensional cones that contain $\tau$. Let
\begin{equation}\label{eqn:TwistedCover}
    u: \cC_{r_1, r_2} \to \fl_{\tau}
\end{equation}
be a representable morphism such that if $\bar{u}: C \to l_{\tau}$ is the induced map between coarse moduli spaces, then $\bar{u}(p_1) = p_{\sigma_1}$ and $\bar{u}(p_2) = p_{\sigma_2}$. Consider the restrictions to the open orbit
$$
    u|_{\bC^*}: \bC^* \to \fo_\tau, \quad \bar{u}|_{\bC^*}: \bC^* \to o_\tau \cong \bC^*.
$$
Let $\gamma$ be the image of the generator of $\pi_1(\bC^*) \cong \bZ$ under the map $(u|_{\bC^*}): \pi_1(\bC^*) \to \pi_1(\fo_\tau) = H_\tau$. Then $d = \pi_\tau(\gamma)>0$ is the image of the generator of $\pi_1(\bC^*)$ under the map $(\bar{u}|_{\bC^*})_*: \pi_1(\bC^*) \to \pi_1(o_\tau) \cong \bZ$, and is the degree of the map $\bar{u}|_{\bC^*}$. The map $u$ is in fact uniquely determined by the element $\gamma \in H_\tau$ up to automorphisms, and we say that $\gamma$ is the \emph{degree} of $u$.

Let $k_1 \in G_{\sigma_1}$ (resp. $k_2 \in G_{\sigma_2}$) be the image of the generator of the stablizer group $\mu_{r_1}$ of $\fp_1$ (resp. $\mu_{r_2}$ of $\fp_2$) under $u$. The representability of $u$ implies that $k_1$ (resp. $k_2$) has order $r_1$ (resp. $r_2$). Moreover, $k_1$ and $k_2$ are determined by $\gamma$ as
$$
    \pi_{(\tau, \sigma_1)}(\gamma) = k_1, \quad \pi_{(\tau, \sigma_2)}(\gamma) = k_2
$$
(see \eqref{eqn:FlagFundGroup}).

\section{Details of localization computations and comparisons}\label{sect:NumericalLocalDetail}
In this section, we supply the details of localization computations and comparisons needed in proving the numerical open/closed correspondence (Theorem \ref{thm:NumericalOuter}) in Section \ref{sect:Numerical}.

\subsection{Contributions from maps to \texorpdfstring{$\cX$}{X}}
In this section, we consider maps from an irreducible twisted curve to a $T'$-fixed point (resp. proper $T'$-invariant line) in $\cX$, which can also be viewed as maps to the corresponding $\tT'$-fixed point (resp. proper $\tT'$-invariant line) in $\tcX$ via the inclusion $\cX \to \tcX$. We explicitly compare their contributions to the localization computations of the disk invariants (Proposition \ref{prop:DiskLocalResultOuter}) and closed invariants (Proposition \ref{prop:ClosedLocalResultOuter}). 


\begin{lemma}[Edges] \label{lem:EdgeComparison}
Let $\tau \in \Sigma(2)_c$, and 
$$
    u: \cC = \cC_{r_1, r_2} \to \fl_{\tau} \subset \cX \subset \tcX
$$
be a morphism as in \eqref{eqn:TwistedCover} in Section \ref{sect:TwistedCovers}. Let
$$
    \bh: = \frac{e_{T'}(H^1(\cC,u^*T\cX)^m)}{e_{T'}(H^0(\cC,u^*T\cX)^m)}, \quad 
    \tbh:= \frac{e_{\tT'}(H^1(\cC,u^*T\tcX)^m)}{e_{\tT'}(H^0(\cC,u^*T\tcX)^m)}.
$$
Then
$$
    \left(\su_4 \tbh \right) \big|_{\su_4 = 0} = \bh.
$$
\end{lemma}

\begin{proof}
Over $\fl_{\tau} = \fl_{\iota(\tau)}$, we have the following relation:
$$
    T\tcX\big|_{\fl_{\iota(\tau)}} \cong \cO \oplus T\cX \big|_{\fl_{\tau}},
$$
which implies that
$$
   \tbh = \frac{1}{e_{\tT'}(H^0(\cC,u^*\cO))} \cdot \frac{e_{\tT'}(H^1(\cC,u^*(T\cX \big|_{\fl_{\tau}}))^m)}{e_{\tT'}(H^0(\cC,u^*(T\cX \big|_{\fl_{\tau}}))^m)} = \frac{1}{\su_4} \cdot \frac{e_{\tT'}(H^1(\cC,u^*(T\cX \big|_{\fl_{\tau}}))^m)}{e_{\tT'}(H^0(\cC,u^*(T\cX \big|_{\fl_{\tau}}))^m)}.
$$
By \eqref{eqn:TanWtU4RestrictOuter}, we may obtain $T'$-weights on $T\cX \big|_{\fl_{\tau}}$ by applying the restriction $\su_4 = 0$ to $\tT'$-weights. The lemma thus directly follows from the explicit computations of $\bh$ and $\tbh$ in \cite[Lemma 130]{Liu13}.
\end{proof}

\begin{lemma}[Flags] \label{lem:FlagComparison}
Let $\sigma \in \Sigma(3)$ and $k \in G_\sigma = G_{\iota(\sigma)}$. Let
$$
    \bh: = e_{T'}\left((T_{\fp_{\sigma}}\cX)^{k}\right), \quad \tbh:= e_{\tT'}\left((T_{\fp_{\iota(\sigma)}}\tcX)^{k}\right).
$$
Then 
$$
    \frac{\tbh}{\su_4} \bigg|_{\su_4 = 0} = \bh.
$$
\end{lemma}

\begin{proof}
We have
$$
    \bh = \prod_{(\tau, \sigma) \in F(\Sigma), k \in G_\tau} \bw(\tau, \sigma), \quad \tbh = \prod_{(\ttau, \iota(\sigma)) \in F(\tSi), k \in G_{\ttau}} \tbw(\ttau, \iota(\sigma)).
$$
First, $\tbh$ contains $\tbw(\sigma, \iota(\sigma)) = \su_4$ as a factor. For the other facets of $\iota(\sigma)$, since $G_{\iota(\tau)} = G_{\tau}$ for any $\tau \in \Sigma(2)$, we have $k \in G_{\iota(\tau)}$ if and only if $k \in G_{\tau}$. Thus $\bw(\iota(\tau), \iota(\sigma))$ is a factor of $\tbh$ if and only if $\bw(\tau, \sigma)$ is a factor of $\bh$. The lemma then follows from \eqref{eqn:TanWtU4RestrictOuter}.
\end{proof}

\begin{lemma}[Stable vertices] \label{lem:VertexComparison}
Let $\sigma \in \Sigma(3)$, $n \in \bZ_{\ge 3}$, $\vk = (k_1, \dots, k_n) \in G_{\sigma}^n = G_{\iota(\sigma)}^n$ such that $k_1 \cdots k_n = 1$, and
$$
    u: (\cC, \fr_1, \dots, \fr_n) \to \fp_{\sigma} \subset \cX \subset \tcX
$$
be a morphism that represents a point in $\Mbar_{0, \vk}(\cB G_\sigma)$. Let
$$
    \bh:= \frac{e_{T'}(H^1(\cC,u^*T\cX)^m)}{e_{T'}(H^0(\cC,u^*T\cX)^m)}, \quad \tbh:= \frac{e_{\tT'}(H^1(\cC,u^*T\tcX)^m)}{e_{\tT'}(H^0(\cC,u^*T\tcX)^m)}.
$$
Then
$$
    \left(\su_4 \tbh \right) \big|_{\su_4 = 0} = \bh.
$$
\end{lemma}

\begin{proof}
We first recall the explicit computations of $\bh$ and $\tbh$ in \cite[Lemma 126]{Liu13}. Consider the Cartesian diagram
$$
\xymatrix{
    \tcC \ar[d] \ar[r]^{\tu} & \pt \ar[d] \\
    \cC \ar[r]^u & \cB G_\sigma.
}
$$
Let $\hG \subseteq G_\sigma$ denote the subgroup generated by the monodromies of the $G_\sigma$-cover $\tcC \to \cC$. For each flag $(\tau, \sigma) \in F(\Sigma)$ (resp. $(\ttau, \iota(\sigma)) \in F(\tSi)$), recall that $\chi_{(\tau, \sigma)}$ (resp. $\chi_{(\ttau, \iota(\sigma))}$) denotes the $G_\sigma$-representation $T_{\fp_\sigma}\fl_{\tau}$ (resp. $T_{\fp_{\iota(\sigma)}} \fl_{\ttau}$). Then, \cite[Lemma 126]{Liu13} states that
$$
    \bh = \frac{\prod_{(\tau, \sigma) \in F(\Sigma)} \Lambda_{\chi_{(\tau, \sigma)}}^\vee(\bw(\tau, \sigma))}{\prod_{(\tau, \sigma) \in F(\Sigma), \hG \subseteq G_\tau}\bw(\tau, \sigma)}, \quad
    \tbh = \frac{\prod_{(\ttau, \iota(\sigma)) \in F(\tSi)} \Lambda_{\chi_{(\ttau, \iota(\sigma))}}^\vee(\tbw(\ttau, \iota(\sigma)))}{\prod_{(\ttau, \iota(\sigma)) \in F(\tSi), \hG \subseteq G_{\ttau}}\tbw(\ttau, \iota(\sigma))}.
$$
For each $(\tau, \sigma) \in F(\Sigma)$, we have $\chi_{(\tau, \sigma)} = \chi_{(\iota(\tau), \iota(\sigma))}$. Thus by \eqref{eqn:TanWtU4RestrictOuter}, we have
$$
    \Lambda_{\chi_{(\iota(\tau), \iota(\sigma))}}^\vee(\tbw(\iota(\tau), \iota(\sigma))) \big|_{\su_4 = 0} = \Lambda_{\chi_{(\tau, \sigma)}}^\vee(\bw(\tau, \sigma)).
$$
Moreover, $\chi_{(\sigma, \iota(\sigma))}$ is trivial, and thus
$$
    \Lambda_{\chi_{(\sigma, \iota(\sigma))}}^\vee(\tbw(\sigma, \iota(\sigma))) = 1.
$$
Finally, since $\hG \subseteq G_\sigma$, $\tbw(\sigma, \iota(\sigma)) = \su_4$ is a factor of the denominator of $\tbh$. The lemma then follows from \eqref{eqn:TanWtU4RestrictOuter}.
\end{proof}

\begin{lemma}[Vertex integrals] \label{lem:IntegralComparison}
Let $\sigma \in \Sigma(3)$, $n \in \bZ_{\ge 1}$, $\vk = (k_1, \dots, k_n) \in G_{\sigma}^n = G_{\iota(\sigma)}^n$ such that $k_1 \cdots k_n = 1$. Let $0 \le n' \le n$, and $\tau_1, \dots, \tau_{n'} \in \Sigma(2)_c$ such that each $\tau_i$ is a facet of $\sigma$ and $k_i \in G_{\tau_i} = G_{\iota(\tau_i)}$. In addition, for each $i = 1, \dots, n'$, let $r_i$ be the order of $k_i$ in $G_{\sigma}$, and take $d_i \in \bZ_{>0}$. Let
$$
    \bh := \int_{\Mbar_{0, \vk}(\cB G_{\sigma})} \frac{1}{\prod_{i = 1}^{n'} \left(\frac{\fr(\tau_i, \sigma)\bw(\tau_i, \sigma)}{r_id_i} - \frac{\bar{\psi}_i}{r_i}\right)}, \quad
    \tbh := \int_{\Mbar_{0, \vk}(\cB G_{\iota(\sigma)})} \frac{1}{\prod_{i = 1}^{n'} \left(\frac{\fr(\iota(\tau_i), \iota(\sigma))\tbw(\iota(\tau_i), \iota(\sigma))}{r_id_i} - \frac{\bar{\psi}_i}{r_i}\right)}.
$$
Then
$$
    \tbh \big|_{\su_4 = 0} = \bh.
$$
\end{lemma}

\begin{proof}
We first show that $\su_4$ is not a factor of the denominator of $\tbh$, i.e. $\tbh \big|_{\su_4 = 0}$ is defined. To see this, note by \eqref{eqn:TanWtU4RestrictOuter} that $\tbw(\iota(\tau_i),\iota(\sigma)) \not \in \bQ \su_4$. In view of Lemma \ref{lem:DescendantIntegral}, the only case where $\su_4$ appears as a factor in the denominator is when $n = n' = 2$ and
$$
    r_1\tbw(\iota(\tau_1),\iota(\sigma)) + r_2\tbw(\iota(\tau_2),\iota(\sigma)) \in \bQ \su_4,
$$
but this equation cannot hold. Then, the lemma is a direct consequence of \eqref{eqn:TanWtU4RestrictOuter}.
\end{proof}

\begin{lemma} \label{lem:GenFNoPower}
Let the quantity $\bh$ be as defined in Lemma \ref{lem:EdgeComparison}, \ref{lem:FlagComparison}, \ref{lem:VertexComparison}, or \ref{lem:IntegralComparison}. Then for generic $f \in \bQ$, $\su_2 - f\su_1$ is not a factor of $\bh$.
\end{lemma}

\begin{proof}
The lemma follows from that the data defining $\bh$ is independent of $f$.
\end{proof}

\subsection{Fixed points corresponding to cones in \texorpdfstring{$\tSi(4) \setminus \iota(\Sigma(3))$}{Sigma}}\label{sect:TanWtsExtraConesOuter}
In this section, we study the tangent $\tT'$-weights at a fixed point $\fp_{\tsi}$ in $\tcX$ corresponding to a cone $\tsi \in \tSi(4) \setminus \iota(\Sigma(3))$. In addition, we study the contributions of vertices and flags associated to such fixed points to the localization computations of the closed invariants (Proposition \ref{prop:ClosedLocalResultOuter}). See Section \ref{sect:ComparisonOuter} for a description of such cones, their associated flags, and stablizers.

First, the tangent $\tT'$-weights at $\fp_{\tsi}$ are given by
\begin{align*}
    \tbw(\iota(\delta_0(\tsi)), \tsi) &= \frac{n_{i_3(\tsi)} - n_{i_2(\tsi)}}{|G_{\tsi}|}\su_1 + \frac{m_{i_2(\tsi)} - m_{i_3(\tsi)}}{|G_{\tsi}|} \su_2 + \frac{m_{i_2(\tsi)}n_{i_3(\tsi)} - m_{i_3(\tsi)}n_{i_2(\tsi)}}{|G_{\tsi}|} \su_4,\\
    \tbw(\delta_4(\tsi), \tsi) &=  \frac{n_{i_2(\tsi)} - n_{i_3(\tsi)}}{|G_{\tsi}|}\su_1 + \frac{m_{i_3(\tsi)} - m_{i_2(\tsi)}}{|G_{\tsi}|} \su_2 + \frac{|G_{\tsi}| + m_{i_3(\tsi)}n_{i_2(\tsi)} - m_{i_2(\tsi)}n_{i_3(\tsi)}}{|G_{\tsi}|} \su_4,\\
    \tbw(\delta_2(\tsi), \tsi) &=  -\frac{\fb}{|G_{\tsi}|}\su_1 + \frac{\fa}{|G_{\tsi}|} \su_2 + \frac{-\fb m_{i_3(\tsi)} + \fa n_{i_3(\tsi)}}{|G_{\tsi}|} \su_4,\\
    \tbw(\delta_3(\tsi), \tsi) &= \frac{\fb}{|G_{\tsi}|}\su_1 - \frac{\fa}{|G_{\tsi}|} \su_2 + \frac{\fb m_{i_2(\tsi)} - \fa n_{i_2(\tsi)}}{|G_{\tsi}|} \su_4.
\end{align*}
Note that $\tbw(\iota(\delta_0(\tsi)), \tsi) \big|_{\su_4 = 0}$ and $\tbw(\delta_4(\tsi), \tsi) \big|_{\su_4 = 0}$ are nonzero and independent of $f$, and 
$$
    \tbw(\delta_2(\tsi), \tsi) \big|_{\su_4 = 0} = \frac{\fa}{|G_{\tsi}|}(\su_2-f\su_1), \quad \tbw(\delta_3(\tsi), \tsi) \big|_{\su_4 = 0} = -\frac{\fa}{|G_{\tsi}|}(\su_2-f\su_1).
$$

We now consider maps from an irreducible twisted curve to a $\tT'$-fixed point in $\tcX$ corresponding to a cone in $\tSi(4) \setminus \iota(\Sigma(3))$. In particular, we study the powers of $\su_4$ and $\su_2 - f\su_1$ in the contributions of such maps to the localization computations of the closed invariants (Proposition \ref{prop:ClosedLocalResultOuter}).

\begin{lemma}[Flags] \label{lem:ExtraFlagOuter}
Let $\tsi \in \tSi(4) \setminus \iota(\Sigma(3))$ and $k \in G_{\tsi}$. Let
$$
    \tbh := e_{\tT'}\left((T_{\fp_{\tsi}}\tcX)^{k}\right).
$$
Then $\su_4$ is not a factor of $\tbh$. Moreover, for generic $f \in \bQ$, if $k = 1$, then the power of $\su_2 - f\su_1$ in $\tbh \big|_{\su_4 = 0}$ is $2$; otherwise, $\su_2 - f\su_1$ is not a factor of $\tbh \big|_{\su_4 = 0}$.
\end{lemma}

\begin{proof}
We have
$$
    \tbh = \prod_{(\ttau, \tsi) \in F(\tSi), k \in G_{\ttau}} \tbw(\ttau, \tsi).
$$
Since none of $\tbw(\ttau, \tsi)$ is a rational multiple of $\su_4$, $\su_4$ is not a factor of $\tbh$. Now for generic $f \in \bQ$, neither of $\tbw(\iota(\delta_0(\tsi)), \tsi) \big|_{\su_4 = 0}$, $\tbw(\delta_4(\tsi), \tsi) \big|_{\su_4 = 0}$ is a rational multiple of $\su_2 - f\su_1$. If $k = 1$, then $\left(\tbw(\delta_2(\tsi), \tsi)\tbw(\delta_3(\tsi), \tsi)\right) \big|_{\su_4 = 0}$, which is a rational multiple of $(\su_2 - f\su_1)^2$, is a factor of $\tbh \big|_{\su_4 = 0}$. Otherwise, $k \neq 1$ is not contained in the trivial group $G_{\delta_2(\tsi)} = G_{\delta_3(\tsi)} = \{1\}$, which implies that $\su_2 - f\su_1$ is not a factor of $\tbh \big|_{\su_4=0}$.
\end{proof}

\begin{lemma}[Stable vertices]\label{lem:ExtraVertexOuter}
Let $\tsi \in \tSi(4) \setminus \iota(\Sigma(3))$, $n \in \bZ_{\ge 3}$, $\vk = (k_1, \dots, k_n) \in G_{\tsi}^n$ such that $k_1 \cdots k_n = 1$, and
$$
    u: (\cC, \fr_1, \dots, \fr_n) \to \fp_{\tsi} \subset \tcX
$$
be a morphism that represents a point in $\Mbar_{0, \vk}(\cB G_{\tsi})$. Let
$$
    \tbh := \frac{e_{\tT'}(H^1(\cC,u^*T\tcX)^m)}{e_{\tT'}(H^0(\cC,u^*T\tcX)^m)}.
$$
Then $\su_4$ is not a factor of the denominator of $\tbh$, i.e. $\tbh \big|_{\su_4=0}$ is defined. Moreover, for generic $f \in \bQ$, if $k_1 = \cdots = k_n = 1$, then the power of $\su_2 - f\su_1$ in $\tbh \big|_{\su_4 = 0}$ is at least $-2$; otherwise, the power of $\su_2 - f\su_1$ in $\tbh \big|_{\su_4 = 0}$ is at least $0$, i.e. $\su_2 - f\su_1$ is not a factor of the denominator of $\tbh \big|_{\su_4 = 0}$.
\end{lemma}

\begin{proof}
Similar to the proof of Lemma \ref{lem:VertexComparison}, we use \cite[Lemma 126]{Liu13} to compute $\tbh$ as
$$
    \tbh = \frac{\prod_{(\ttau, \tsi) \in F(\tSi)} \Lambda_{\chi_{(\ttau, \tsi)}}^\vee(\tbw(\ttau, \tsi))}{\prod_{(\ttau, \tsi) \in F(\tSi), \hG \subseteq G_{\ttau}}\tbw(\ttau, \tsi)}
$$
where $\hG \subseteq G_{\tsi}$ denotes the subgroup generated by the monodromies of the $G_{\tsi}$-cover $\tcC \to \cC$ pulled back from $\pt \to \cB G_{\tsi}$ under $u$. We focus on the denominator of $\tbh$. Observe first that none of $\tbw(\ttau, \tsi)$ is a rational multiple of $\su_4$, which implies that $\su_4$ is not a factor of the denominator of $\tbh$. Now for generic $f \in \bQ$, neither of $\tbw(\iota(\delta_0(\tsi)), \tsi) \big|_{\su_4 = 0}$, $\tbw(\delta_4(\tsi), \tsi) \big|_{\su_4 = 0}$ is a rational multiple of $\su_2 - f\su_1$. If $k_1 = \cdots = k_n = 1$, the group $\hG$ is trivial, and thus $\left(\tbw(\delta_2(\tsi), \tsi)\tbw(\delta_3(\tsi), \tsi)\right) \big|_{\su_4 = 0}$, which is a rational multiple of $(\su_2 - f\su_1)^2$, is a factor of the denominator of $\tbh$ after the restriction $\su_4 = 0$. Otherwise, $\hG$ is non-trivial and is not contained in the trivial group $G_{\delta_2(\tsi)} = G_{\delta_3(\tsi)} = \{1\}$, which implies that $\su_2 - f\su_1$ is not a factor of the denominator of $\tbh$ after the restriction $\su_4 = 0$.
\end{proof}

\begin{lemma}[Vertices labeled by $\tsi_0$ with all non-trivial twistings]\label{lem:VertexTsi0Outer}
Let $n \in \bZ_{\ge 3}$, $\vk = (k_1, \dots, k_n) \in G_{\tsi_0}^n$ such that $k_1 \cdots k_n = 1$ and $k_i \neq 1$ for all $i$, and
$$
    u: (\cC, \fr_1, \dots, \fr_n) \to \fp_{\tsi_0} \subset \tcX
$$
be a morphism that represents a point in $\Mbar_{0, \vk}(\cB G_{\tsi_0})$. Let
$$
    \tbh := \frac{e_{\tT'}(H^1(\cC,u^*T\tcX)^m)}{e_{\tT'}(H^0(\cC,u^*T\tcX)^m)}.
$$
Then for generic $f \in \bQ$, the power of $\su_2 - f\su_1$ in $\tbh \big|_{\su_4 = 0}$ is $n-2 \ge 1$.
\end{lemma}

\begin{proof}
Similar to the proof of Lemma \ref{lem:ExtraVertexOuter}, we use \cite[Lemma 126]{Liu13} and that $k_i \neq 1$ for all $i$ to compute $\tbh$ as 
$$
    \tbh = \begin{cases}
        \displaystyle{ \frac{\prod_{(\ttau, \tsi_0) \in F(\tSi)} \Lambda_{\chi_{(\ttau, \tsi_0)}}^\vee(\tbw(\ttau, \tsi_0))}{\tbw(\iota(\tau_0), \tsi_0)\tbw(\delta_4(\tsi_0), \tsi_0)} = \frac{\prod_{(\ttau, \tsi_0) \in F(\tSi)} \Lambda_{\chi_{(\ttau, \tsi_0)}}^\vee(\tbw(\ttau, \tsi_0))}{(-\frac{1}{\fa}\su_1)(\frac{1}{\fa}\su_1 + \su_4)} } & \text{if } \age(k_i) = 1 \text{ for all } i,\\
        \prod_{(\ttau, \tsi_0) \in F(\tSi)} \Lambda_{\chi_{(\ttau, \tsi_0)}}^\vee(\tbw(\ttau, \tsi_0)) & \text{otherwise.}\\
    \end{cases}
$$
We focus on the term $\prod_{(\ttau, \tsi_0) \in F(\tSi)} \Lambda_{\chi_{(\ttau, \tsi_0)}}^\vee(\tbw(\ttau, \tsi_0))$. Observe first that $\fr(\iota(\tau_0), \tsi_0) = \fr(\delta_4(\tsi_0), \tsi_0) = 1$, which implies that $\chi_{(\iota(\tau_0), \tsi_0)}$ and $\chi_{(\delta_4(\tsi_0), \tsi_0)}$ are trivial. Then,
$$
    \Lambda_{\chi_{(\iota(\tau_0), \tsi_0)}}^\vee(\tbw(\iota(\tau_0), \tsi_0)) = \Lambda_{\chi_{(\delta_4(\tsi_0), \tsi_0)}}^\vee(\tbw(\delta_4(\tsi_0), \tsi_0)) = 1.
$$
Moreover, the product of the isomorphisms
$$
    \chi_{(\delta_2(\tsi_0), \tsi_0)}, \chi_{(\delta_3(\tsi_0), \tsi_0)}: G_{\tsi_0} \to \mu_{\fm} \subset \bC^*
$$
is trivial. Lemma \ref{lem:Mumford} then implies that
\begin{align*}
    & \left( \Lambda_{\chi_{(\delta_2(\tsi_0), \tsi_0)}}^\vee(\tbw(\delta_2(\tsi_0), \tsi_0)) \Lambda_{\chi_{(\delta_3(\tsi_0), \tsi_0)}}^\vee(\tbw(\delta_3(\tsi_0), \tsi_0)) \right) \big|_{\su_4 = 0} \\
        &= \Lambda_{\chi_{(\delta_2(\tsi_0), \tsi_0)}}^\vee(\tfrac{\su_2 - f\su_1}{\fm}) \Lambda_{\chi_{(\delta_3(\tsi_0), \tsi_0)}}^\vee(-\tfrac{\su_2 - f\su_1}{\fm})\\
        &= \pm \left(\frac{\su_2 - f\su_1}{\fm}\right)^{\rank(\bE_{\chi_{(\delta_2(\tsi_0), \tsi_0)}}) + \rank(\bE_{\chi_{(\delta_3(\tsi_0), \tsi_0)}})}.
\end{align*}
Finally, since $k_i \neq 1$ for each $i = 1, \dots, n$, we have by \eqref{eqn:HHBundleRank} that
$$
    \rank(\bE_{\chi_{(\delta_2(\tsi_0), \tsi_0)}}) + \rank(\bE_{\chi_{(\delta_3(\tsi_0), \tsi_0)}}) = n-2.
$$
\end{proof}

\begin{lemma}[Vertex integrals] \label{lem:ExtraIntegralOuter}
Let $\tsi \in \tSi(4) \setminus \iota(\Sigma(3))$, $n \in \bZ_{\ge 1}$, $\vk = (k_1, \dots, k_n) \in G_{\tsi}^n$ such that $k_1 \cdots k_n = 1$. Let $0 \le n' \le n$, and $\ttau_1, \dots, \ttau_{n'} \in \tSi(3)_c$ such that each $\ttau_i$ is a facet of $\tsi$ and $k_i \in \pi_{(\ttau_i, \tsi)}(G_{\ttau_i})$. In addition, for each $i = 1, \dots, n'$, let $r_i$ be the order of $k_i$ in $G_{\tsi}$, and take $d_i \in \bZ_{>0}$. Let
$$
    \tbh := \int_{\Mbar_{0, \vk}(\cB G_{\tsi})} \frac{1}{\prod_{i = 1}^{n'} \left(\frac{\fr(\ttau_i, \tsi)\tbw(\ttau_i, \tsi)}{r_id_i} - \frac{\bar{\psi}_i}{r_i}\right)}.
$$
If $n = n' = 2$ and $\{\ttau_1, \ttau_2\} = \{\delta_2(\tsi), \delta_3(\tsi)\}$, then the power of $\su_4$ in $\tbh$ is $-1$; otherwise, $\su_4$ is not a factor of $\tbh$. Moreover, if $\ttau_i \in \iota(\Sigma(2))$ for all $i$ and $f \in \bQ$ is generic, then $\su_2 - f\su_1$ is not a factor of $\tbh \big|_{\su_4 = 0}$.
\end{lemma}

\begin{proof}
Note by \eqref{eqn:TanWtU4RestrictOuter} that $\tbw(\iota(\tau_i),\iota(\sigma)) \not \in \bQ \su_4$. In view of Lemma \ref{lem:DescendantIntegral}, the only case where $\su_4$ appears as a factor in the denominator of $\tbh$ is when $n = n' = 2$ and
$$
    r_1\tbw(\iota(\tau_1),\iota(\sigma)) + r_2\tbw(\iota(\tau_2),\iota(\sigma)) \in \bQ \su_4,
$$
which holds only if $\{\ttau_1, \ttau_2\} = \{\delta_2(\tsi), \delta_3(\tsi)\}$.

Now suppose $\ttau_i \in \iota(\Sigma(2))$ for all $i$. Then the data defining $\tbh$ is independent of $f$, which implies the final claim of the lemma.
\end{proof}

\subsection{Proper \texorpdfstring{$\tT'$}{T}-invariant lines corresponding to cones in \texorpdfstring{$\tSi(3)_c \setminus \iota(\Sigma(2)_c)$}{Tau}}\label{sect:EdgesExtraConesOuter}
In this section, we consider maps from an irreducible twisted curve to a proper $\tT'$-invariant line in $\tcX$ corresponding to a cone $\ttau \in \tSi(3)_c \setminus \iota(\Sigma(2)_c)$ and their contributions to the localization computations of the closed invariants (Proposition \ref{prop:ClosedLocalResultOuter}). For the distinguished cone $\ttau = \iota(\tau_0)$, we explicitly compute the contribution and compare it to the disk factor (Section \ref{sect:DiskFactorOuter}). For a general $\ttau$, we study the power of $\su_4$ in the contribution.

\begin{lemma}[Edges labeled by $\iota(\tau_0)$] \label{lem:EdgeIotaTau0Outer}
Let
$$
    u: \cC = \cC_{r_1, r_2} \to \fl_{\iota(\tau_0)}\subset \tcX
$$
be a morphism as in \eqref{eqn:TwistedCover} in Section \ref{sect:TwistedCovers} with degree $(d,\lambda) \in H_{\iota(\tau_0)} \cong \bZ \times G_{\tau_0}$. Let
$$ 
    \tbh:= \frac{e_{\tT'}(H^1(\cC,u^*T\tcX)^m)}{e_{\tT'}(H^0(\cC,u^*T\tcX)^m)}.
$$
Then $\su_4$ is not a factor of $\tbh$. Moreover, let
$$
    h(d, \lambda)= \pi_{(\tau_0, \sigma_0)}(d, \lambda) \in G_{\sigma_0}, \quad \tk= \pi_{(\iota(\tau_0), \tsi_0)}(d, \lambda) \in G_{\tsi_0}.
$$
If $\tk=1$, we have
$$
   \tbh\big|_{\su_4 = 0} 
        = \frac{(-1)^{\floor{dw_3-\epsilon_3}+ \frac{d}{\fa} + 1}}{\floor{dw_0}!} \left(\frac{\su_1}{d} \right)^{\age(h(d,\lambda)) - 1}\left(\frac{\su_1}{\fa} \right)^{-1} \left( \frac{\su_2-f\su_1}{\fm} \right)^{-1} \cdot \prod_{a = 1}^{\floor{dw_0}+\age(h(d,\lambda))-1} \left(\frac{d\bw_2}{\su_1} + a - \epsilon_2 \right).
$$
If $\age(\tk)=1$, we have
$$
  \tbh \big|_{\su_4 = 0} 
        = \frac{(-1)^{\floor{dw_3-\epsilon_3}+ \frac{d}{\fa} + 1}}{\floor{dw_0}!} \left(\frac{\su_1}{d} \right)^{\age(h(d,\lambda)) - 1} \left(\frac{\su_1}{\fa} \right)^{-1} \cdot \prod_{a = 1}^{\floor{dw_0}+\age(h(d,\lambda))-1} \left(\frac{d\bw_2}{\su_1} + a - \epsilon_2 \right).
$$
If $\age(\tk)=2$, we have
$$
 \tbh \big|_{\su_4 = 0} 
        = \frac{(-1)^{\floor{dw_3-\epsilon_3}+ \floor{\frac{d}{\fa}} + 1}}{\floor{dw_0}!} \left(\frac{\su_1}{d} \right)^{\age(h(d,\lambda)) - 1} \cdot \prod_{a = 1}^{\floor{dw_0}+\age(h(d,\lambda))-1} \left(\frac{d\bw_2}{\su_1} + a - \epsilon_2 \right).
$$
Here, the quantities $\sw_2, w_0, w_2, w_3, \epsilon_2, \epsilon_3$ are defined in Section \ref{sect:DiskLocalOuter}. 
\end{lemma}

\begin{proof}
We use the explicit computation of $\tbh$ in \cite[Lemma 130]{Liu13}. We set
$$
    \bb_1 = \frac{1}{e_{\tT'}(H^0(\cC,u^*T\fl_{\iota(\tau_0)})^m)}.
$$
Moreover, the normal bundle $N_{\fl_{\iota(\tau_0)}/\tcX}$ splits as a direct sum of $\tT'$-equivariant line bundles $L_2, L_3, L_4$ given by the normal bundles of $\fl_{\iota(\tau_0)}$ in the $2$-dimensional $\tT$-invariant closed substacks of $\tcX$ corresponding to the cones spanned by
$$
    \{\trho_{3}, \trho_{R+2}\}, \quad \{\trho_{2}, \trho_{R+2}\}, \quad \{\trho_{2}, \trho_{3}\}
$$
respectively. We set
$$
    \bb_i := \frac{e_{\tT'}(H^1(\cC,u^*L_i)^m)}{e_{\tT'}(H^0(\cC,u^*L_i)^m)}, \quad i = 2,3,4.
$$
Then, we have
$$
    \tbh =  \bb_1 \bb_2 \bb_3 \bb_4,
$$
where the $\bb_i$'s are computed in \cite[Lemma 130]{Liu13} as
$$
    \bb_1 = \frac{(-1)^{\floor{\frac{d}{\fa}}}}{\floor{dw_0}!\floor{\frac{d}{\fa}}!} \left(\frac{\su_1}{d}\right)^{-\floor{dw_0}-\floor{\frac{d}{\fa}}}, \quad 
    \bb_4 = \prod_{a=1}^{\ceil{\frac{d}{\fa}-1}} \left(\su_4 + \frac{a\su_1}{d} \right),
$$
$$
    \bb_2 = \begin{cases}
        \prod_{a=0}^{\floor{dw_2-\epsilon_2}} \left(\tbw_2- \frac{a+\epsilon_2}{d}\su_1 \right)^{-1} & \text{if } w_2 \ge 0\\
        \prod_{a=1}^{\floor{\epsilon_2-dw_2-1}} \left(\tbw_2+ \frac{a-\epsilon_2}{d}\su_1 \right) & \text{if } w_2 < 0, 
    \end{cases}
    \quad 
    \bb_3 = \begin{cases}
        \prod_{a=0}^{\floor{dw_3-\epsilon_3}} \left(\tbw_3- \frac{a+\epsilon_3}{d}\su_1 \right)^{-1} & \text{if } w_3 \ge 0\\
        \prod_{a=1}^{\floor{\epsilon_3-dw_3-1}} \left(\tbw_3+ \frac{a-\epsilon_3}{d}\su_1 \right) & \text{if } w_3 < 0.
    \end{cases}
$$
It follows that $\su_4$ is not a factor of $\tbh$. We can compute that 
$$
    \bb_4 \big|_{\su_4 = 0} = \ceil{\frac{d}{\fa}-1}! \left(\frac{\su_1}{d}\right)^{\ceil{\frac{d}{\fa}-1}}
$$
and
$$
    (\bb_1\bb_4) \big|_{\su_4 = 0} = \begin{cases}
        \displaystyle{ \frac{(-1)^{\frac{d}{\fa}}}{\floor{dw_0}!}\left(\frac{\su_1}{d}\right)^{-\floor{dw_0}} \left(\frac{\su_1}{\fa}\right)^{-1} } & \text{if } \fa \mid d,\\
        \displaystyle{ \frac{(-1)^{\floor{\frac{d}{\fa}}}}{\floor{dw_0}!}\left(\frac{\su_1}{d}\right)^{-\floor{dw_0}} } & \text{if } \fa \nmid d.\\
    \end{cases}
$$
Recall that $\fa \mid d$ if and only if $\age(\tk) \le 1$. Moreover, if $\tk = 1$, then
$$
    \left(\bb_2 \bb_3\right) \bigg|_{\su_4 = 0} = (-1)^{\floor{dw_3-\epsilon_3}+1} \left(\frac{\su_1}{d} \right)^{\floor{dw_0} + \age(h(d,\lambda)) -1} \left( \frac{\su_2-f\su_1}{\fm} \right)^{-1} \cdot \prod_{a = 1}^{\floor{dw_0}+\age(h(d,\lambda))-1} \left(\frac{d\bw_2}{\su_1} + a - \epsilon_2 \right).
$$
On the other hand, if $\tk \neq 1$, then
$$
    \left(\bb_2 \bb_3\right) \big|_{\su_4 = 0} = (-1)^{\floor{dw_3-\epsilon_3}+1} \left(\frac{\su_1}{d} \right)^{\floor{dw_0} + \age(h(d,\lambda)) -1} \cdot \prod_{a = 1}^{\floor{dw_0}+\age(h(d,\lambda))-1} \left(\frac{d\bw_2}{\su_1} + a - \epsilon_2 \right).
$$
The lemma thus follows.
\end{proof}

We now consider a more general cone $\ttau \in \iota(\Sigma(2) \setminus \Sigma(2)_c)$. Here, $\ttau = \iota(\delta_0(\tsi))$ for some $\tsi \in \tSi(4) \setminus \iota(\Sigma(3))$ and $I_{\ttau}' = \{i_2(\tsi), i_3(\tsi), R+2\}$. We have
$$
    G_{\ttau} \cong \mu_{\gcd(|m_{i_2(\tsi)} - m_{i_3(\tsi)}|, |n_{i_2(\tsi)} - n_{i_3(\tsi)}|)} \subseteq G_{\tsi} \cong \mu_{|G_{\tsi}|}.
$$
Given $\gamma \in G_{\ttau} \subseteq G_{\tsi}$, we have $\chi_{(\ttau, \tsi)}(\gamma) = \chi_{(\delta_4(\tsi), \tsi)}(\gamma)=1$, and since $\chi_{(\delta_2(\tsi), \tsi)}, \chi_{(\delta_3(\tsi), \tsi)}: G_{\tsi} \to \mu_{|G_{\tsi}|}$ are isomorphisms, $\chi_{(\delta_2(\tsi), \tsi)}(\gamma) = \chi_{(\delta_3(\tsi), \tsi)}(\gamma)=1$ if and only if $\gamma = 1$.

\begin{lemma}[Edges with label in $\iota(\Sigma(2) \setminus \Sigma(2)_c)$] \label{lem:ExtraEdge1Outer}
Let $\ttau \in \tSi(3)_c \cap \iota(\Sigma(2) \setminus \Sigma(2)_c)$, and
$$
    u: \cC = \cC_{r_1, r_2} \to \fl_{\ttau}\subset \tcX
$$
be a morphism as in \eqref{eqn:TwistedCover} in Section \ref{sect:TwistedCovers} with degree $(d,\lambda) \in H_{\ttau} \cong \bZ \times G_{\ttau}$. Let
$$ 
    \tbh:= \frac{e_{\tT'}(H^1(\cC,u^*T\tcX)^m)}{e_{\tT'}(H^0(\cC,u^*T\tcX)^m)}.
$$
Then $\su_4$ is not a factor of $\tbh$.
\end{lemma}

\begin{proof}
Similar to the proof of Lemma \ref{lem:EdgeIotaTau0Outer}, we use the explicit computation of $\tbh$ in \cite[Lemma 130]{Liu13}. We set
$$
    \bb_1 = \frac{1}{e_{\tT'}(H^0(\cC,u^*T\fl_{\ttau})^m)}.
$$
Moreover, let $\tsi \in \tSi(4) \setminus \iota(\Sigma(3))$ such that $\ttau = \iota(\delta_0(\tsi))$. The normal bundle $N_{\fl_{\ttau}/\tcX}$ splits as a direct sum of $\tT'$-equivariant line bundles $L_2, L_3, L_4$, which are the normal bundles of $\fl_{\ttau}$ in the $2$-dimensional $\tT$-invariant closed substacks of $\tcX$ corresponding to the cones spanned by
$$
    \{\trho_{i_3(\tsi)}, \trho_{R+2}\}, \quad \{\trho_{i_2(\tsi)}, \trho_{R+2}\}, \quad \{\trho_{i_2(\tsi)}, \trho_{i_3(\tsi)}\}
$$
respectively. We set
$$
    \bb_i := \frac{e_{\tT'}(H^1(\cC,u^*L_i)^m)}{e_{\tT'}(H^0(\cC,u^*L_i)^m)}, \quad i = 2,3,4.
$$
Then, we have
$$
    \tbh =  \bb_1 \bb_2 \bb_3 \bb_4.
$$
By the computations of the $\bb_i$'s in \cite[Lemma 130]{Liu13}, $\su_4$ is not a factor of any of them, which implies the lemma. 
\end{proof}

Lastly, we consider a cone $\ttau \in \tSi(3)_c \setminus \iota(\Sigma(2))$. Here, $\ttau$ is a common facet of two distinct $4$-cones in $\tSi(4) \setminus \iota(\Sigma(3))$, and $G_{\ttau} = \{1\}$, which implies that $H_{\ttau} \cong \bZ$.

\begin{lemma}[Edges with label in $\tSi(3)_c \setminus \iota(\Sigma(2))$] \label{lem:ExtraEdge2Outer}
Let $\ttau \in \tSi(3)_c \setminus \iota(\Sigma(2))$, and
$$
    u: \cC = \cC_{r_1, r_2} \to \fl_{\ttau}\subset \tcX
$$
be a morphism as in \eqref{eqn:TwistedCover} in Section \ref{sect:TwistedCovers} with degree $d \in H_{\ttau} \cong \bZ$. Let
$$ 
    \tbh:= \frac{e_{\tT'}(H^1(\cC,u^*T\tcX)^m)}{e_{\tT'}(H^0(\cC,u^*T\tcX)^m)}.
$$
Then the power of $\su_4$ in $\tbh$ is $1$.
\end{lemma}

\begin{proof}
We again use the explicit computation of $\tbh$ in \cite[Lemma 130]{Liu13}. We set
$$
    \bb_2 = \frac{1}{e_{\tT'}(H^0(\cC,u^*T\fl_{\ttau})^m)}.
$$
Moreover, let $I_{\ttau}' = \{i, R+1, R+2\}$. The normal bundle $N_{\fl_{\ttau}/\tcX}$ splits as a direct sum of $\tT'$-equivariant line bundles $L_1, L_3, L_4$, which are the normal bundles of $\fl_{\ttau}$ in the $2$-dimensional $\tT$-invariant closed substacks of $\tcX$ corresponding to the cones spanned by
$$
    \{\trho_{i}, \trho_{R+2}\}, \quad \{\trho_{R+1}, \trho_{R+2}\}, \quad \{\trho_{i}, \trho_{R+1}\}
$$
respectively. We set
$$
    \bb_i := \frac{e_{\tT'}(H^1(\cC,u^*L_i)^m)}{e_{\tT'}(H^0(\cC,u^*L_i)^m)}, \quad i = 1,3,4.
$$
Then, we have
$$
    \tbh =  \bb_2 \bb_1 \bb_3 \bb_4.
$$
By the computations of the $\bb_i$'s in \cite[Lemma 130]{Liu13}, $\su_4$ is not a factor of any of $\bb_2, \bb_1, \bb_4$, but is a factor of $\bb_3$ with power $1$. The lemma thus follows.
\end{proof}


\end{document}